%% file: macro.tex
\documentclass[letterpaper,11pt,reqno]{amsart} 
\usepackage[margin=1in]{geometry}
\usepackage[T1]{fontenc}
\input{SL_style}

\begin{document}
\title[FPP on spread-out cycle graphs]{Phase Transition and Fluctuation Results for First-Passage Percolation on Spread-Out Cycle Graphs}
\author[Dey]{Partha S.~Dey$^\star$}
\author[Kim]{Daecheol Kim$^\dagger$}
\address{Department of Mathematics, University of Illinois Urbana-Champaign, Urbana, Illinois 61801}
\email{\{$^\star$psdey,$^\dagger$dk43 \}@illinois.edu}
%\date{\today}
%%%%%%%%%%%%%%%%%%%%%%%%%%%%%%%%%%%%%%%%%%%%%%%%%%%
\begin{abstract}
We study first-passage percolation on the $\ell$-spread-out one-dimensional cycle of size $n$, where vertices are connected if their graph distance is at most $\ell$. We assign i.i.d.~non-negative random weights from a Weibull distribution $\go_e \sim \Exp(1)^{1/\theta}$ to the edges for $\theta>0$ fixed. This paper investigates the transition in the asymptotic behavior of the passage time $T_n$ between two typical vertices and the hop-count of the optimal path as the connectivity parameter $\ell$ diverges with $n$. We identify two fundamentally distinct geometric regimes. In the \emph{mesoscopic} regime ($1 \ll \ell \ll n$), the optimal path locally mimics a spatial branching random walk but remains globally constrained to a one-dimensional geometry. We establish a law of large numbers characterized by the front speed of a Crump--Mode--Jagers branching random walk, prove a central limit theorem with Gaussian fluctuations when  $\ell\ll n^{1/4}$, and show that the expected hop-count grows proportionally with the spatial distance. In the \emph{macroscopic} regime ($\ell \approx \gl n$ for $\gl \in (0,1/2)$), the graph becomes a highly connected mean-field network. We prove that the passage time collapses to a $\log n$ scale with constant order non-Gaussian fluctuations, explicitly determining the extreme-value limit driven by the collision of two independent non-spatial CMJ processes. We establish a law of large numbers for the hop-count. Finally, we rigorously trace the transition in the order of the mean of $T_n$ between these two regimes, demonstrating an order transition for the passage time across the critical connectivity threshold $\ell \asymp n/\log n$. Our results provide a comprehensive deterministic-range interpolation from spatial Gaussian fluctuations to mean-field extreme-value fluctuations.
\end{abstract}
%%%%%%%%%%%%%%%%%%%%%%%%%%%%%%%%%%%%%%%%%%%%%%%%%%%
\maketitle
\setcounter{tocdepth}{1}\tableofcontents

%%%%%%%%%%%%%%%%%%%%%%%%%%%%%%%%%%%%%%%%%%%%%%%%%%%
\section{Introduction}\label{sec:intro}
\subsection{The Model}\label{ssec:model}
Consider a locally finite base graph $G$. Given an integer $\ell \ge 2$, we construct the $\ell$-spread-out graph $G^{(\ell)}$ by connecting every pair of vertices whose graph distance is at most $\ell$. In this article, we study first-passage percolation on $G^{(\ell)}$ with independent and identically distributed non-negative edge weights. Let 
\[
T^{(\ell)}_G(u,v):=\text{first-passage time from vertex $u$ to vertex $v$ in } G^{(\ell)},
\]
representing the minimum time required for information to propagate across the network from vertex $u$ to vertex $v$.

One motivation comes from propagation on a periodic network with a growing interaction range. Vertices represent individuals or routers, edges correspond to direct communication links, and weights represent a transmission time or latency. Thus, $T_G^{(\ell)}(u,v)$ is the minimum travel time for an infection or message from $u$ to $v$; such FPP interpretations naturally arise in epidemic and routing models~\cite{bhk14epidemic,baccelli09sinr,ln18}. When $G$ is a cycle, the model also serves as a deterministic counterpart to ring-based small-world networks. While classical small-world models introduce random long-range connections via rewiring or sparse shortcuts~\cite{wattsstrogatz98,newmanwatts99,kv16}, our model connects every pair within distance $\ell$ deterministically.

The parameter $\ell$ controls the interaction range while the edge-weight law remains fixed. When $\ell$ is fixed, the graph geometry dominates at large scales but there are several local bypasses. When $\ell$ grows with $n$, the number of available local routes increases. Finally, the model moves toward a highly connected, mean-field-like regime, when $\ell$ is comparable to the graph diameter. Our goal is to determine how the asymptotic mean, fluctuation, and distributional properties of $T^{(\ell)}(u,v)$ evolve when the graph distance between $u$ and $v$ increases to infinity and the connectivity parameter $\ell$ varies. We ask three interrelated questions, namely, a)~Propagation Speed, b)~Distributional Phase Transitions, and c)~Optimal Path Geometry.

The behavior of the first-passage time across the domain of size $n$ depends intrinsically on the relative growth of the connectivity range $\ell$ as $n \to \infty$. This model naturally partitions into three distinct regimes:
\begin{enumerate}[i)]
 \item \textbf{Microscopic Regime} ($\ell\asymp 1$):~The connectivity is local and bounded.
 \item \textbf{Mesoscopic Regime} ($1 \ll \ell \ll n$):~The connectivity diverges, but stays sublinear in the system size.
 \item \textbf{Macroscopic Regime} ($\ell \approx \gl n$, for some $0<\gl\le 1/2$):~The connectivity range is a strictly positive fraction of the entire network.
\end{enumerate}
In this article, we mainly focus on the case when the base graph $G$ is the one-dimensional cycle graph $\dT_{n}:=\dZ/n\dZ$ on vertex set $\cV_{n}:=\{0,1,\ldots,n-1\}$ and $\ell=\ell_{n}$ increases to $\infty$ as the graph size increases to $\infty$. We write
\begin{align}\label{eq:cycle-distance}
 d_n(x,y):=\min\bigl\{\abs{x-y},n-\abs{y-x}\bigr\}
\end{align}
for the graph distance on $\dT_n$, and we identify $\cV_n$ with $\{0,1,\ldots,n-1\}$ throughout; in particular the vertex $n$ is the vertex $0$. On the $n$-cycle the graph distance never exceeds $\lfloor n/2\rfloor$, so $\gl\in(0,1/2]$ exhausts the macroscopic range, and $\gl=1/2$ gives the complete graph.

Formally, given a positive integer $\ell\le n/2$, the $\ell$-spread-out version (also known as the $\ell$-th power) of $\dT_n$ is defined by $\dT_n^{(\ell)} = (\cV_n, \cE_n^{(\ell)})$, where $\cE_n^{(\ell)}$ is the set of edges
\begin{align}\label{eq:cycle-edge-set}
	\cE_n^{(\ell)} := \bigl\{ \{x,y\}\subseteq\cV_n : 1\le d_n(x,y) \leq \ell \bigr\}.
\end{align}
Let $\go_{e}=\go_{x,y}=\go_{y,x}$ denote non-negative, i.i.d.~random weights assigned to every edge $e=\{x, y\} \in \cE_n^{(\ell)} $ with common distribution function $F$.

\begin{ass}\label{ass:weight}
 We assume that the edge weights $(\go_e,e\in\cE_n^{(\ell)})$ follow a Weibull distribution, \ie
\begin{align}\label{eq:edge-law}
 \go_e\sim \Exp(1)^{1/\theta} \text{ for some fixed } \theta>0.
\end{align}
We denote for $t\ge 0$,
\begin{alignat*}{2}
 F(t) &:=\pr(\go_e\le t)=1-e^{-t^\theta},
 \qquad 
 &\bar F(t)&:=1-F(t)=e^{-t^\theta},
 \\
 H(t)&:=-\log\bar F(t)=t^\theta,
 \qquad
 &h(t)&:=H'(t)=\theta t^{\theta-1},
\end{alignat*}
the distribution function, survival function, cumulative hazard function and hazard density, respectively.
Thus $H(\go_e)\sim\Exp(1)$.
\end{ass} 

For vertices $u,v$ in the graph $\dT_{n}^{(\ell)}$, we write
\begin{align*}
 T_{n}^{(\ell)}(u,v)
 := T^{(\ell)}_{\dT_n}(u,v) = 
 \inf_{\pi:u\to v}\sum_{e\in\pi} \go_e
\end{align*}
for the first-passage time restricted to paths entirely within $\dT_{n}^{(\ell)}$. In the companion paper~\cite{DK26a}, we fully characterized the \textit{microscopic regime} for general $F$. By relying on the spatial one-dimensional structure, we identified random ``pivot nodes'' that decompose the optimal path into exactly independent and identically distributed renewal blocks. This decomposition led to classical Gaussian central limit theorems or non-Gaussian stable limit laws, depending on the tail of the edge weights.

However, as $\ell \to \infty$, the regeneration framework completely breaks down. The increasing number of available paths allows geodesics to bypass local bottlenecks with high probability. As a result, ordinary pivot nodes disappear, and the regenerative approach used in the microscopic regime no longer applies. In this paper, we turn our focus to the \emph{mesoscopic} and \emph{macroscopic} regimes to characterize the precise fluctuation limits and the distributional phase transitions as the model evolves from a one-dimensional local regime toward a mean-field regime. Analysis of the meso- and macroscopic regimes in this article is self-contained and can be read independently of the companion paper~\cite{DK26a}.

Fix a deterministic target location $\sfu\in(0,1/2)$. \emph{Throughout the paper, assume that $\sfu_n \in \dN$ and $\sfu_n/n \to \sfu$ as $n\to\infty$}. We are interested in the fixed-target point-to-point first-passage time on the $\ell$-spread-out cycle
\begin{align}\label{eq:fixed-target-Tn}
 T_n=T_n^{(\ell)}:=T_{\dT_n}^{(\ell)}(0,\sfu_n).
\end{align}
This deterministic-target formulation is the common setup for both the mesoscopic and macroscopic results. In the mesoscopic proofs, we will exploit the conditions $\sfu<1/2$ and $1\ll \ell\ll n$ to compare the cycle geometry near the geodesic scale with an $\ell$-spread-out line segment; these line-segment restrictions are needed only within the proof sections. In the macroscopic regime, the limiting law is independent of the particular fixed target $\sfu\in(0,1/2)$. For proof convenience, the macroscopic proof may therefore use an independent uniform target $U_n$ in place of $\sfu_n$; the corresponding random-target formulation is
\begin{align}\label{eq:unif-target-Tn}
	T_n^{\unif}
	:=
	T_n^{(\ell)}(0,U_n).
\end{align}

\subsection{Main results}\label{ssec:main-results}
The mesoscopic and macroscopic regimes exhibit different limiting mechanisms. In the mesoscopic regime, the range $\ell$ diverges but remains sublinear. The large-scale geometry is still one-dimensional, while the local optimization no longer has the exact pivot-renewal structure of the microscopic regime. Instead, local Dijkstra explorations are approximated by a spatial CMJ branching random walk, and the remaining one-dimensional structure can be recovered through a block decomposition.

Define 
\begin{alignat}{2}
	\fa_\ell &:=\bigl((2\ell-1)\cdot \Gamma(\theta+1)\bigr)^{1/\theta},\qquad
 &&\gth\ :=\max\{1,1/\theta\}, \label{eq:aell-def}\\
 \gf(\beta)&:=(\sinh\beta)/{\beta},
 \qquad
 &&\vstar:=
 \inf_{\beta>0} \gf(\beta)^{1/\theta}/\beta,\label{def:vstar}
\end{alignat}
and let
\begin{align}\label{eq:beta_c}
 \gbc:=\text{the unique minimizer in the definition of $\vstar$}.
\end{align}
The infimum in~\eqref{def:vstar} is attained at a unique $\gbc\in(0,\infty)$; this is verified in Lemma~\ref{lem:betac-unique} below. Note that $\gf(\beta)=\E e^{\beta\xi}$ for $\xi\sim\Unif[-1,1]$, and that $\vstar\in(0,\infty)$.
The constant $\vstar$ can be interpreted as the right-front speed of the limiting spatial CMJ branching random walk (see Section~\ref{sec:local-cmj} for more details). Now, we present our first main result for deterministic target. 

\begin{thm}[Mesoscopic]\label{thm:meso-main}
Assume that the edge weights satisfy Assumption~\ref{ass:weight}. Fix $\sfu\in(0,1/2)$ and let $T_n^{(\ell)}(0,\sfu_n)$ be the first-passage time defined in~\eqref{eq:fixed-target-Tn}. The following holds.
\begin{enumeratea}
 \item \label{meso-mean} When $1\ll \ell\log\ell \ll n$, we have
 \begin{align} \label{eq:meso-mean}
 \frac{\fa_\ell\ell}{\sfu_n} \cdot T_n^{(\ell)}(0,\sfu_n)
 \stackrel{\pr}{\to}
 \frac1{\vstar}
 \text{ and }
 \frac{\fa_\ell\ell}{\sfu_n} \cdot\E T_n^{(\ell)}(0,\sfu_n)
 \to \frac{1}{\vstar} \text{ as }n\to\infty,
\end{align}
where $\vstar\in (0,\infty)$ is the constant defined in~\eqref{def:vstar}.

 \item \label{meso-var} When $1\ll\ell\ll n$, there is a constant $C=C(\theta,\sfu)\in (0,\infty)$ such that, for all large $n$,
\begin{align} \label{eq:meso-var-ublb}
 C^{-1}\cdot \left( \frac n{\ell^{2+2/\theta}} \vee\frac1{\ell^{2/\theta}} \right)
 \le \var\bigl(T_n^{(\ell)}(0,\sfu_n)\bigr)
 \le C\cdot \frac{n}{\ell^{1+2/\theta}}\cdot (\log\ell)^{2\gth}.
\end{align}

 \item \label{meso-clt} When $\ell\ll n^{1/4}(\log n)^{-3\gth/2}$, we have
 \begin{align}\label{eq:meso-clt}
 \frac{T_n^{(\ell)}(0,\sfu_n)-\E T_n^{(\ell)}(0,\sfu_n)}{\sqrt{\var\bigl(T_n^{(\ell)}(0,\sfu_n)\bigr)}}
 \Rightarrow\N(0,1)\text{ as } n\to\infty.
\end{align}
\end{enumeratea}
\end{thm}

\begin{rem}[Suboptimality of the variance lower bound]
\label{rem:dijkstra-lower-bound-loss}
Up to logarithmic corrections, the true variance is expected to scale as $n\ell^{-1-2/\theta+o(1)}$.
Heuristically, a geodesic from $0$ to $n$ can be decomposed into roughly $n\ell^{-1}$ many independent local crossings, each traversing a spatial distance of scale $\ell$. Since each local crossing spans a time scale of $\ell^{-1/\theta}$, it contributes a local variance of order $\ell^{-2/\theta+o(1)}$. Summing these independent contributions yields the conjectural optimal scale $n\ell^{-1}\cdot \ell^{-2/\theta+o(1)} = n\ell^{-1-2/\theta+o(1)}$,
which matches the variance upper bound up to a polylogarithmic factor.

The global Dijkstra perturbation used in the proof of the variance lower bound combined with the result from~\cite[Lemma 1.2]{cha19} loses exactly one power of $\ell$. Because our comparison rescales the entire Dijkstra exploration simultaneously, it accumulates a total variation penalty of order $\eps\sqrt{n}$, forcing the choice $\eps\asymp n^{-1/2}$ to keep the error small. Given that the typical passage time is of order $n\ell^{-1-1/\theta}$, the deterministic shift produced by this global scaling is only
$
 \eps\cdot n\ell^{-1-1/\theta}
 \asymp
 \sqrt n\cdot \ell^{-1-1/\theta}.
$
Squaring this separation via Chatterjee's criterion yields precisely the suboptimal lower bound $n\ell^{-2-2/\theta}$.
\end{rem}
\begin{rem}[Expected CLT window]\label{rem:meso-clt-window}
We expect the Gaussian central limit theorem to hold for all $\ell\ll n/\log n$. Heuristically, as long as $n/\ell\to\infty$, the geodesic still traverses a diverging number of essentially one-dimensional correlation segments. Although each segment is locally governed by a CMJ-type wavefront, the total passage time remains a sum of many weakly dependent local contributions. Thus, the fluctuation mechanism should remain Gaussian until the number of these independent segments stops diverging, which is precisely the macroscopic regime.

In the proof of the CLT, we use a block decomposition approach with blocks of size $b$ with $\ell\ll b\ll n$. There are approximately $n/b$ many independent blocks and each interface introduces an approximation error of the order $\ell^{-1/\theta+o(1)}$, with total error bounded by $\ell^{-1/\theta+o(1)}\cdot \sqrt{n/b}$. To apply a CLT for triangular arrays, we need the error to be smaller than the standard deviation of the first-passage time and to verify Lyapunov fourth-moment condition we need $n/b\cdot \text{fourth central moment for fpp in a block} \ll \text{squared variance}$. Combining with our current variance lower bound, the first condition needs $n\ell^{-2-2/\theta}\gg \ell^{-2/\theta}\cdot n/b$ or $b\gg \ell^2$ and the second condition needs $n/b\cdot b^2\ell^{-2-4/\theta}\ll n^2\ell^{-4-4/\theta}$ or $b\ll n\ell^{-2}$. This restricts $\ell \ll n^{1/4+o(1)}$ in Theorem~\ref{thm:meso-main}~\eqref{meso-clt}. This is not expected to be optimal. Rather than reflecting a genuine phase transition of the model, this restriction is a technical bottleneck arising from our current variance lower bound. Establishing the sharp variance order
$
 \var(T_n)\asymp n\ell^{-1-2/\theta}
$
would extend the Gaussian fluctuation limit essentially throughout the entire mesoscopic regime $1\ll \ell\ll n$, up to polylogarithmic factors. See the discussion in Section~\ref{sec:open}.
\end{rem}

Next, we focus on the macroscopic regime. As $\ell$ becomes proportional to $n$, the graph becomes dense and the model enters a mean-field-like regime. In this setting, the linear growth of the passage time and its Gaussian fluctuations completely vanish. Instead, the passage time collapses to a logarithmic scale, with fluctuations that are of order $O(1)$ and non-Gaussian.

\begin{thm}[Macroscopic fluctuation limit]\label{thm:Macro}
Assume that the edge weights satisfy Assumption~\ref{ass:weight} and $\ell/n\to \gl\in(0,1/2)$. Let $U_n$ be an independent uniform vertex in $\{1,2,\ldots,n-1\}$, $T_n^{(\ell)}(0,U_n)$ be the first-passage time defined in~\eqref{eq:unif-target-Tn}, and let $\fa_\ell$ be as in
\eqref{eq:aell-def}. Then
\begin{align*}
 \fa_\ell\cdot T_n^{(\ell)}(0,U_n)-\log n
 \Rightarrow S_\infty\text{ as } n\to\infty,
\end{align*}
where
\begin{align}\label{eq:macro-limit-representation}
 S_\infty\stackrel{d}{=}
 -\log M_\infty^{(1)}-\log M_\infty^{(2)}-G_3-\log\theta,
\end{align}
$M_\infty^{(1)},M_\infty^{(2)}$ are i.i.d.~copies of $M_\infty$ defined in~\eqref{eq:Minfty-def}, $G_3$ has the standard Gumbel distribution and is independent of $(M_\infty^{(1)},M_\infty^{(2)})$.
In particular,
\begin{align*}
 \frac{n^{1/\theta}}{\log n}\cdot T_n^{(\ell)}(0,U_n)
 \stackrel{\pr}{\to} \bigl(2\gl\Gamma(\theta+1)\bigr)^{-1/\theta} \text{ as } n\to\infty.
\end{align*}
The same results hold if we replace the random target $U_n$ by a fixed target $\sfu_n$ as defined in~\eqref{eq:fixed-target-Tn}.
\end{thm}

\begin{cor}[The exponential case $\theta=1$]\label{cor:exponential-case}
When $\theta=1$, the limiting CMJ process is the rate-one Yule process and $M_\infty\sim\Exp(1)$. Hence, if $G_i:=-\log M_\infty^{(i)}$ for $i=1,2$, then $G_1,G_2,G_3$ are independent standard Gumbel random variables and $S_\infty\equald G_1+G_2-G_3$.
\end{cor}

The mesoscopic law gives a passage-time scale $\sfu_n/(\fa_\ell\ell)$, whereas the macroscopic law gives the scale $(\log n)/\fa_\ell$. These two scales meet when $n/\ell$ is of order $\log n$, or equivalently when $\ell$ is of order $n/\log n$. The next theorem records the resulting first-order crossover and a proof is given in Section~\ref{subsec:Macro-mean}.

\begin{thm}[First-order crossover]\label{thm:macro-cross}
Assume Assumption~\ref{ass:weight}, fix $\sfu\in(0,1/2)$, and let $T_n=T_n^{(\ell)}(0,\sfu_n)$.
\begin{enumeratea}
 \item \label{cross:meso}
 If $1\ll\ell\ll n/\log n$, then
 \begin{align*}
 \frac{\fa_\ell\ell}{\sfu_n}T_n
 \stackrel{\pr}{\to}
 \frac1{\vstar},
 \qquad
 \frac{\fa_\ell\ell}{\sfu_n}\E T_n
 \to
 \frac1{\vstar}.
 \end{align*}

 \item \label{cross:right}
 If $\liminf_{n\to\infty}\frac{\ell\log n}{n} >0$ and $\ell\le\frac n2,$
 then
 \begin{align*}
 T_n
 =
 \Theta_{\pr}\left(
 \frac{\log n}{\ell^{1/\theta}}
 \right),
 \qquad
 \E T_n
 =
 \Theta\left(
 \frac{\log n}{\ell^{1/\theta}}
 \right).
 \end{align*}

 \item \label{cross:macro}
 If $\ell/n\to\gl\in(0,1/2)$, then
 \begin{align*}
 \frac{\fa_\ell T_n}{\log n}
 \stackrel{\pr}{\to}
 1,
 \qquad
 \frac{\fa_\ell\E T_n}{\log n}
 \to
 1.
 \end{align*}
\end{enumeratea}
\end{thm}

The second item includes the critical window $\ell\asymp n/\log n$. In particular, if $\ell\asymp n/\log n$, then
\begin{align*}
 T_n
 =
 \Theta_{\pr}\left(
 \frac{(\log n)^{1+1/\theta}}{n^{1/\theta}}
 \right),
 \qquad
 \E T_n
 =
 \Theta\left(
 \frac{(\log n)^{1+1/\theta}}{n^{1/\theta}}
 \right).
\end{align*}

For $\theta=1$, this is the order $(\log n)^2/n$. The theorem does not identify the exact constant in the critical window; the constant $1$ is proved here only when $\ell/n\to\gl\in(0,1/2)$. See Theorem~\ref{thm:Macro}.

We also analyze the hop-count, that is, the number of edges in the optimal path achieving the first-passage time. Since the edge-weight distribution is continuous and the graph is finite, the geodesic between any two vertices is almost surely unique; see the discussion around $\Omega_*$ in Section~\ref{sec:collision}.
\begin{thm}[Hop-count asymptotics]\label{thm:hop-count-asymptotics}
Assume that the edge weights satisfy Assumption~\ref{ass:weight}.
\begin{enumeratea}
\item \label{hop:meso} \textup{(Mesoscopic law of large numbers)}
Assume that $1 \ll \ell \ll n/\log n$. Let $H_n^{(\ell)}(0,\sfu_n)$ denote the number of edges in the a.s.\ unique geodesic realising $T_n^{(\ell)}(0,\sfu_n)$. Then, for fixed $\sfu\in(0,1/2)$, we have
\begin{align}\label{eq:meso-hop-mean}
 \frac{\ell}{\sfu_n} H_n^{(\ell)}(0,\sfu_n) \stackrel{\pr}{\to} \frac{\gbc}{\theta}
 \text{ and }
 \frac{\ell}{\sfu_n} \E H_n^{(\ell)}(0,\sfu_n)
 \to \frac{\gbc}{\theta} \text{ as } n\to\infty,
\end{align}
where $\gbc$ is the constant defined in~\eqref{eq:beta_c}.

\item \label{hop:macro} \textup{(Macroscopic law of large numbers)}
Assume that $\ell/n\to\gl\in(0,1/2)$. Let $U_n$ be chosen uniformly from $\dT_n\setminus\{0\}$ and let $H_n^{(\ell)}(0,U_n)$ be the number of edges in the geodesic from $0$ to $U_n$. Then
\begin{align}\label{eq:macro-hop-lln-main}
 \frac{H_n^{(\ell)}(0,U_n)}{\log n}
 \stackrel{\pr}{\to}\frac1\theta ,
 \text{ and }
 \frac{\E H_n^{(\ell)}(0,U_n)}{\log n} \to \frac{1}{\theta} \text{ as } n\to\infty.
\end{align}
\end{enumeratea}
\end{thm}
\begin{rem}
 Note that, in the mesoscopic regime, the average time per hop is $(\vstar \fa_\ell)^{-1}/(\gbc/\theta) = {\theta}/(\gf(\gbc)^{1/\theta} \fa_\ell)$. The constant ${\theta}/\gf(\gbc)^{1/\theta}$ can be viewed as the mean age between births under the tilted spine law with $\beta=\gbc,\kappa=\gf(\gbc)^{1/\theta}$, as $\int se^{-\kappa s}\mu_\theta(\dd s)/\int e^{-\kappa s}\mu_\theta(\dd s) =\theta/\kappa$. Theorem~\ref{thm:meso-main}~\eqref{meso-mean} and Theorem~\ref{thm:hop-count-asymptotics}~\eqref{hop:meso} are therefore mutually consistent. The same comparison in the macroscopic regime returns $\log n/\theta$ generations, matching Theorem~\ref{thm:hop-count-asymptotics}~\eqref{hop:macro}. See Section~\ref{ssec:local-cmj-objects} for more details. When $\theta\gg1$, we have $\gbc\sim \theta, \vstar\sim e/\theta, \fa_\ell\sim (2\ell)^{1/\theta}\theta/e$, so that $\ell/\sfu_n\cdot \E T_n\sim 1$ and $\ell/\sfu_n\cdot \E H_n\sim 1$ for large $n$. When $\theta\ll1$, we have $\gbc/\theta\sim \sqrt{3/\theta}\gg1$ and many short hops are taken by the geodesic.
\end{rem}

The proof of Theorem~\ref{thm:hop-count-asymptotics} is presented in Section~\ref{sec:hop-count-fluct}. Next, we present a brief roadmap for the proof of our main results presented in this section.

%%%%%%%%%%%%%%%%%%%%%%%%%%%%%%%%%%%%%%%%%%%%%%%%%%%
\subsection{Roadmap for the proofs}\label{subsec:roadmap}
Both mesoscopic and macroscopic regimes share the same local CMJ approximation, and we use standard tools from CMJ and renewal theory, such as branching random walks and stable-age distributions. However, the global proof strategies are fundamentally different. The main technical novelty of our approach lies in the exact, model-specific reductions that connect the geometry of the spread-out graph to these classical theories.

\subsubsection{Roadmap for the mesoscopic regime}
We identify $\vstar$ as the front speed of the limiting spatial CMJ branching random walk. To apply this ideal speed to our graph, we locally couple the branching tree to the spread-out line before identifying the optimal path. Proposition~\ref{prop:finite-graph-window-tools} establishes the time needed to cross a local segment, along with the independence properties required to concatenate successive segments together. 
The upper bound on passage time follows by iterating these, while the matching lower bound arises because the tree structure overcounts available paths (path overlaps). Finally, a standard comparison transfers this segment estimate to the cycle.

For fluctuation control, Lemma~\ref{lem:block-cut} decomposes the total passage time into a sum of independent block times with small boundary errors. Sections~\ref{sec:meso-moment} and~\ref{sec:meso-fluct} show that these boundary errors are negligible, establish the required variance bounds, and yield the Gaussian central limit theorem. Finally, the restriction $\ell^4 (\log \ell)^{6\gth} \ll n$ is simply an artifact of our variance lower bound proof, rather than a genuine transition in the fluctuation behavior.

\subsubsection{Roadmap for the macroscopic regime}
Here, we grow two exploration clusters from the source and target. This reduction separates into a deterministic geometric step and a probabilistic conditional step. The half-distance identity and Proposition~\ref{prop:cox}, respectively, yield
\begin{align*}
 \tau_{\col}=\text{the collision time}=\frac{1}{2}T_n,
 \text{ and }
 \pr(\fa_\ell \tau_{\col}>s)=\E e^{-\gL_s},
\end{align*}
where $\gL_s$, defined in~\eqref{eq:Lambda-def}, is the cumulative hazard accumulated by all bridge edges that have become active by rescaled time $s$. Informally, each active bridge contributes the hazard accumulated since it became available to connect the two growing clusters, so $\gL_s$ measures their total opportunity to collide by time $s$. Conditional on the exploration history, $e^{-\gL_s}$ is exactly the probability that none of these bridges has yet triggered a collision. Here, we note that the collision hazard $\gL_s$ depends on the actual spatial locations of the available vertices. 

\smallskip
\noindent{\bfseries Removal of the spatial geometry.}
We couple each growing cluster to an ideal branching tree, showing that internal self-intersections within a single cluster are negligible. However, tracking population sizes alone is insufficient, since the collision hazard $\gL_s$ depends on the actual spatial locations of the available vertices. Proposition~\ref{prop:pair} uses Fourier decay and an application of the Erd\H{o}s--Tur\'an inequality for interval discrepancy~\cite[Theorem~2.5]{kn74} to establish that the two clusters become uniformly mixed across the space. Because of this spatial mixing, the rate at which the two clusters connect simplifies to a mean-field product of their population sizes, $2\gl N_1(s)N_2(s)$, and the states of the connecting vertices become independent. This is the critical step that allows us to replace the complex spatial exploration with tractable, purely probabilistic calculations.

\smallskip
\noindent{\bfseries Integrated hazard and hop-count.}
This spatial mixing result implies that for every fixed $y\in\dR$,
\begin{align*}
 \gL_{\frac12\log n+y}
 \Rightarrow
 \theta M_\infty^{(1)}M_\infty^{(2)}e^{2y}\text{ as } n\to\infty.
\end{align*}
Combining this limit with the exact Cox formula yields an explicit Gumbel-type limit for the passage time. Tracking generations within the underlying branching process gives the macroscopic hop-count law of large numbers, while a path-counting argument handles the crossover behavior at the regime boundary.

%%%%%%%%%%%%%%%%%%%%%%%%%%%%%%%%%%%%%%%%%%%%%%%%%%%
\subsection{Literature review}\label{ssec:literature}
In this section, we will briefly review existing literature related to the model discussed in this article.
First-passage percolation on lattices was introduced by Hammersley and Welsh~\cite{hamwelsh65} as a model for flow of liquid through porous media; see Auffinger, Damron, and Hanson~\cite{adh17} for a modern account. In the classical lattice setting the graph is fixed, whereas here the deterministic range $\ell$ and hence the degree grow with $n$. We organize the discussion around the two mechanisms in our proofs: one-dimensional block accumulation in the mesoscopic regime and two-source collision in the macroscopic regime.

\begin{enumeratea}
\item \textbf{Finite-range FPP and regenerative models.}
Fixed-range and finite-width one-dimensional models are typically analyzed via regeneration or finite-dimensional state processes. Relevant results include finite-height and periodic graphs~\cite{herrndorf85,a15}, ladder and width-two models~\cite{renlund10,schlemm09,schlemm11,fgs11}, and thin cylinders with growing width~\cite{cd13}. Building on this regenerative picture, the microscopic companion paper~\cite{DK26a} shows that for a fixed range $\ell$, pivot vertices occur with positive density. This yields an exact renewal-reward decomposition and Gaussian or stable fluctuations. However, this mechanism is not uniform in $\ell$: as $\ell\to\infty$, the proliferation of bypasses makes pivot vertices too rare to decompose the path.

\item \textbf{Complete graphs and random-graph FPP.}
For the complete graph with independent mean-one exponential weights, Janson~\cite{Jan99} identified the typical distance, flooding time, and weighted diameter scales as $\log n/n$, $2\log n/n$, and $3\log n/n$, respectively. Bhamidi~\cite{bh08} subsequently developed an exploration-process approach. The closest precursor to our macroscopic theorem is Bhamidi and van der Hofstad~\cite{bh12}, where two growing clusters connect through a point-process collision, yielding a limit that involves two branching-process martingales and a Gumbel variable. When $\ell=\lfloor n/2\rfloor$, our graph is $K_n$, and their model matches our complete-graph endpoint after time normalization. For $\gl<1/2$, however, bridge pairs must also satisfy a deterministic spatial constraint, necessitating the spatial mixing arguments developed here. The corresponding complete-graph diameter is studied in~\cite{bvdh17}.

Related passage-time, hop-count, and flooding results for configuration, Erd\H{o}s--R\'enyi, and inhomogeneous random graphs appear in~\cite{bhh10a,bhh11,kk15,adl13,mountfordSaliba19}, with more general edge-weight laws treated in~\cite{bhh17,dss23}. In these models, the underlying graph is random and locally tree-like. In contrast, our branching approximations stem from different structural features: in the mesoscopic regime, they arise from the minimum of many incident weights in a deterministic band graph; in the macroscopic regime, bridge availability remains explicitly constrained by the deterministic cycle geometry.

\item \textbf{Small-world, long-range, and spatial models.}
The cycle serves as the local substrate for classical small-world models~\cite{wattsstrogatz98,newmanwatts99}. Barbour and Reinert~\cite{br01,br06} obtained branching approximations for graph distances in ring-based small worlds. For weighted Newman--Watts graphs, Komj\'athy and Vadon~\cite{kv16} proved a logarithmic passage-time limit and a hop-count central limit theorem. Their graph features a bounded mean degree and sparse random shortcuts, whereas our graph contains all edges up to the deterministic range $\ell$. 

In long-range percolation, edges are random with probabilities decaying over distance~\cite{bb01,biskup04,chatterjeeDey2016LRFPP}. Scale-free spatial models introduce heavy-tailed vertex heterogeneity~\cite{dvdhh13,dhw14,hhj17}. Other spatial comparisons include penalized complete-graph FPP~\cite{vdhl23}, random geometric graphs~\cite{colettiEtAl23}, and deterministic growing-degree graphs~\cite{martinsson18}. Notably, Mallein and Tesemnikov~\cite{mt22} demonstrated how varying a connectivity parameter in a sparse ordered directed graph sharply alters optimal path geometry. Our model similarly explores a connectivity-driven transition, but it retains a one-dimensional ambient geometry with a hard deterministic cutoff at distance $\ell$. For a broader perspective on empirical network densification, spatial features, and disordered media, see~\cite{barthelemy11,leskovec07,schlapfer14,shutters18,bbchs03}.

\item \textbf{Branching processes and CMJ theory.}
The martingale and stable-age inputs rely on the theory of Crump--Mode--Jagers processes~\cite{nerman81,jn84}, while the spatial front analysis utilizes branching-random-walk speed and change-of-measure methods~\cite{biggins95,lyons97,shi15,br05,Meiners2010}. In random-graph FPP, CMJ processes naturally arise from locally tree-like neighborhoods. Here, they emerge differently: after rescaling by $\fa_\ell$, the small incident edge weights converge to a Poisson reproduction process, while the displacement marks preserve the deterministic band geometry.

\end{enumeratea}

The present model provides a deterministic growing-range baseline for weighted propagation on a cycle. To the best of our knowledge, i.i.d.~FPP on deterministic powers of a line or cycle has not been previously studied in a regime where $\ell\to\infty$ with the system size. The mesoscopic analysis combines a spatial CMJ front with one-dimensional block fluctuations, while the macroscopic analysis couples the collision mechanism of~\cite{bh12} with Fourier mixing for spatially admissible bridges.

\subsection{Notations}\label{ssec:notation}
All asymptotics are taken as $n\to\infty$. The range $\ell$ may depend on $n$. We write $a_n\ll b_n$ if $a_n/b_n\to0$, and $a_n\asymp b_n$ if the ratio is bounded away from both $0$ and $\infty$. Constants denoted by $c,C,C_p$ may change from line to line and are independent of $n$ and $\ell$, although they may depend on fixed parameters such as $\theta,\sfu,\gl$, and $p$. For real numbers $a<b$, $[a,b]$ denotes the continuous interval or the discrete set $[a,b]\cap\dZ$ depending on the context.

For random variables, $X_n\Rightarrow X$ denotes convergence in distribution and $X_n\stackrel{\pr}{\to}X$ denotes convergence in probability. We write $X\equald Y$ for equality in distribution and $ \norm{X}_p:=\bigl(\E|X|^p\bigr)^{1/p}
$ for the $L_p$ norm of $X$ for $p\ge1$.
For an integrable random variable $X$, we write $\overline X:=X-\E X$ when no confusion is possible. The total variation distance between two probability measures $\mu$ and $\nu$ is denoted by $\dTV(\mu,\nu)$.
The time normalization $\fa_\ell=\fa_{\ell,\theta}$ is defined in~\eqref{eq:aell-def}.

For a graph or an induced subgraph $A$, $T_A(u,v)$ denotes the restricted first-passage time from $u$ to $v$, and
\begin{align*}
 T_A(u,B):=\inf_{b\in B}T_A(u,b)
\end{align*}
for $B\subseteq A$. If the ambient graph is clear, the subscript is omitted.
The CMJ reproduction measure, the spatial CMJ speed $\vstar$, and the non-spatial CMJ martingale limit $M_\infty$ are defined in Section~\ref{sec:local-cmj}. Most proof-specific variables, such as block lengths, cut times, collision times, bridge hazards, and auxiliary CMJ couplings, are defined locally in the sections where they are used.

%%%%%%%%%%%%%%%%%%%%%%%%%%%%%%%%%%%%%%%%%%%%%%%%%%%
\subsection{Organization of the paper}\label{ssec:organization}
The rest of the paper is organized as follows.
In Section~\ref{sec:prelim}, we present the preliminary results on local CMJ approximation and the block decomposition technique. 
Section~\ref{sec:meso-moment} develops the mesoscopic moment estimates: the
complete-graph flooding bound (Lemma~\ref{lem:complete-flooding}), the dyadic
second- and fourth-moment upper bounds on a line segment
(Lemma~\ref{lem:block-moments}), the Dijkstra hazard-rescaling variance lower
bound (Proposition~\ref{prop:var-lb}), and the transfer of the variance estimates
to the cycle; together these prove Theorem~\ref{thm:meso-main}~\eqref{meso-var}.
Section~\ref{sec:meso-fluct} converts a variance lower bound into a Gaussian
limit via the block decomposition and Lyapunov's theorem, proving
Theorem~\ref{thm:meso-main}~\eqref{meso-clt}.
Section~\ref{sec:cmj-mean} identifies the mesoscopic time constant $\vstar$
through a spatial CMJ branching random walk, a coupling of the limiting CMJ tree
with the $\ell$-spread-out tree, and a chaining of short-window crossings; this
proves Theorem~\ref{thm:meso-main}~\eqref{meso-mean} and completes the transfer
of the CLT from the segment to the cycle.
Section~\ref{sec:macro-cox} sets up the two-source censored exploration in the
macroscopic regime and derives the exact half-distance identity
(Proposition~\ref{prop:half}) and the exact Cox bridge-hazard representation
(Proposition~\ref{prop:cox}).
Section~\ref{sec:macro-cmj} carries out the finite-$n$ CMJ approximation, the
Fourier method based Erd\H{o}s--Tur\'an spatial-mixing argument, and the asymptotics of the
integrated bridge hazard, and proves Theorem~\ref{thm:Macro}; its final
Subsection~\ref{subsec:Macro-mean} proves the first-order crossover result,
Theorem~\ref{thm:macro-cross}.
Section~\ref{sec:hop-count-fluct} proves the hop-count asymptotics of
Theorem~\ref{thm:hop-count-asymptotics} in both regimes.
Section~\ref{sec:open} collects open problems. Appendix~\ref{app:cmj-to-lspread} contains the proofs of the CMJ process to $\ell$-spread out coupling results used in Section~\ref{sec:cmj-mean} and Appendix~\ref{app:macro-cmj-inputs} contains the proofs of the finite-$n$ CMJ
results used in Section~\ref{sec:macro-cmj}.

%%%%%%%%%%%%%%%%%%%%%%%%%%%%%%%%%%%%%%%%%%%%%%%%%%%
\section{Preliminary results}\label{sec:prelim}

In this Section, we fix the local normalization and the limiting CMJ objects used throughout the mesoscopic and macroscopic proofs; and setup the basics of block decomposition technique. The global proof arguments for the main results are kept separate: the mesoscopic fluctuation proof (Section~\ref{sec:meso-fluct}) uses block decompositions, the mesoscopic mean proof (Section~\ref{sec:cmj-mean}) uses a spatial CMJ branching random walk, and the macroscopic proof (Sections~\ref{sec:macro-cox} and~\ref{sec:macro-cmj}) uses two-source growth together with a Cox bridge-hazard representation. 
The common theme is the same in all cases: after the correct time rescaling, the small edge weights incident to a typical vertex converge to a CMJ reproduction point process. 

%%%%%%%%%%%%%%%%%%%%%%%%%%%%%%%%%%%%%%%%%%%%%%%%%%%
\subsection{Local scaling and limiting CMJ processes}\label{sec:local-cmj}
Recall that the edge distribution satisfies Assumption~\ref{ass:weight}. The exact exponential structure is not in the original time variable $\go$, but in the hazard variable $H(\go)$. This observation is the fundamental mechanism behind both the Dijkstra perturbation argument and the macroscopic Cox representation.

\begin{lem}[Deferred decisions in hazard scale]
\label{lem:hazard-residual}
Let $\cI$ be a finite set of edges and let $L_e\ge0$, $e\in\cI$, be deterministic lower bounds. Conditionally on the event
$
 \{\go_e>L_e\text{ for all }e\in\cI\},
$
the random variables
\begin{align*}
 H(\go_e)-H(L_e)=\go_e^\theta-L_e^\theta,
 \quad e\in\cI,
\end{align*}
are independent rate-one exponentials. The same holds conditionally on any exploration history that reveals only the lower bounds $L_e$ for the edges in $\cI$.
\end{lem}
\begin{proof}
For $x\ge0$, we have 
$
\pr\bigl(H(\go_e)-H(L_e)>x \bigm| \go_e>L_e\bigr)
= e^{-H(L_e)-x}/e^{-H(L_e)}
=e^{-x}$. Independence follows because the conditioning event factors over the edges.
\end{proof}

We use the common time normalization
$\fa_\ell$ from~\eqref{eq:aell-def}.
Since $\theta$ is fixed throughout the paper, we usually suppress it from the notation. The choice of $\fa_\ell$ in~\eqref{eq:aell-def} incorporates both the effective non-backtracking degree $(2\ell-1)$, which dominates the local exploration geometry asymptotically, and the factor $\Gamma(\theta+1)$. The latter is specifically chosen to ensure that the limiting offspring measure $\mu_\theta$ has a Malthusian parameter of exactly one.
Indeed, if $\go\sim \Exp(1)^{1/\theta}$, using the fact that $\fa_\ell^\theta=(2\ell-1)\theta\Gamma(\theta)$, we obtain the identity
\begin{align}\label{eq:aell-density}
 (2\ell-1)\cdot \pr(\fa_\ell\go\in\dd u)
 =
 \frac{u^{\theta-1}}{\Gamma(\theta)}
 e^{-(u/\fa_\ell)^\theta}\,\dd u,
 \quad u>0.
\end{align}
Thus, for every fixed $\ga>0$, we have $
(2\ell-1)\E e^{-\ga \fa_\ell\go}\to \ga^{-\theta}$ as $\ell \to \infty$.
The limiting CMJ reproduction measure is
\begin{align}\label{eq:limiting-mu}
 \mu_\theta(\dd s)
 := \frac{s^{\theta-1}}{\Gamma(\theta)}\,\dd s,
 \quad s>0.
\end{align}
It is locally finite and satisfies
\begin{align}\label{eq:mu-laplace}
 \int_0^\infty e^{-\ga s}\mu_\theta(\dd s)
 =\ga^{-\theta},
 \quad \ga>0.
\end{align}
In particular,
$
 \int_0^\infty e^{-s}\mu_\theta(\dd s)=1,
$
so the Malthusian parameter is $1$. Also,
$
 \int_0^\infty s e^{-s}\mu_\theta(\dd s)=\theta.
$
The following point-process formulation is the local convergence used in the CMJ couplings.

\begin{lem}[Local offspring point-process limit]
\label{lem:local-offspring-pp}
Let $D_\ell\in\{2\ell-1,2\ell\}$ and let $\go_1,\ldots,\go_{D_\ell}$ be i.i.d.~copies of $\Exp(1)^{1/\theta}$. Then, as $\ell \to \infty$, the point process $\Pi_\ell:=\sum_{j=1}^{D_\ell}\gd_{\fa_\ell\go_j}$ converges in distribution (in the vague topology on locally finite point measures on $(0,\infty)$) to a $\operatorname{PPP}(\mu_\theta)$.
\end{lem}

\begin{proof}
Let $f$ be non-negative, continuous and compactly supported in $(0,\infty)$, and
put $g:=1-e^{-f}$. By~\eqref{eq:aell-density} and dominated convergence,
$D_\ell\,\E g(\fa_\ell\go)\to\int_0^\infty g(s)\mu_\theta(\dd s)$, using
$D_\ell/(2\ell-1)\to1$. Therefore
\begin{align*}
 \E e^{-\la f,\Pi_\ell\ra}
 =\bigl(1-\E g(\fa_\ell\go)\bigr)^{D_\ell}
 \to
 \exp\left\{-\int_0^\infty\bigl(1-e^{-f(s)}\bigr)\mu_\theta(\dd s)\right\},
\end{align*}
which is the Laplace functional of $\operatorname{PPP}(\mu_\theta)$.
\end{proof}

%%%%%%%%%%%%%%%%%%%%%%%%%%%%%%%%%%%%%%%%%%%%%%%%%%%
\subsubsection{The Crump--Mode--Jagers process and the observables used in the two regimes}
\label{ssec:local-cmj-objects}

A Crump--Mode--Jagers process is specified by assigning to every individual an independent copy of a point process of ages at which that individual gives birth. This reproduction point process is a $\operatorname{PPP}(\mu_\theta)$. Let
\begin{align*}
 \sT:=\bigcup_{k\ge0}\dN^k
\end{align*}
be the Ulam--Harris tree, with root $\varnothing$. For each $u\in\sT$, let
$
 0<\tau_{u,1}<\tau_{u,2}<\cdots
$
be the points of an independent Poisson point process on $(0,\infty)$ with intensity $\mu_\theta$. Writing $ui$ for the $i$-th child of $u$, define the birth times recursively by
\begin{align*}
 \sft(\varnothing):=0,
 \qquad
 \sft(ui):=\sft(u)+\tau_{u,i},
 \quad i\ge1.
\end{align*}
For $t\ge0$, let
\begin{align*}
 \cN_t:=\{v\in\sT:\sft(v)\le t\},
 \qquad
 Z_\theta(t):=|\cN_t|.
\end{align*}
Thus $Z_\theta(t)$ is the number of individuals born by time $t$, including the root. Since $\E Z_\theta(t) = \sum_{k\ge0} t^{k\theta}/\Gamma(k\theta+1) < \infty$, the process is non-explosive.

The Malthusian parameter is $1$ by~\eqref{eq:mu-laplace}. The standard non-lattice and $x\log x$ conditions are readily verified (e.g., $\E(\sum_{i\ge1}e^{-\tau_{\varnothing,i}})^2 = 1+2^{-\theta} < \infty$). Hence, the Malthusian martingale and random-characteristic theorems~\cite{nerman81,jn84} give
\begin{align}\label{eq:Minfty-def}
 e^{-t}Z_\theta(t) \to M_\infty \quad \text{a.s.\ and in } L^1.
\end{align}
The limit $M_\infty$ is strictly positive almost surely and non-degenerate, with $\E M_\infty = 1/\theta$. For $\theta=1$, this corresponds to the rate-one Yule process and $M_\infty \sim \Exp(1)$.

The random-characteristic theory also gives the stable-age distribution: for every bounded continuous function $g$,
\begin{align*}
 \frac1{Z_\theta(t)} \sum_{v\in\cN_t}g(t-\sft(v))
 \stackrel{\pr}{\to} 
 \int_0^\infty g(a)e^{-a}\,\dd a.
\end{align*}
Thus, the limiting age of a uniformly sampled individual has the $\Exp(1)$ law. If $A_1, A_2 \stackrel{\mathrm{i.i.d.}}{\sim} \Exp(1)$ represent two independent stable ages, their sum follows a $\mathrm{Gamma}(2,1)$ distribution, yielding
\begin{align}\label{eq:age-pair-gamma}
 \E(A_1+A_2)^{\theta-1}
 =
 \int_0^\infty a^{\theta-1}a e^{-a}\,\dd a
 =
 \Gamma(\theta+1).
\end{align}
Proposition~\ref{prop:pair} later proves that the endpoint ages sampled by macroscopic bridge edges asymptotically follow this exact product law, supplying the necessary uniform integrability for $0<\theta<1$. This stable-age moment~\eqref{eq:age-pair-gamma} naturally explains the presence of $\Gamma(\theta+1)$ in the normalization $\fa_\ell$ defined in~\eqref{eq:aell-def}.

For the mesoscopic regime, we study a spatially marked version of this CMJ process. We assign each parent--child pair $(u, ui)$ an independent displacement $\xi_{u,i} \sim \Unif(-1,1)$ and define
\begin{align*}
 \sfp(\varnothing):=0,
 \qquad
 \sfp(ui):=\sfp(u)+\xi_{u,i}.
\end{align*}
The resulting marked CMJ process is a branching random walk, with rightmost position
\begin{align*}
 \sfM_t:=\max_{v\in\cN_t}\sfp(v)
\end{align*}
at time $t$. For $\beta\in\dR$, let
$
 \gf(\beta):=\E e^{\beta\xi}
$
be the moment generating function of $\xi$.
For a spatial tilt $\beta>0$, the tilted Malthusian exponent $\gk_\theta(\beta)$ satisfying $\gf(\beta) \int_0^\infty e^{-\gk_\theta(\beta)s}\mu_\theta(\dd s) = 1$ is simply $\gk_\theta(\beta) = \gf(\beta)^{1/\theta}$.
 By the general branching-random-walk speed theory~\cite{biggins95}, the right-front speed is therefore
\begin{align*}
 \vstar = \vstar_\theta = \inf_{\beta>0}{\gk_\theta(\beta)}/{\beta} = \inf_{\beta>0}{\gf(\beta)^{1/\theta}}/{\beta}.
\end{align*}
Proposition~\ref{prop:cmj-speed} later records the exact form used here.
Hence the two regimes use the same CMJ process: the mesoscopic proof retains the spatial marks and studies $\sfM_t$, whereas the macroscopic proof forgets the marks and studies $Z_\theta(t)$ and the associated ages.

%%%%%%%%%%%%%%%%%%%%%%%%%%%%%%%%%%%%%%%%%%%%%%%%%%%
\subsection{Block decomposition}
\label{ssec:meso-block}
Block decomposition plays an important role in the proof of the moment upper bounds and of the central limit theorem in the mesoscopic regime. We first work on a finite line segment, where the block decomposition produces an exact sum of independent block passage times, and then compare the resulting estimates with the passage time on the cycle.

For two integers $u<v$, we write $[u,v]:=\{u,u+1,\ldots,v\}$. For an integer $m\ge1$, we write
\begin{align*}
 \cE_{[0,m]}^{(\ell)}
 :=\bigl\{\{u,v\}:0\le u<v\le m,\ v-u\le\ell\bigr\}
\end{align*}
for the edge set of the $\ell$-spread-out segment on $[0,m]$, and we set
\begin{align}\label{eq:segment-Tm}
 T_{[m]}:=T^{(\ell)}_{[0,m]}(0,m).
\end{align}
\emph{Throughout the mesoscopic sections, $T_{[m]}$ always denotes this
$\ell$-spread-out \emph{segment} passage time of spatial length $m$, on a segment
carrying an independent copy of the weight field; $T_n$ always denotes the cycle
passage time~\eqref{eq:fixed-target-Tn}.} In particular $T_{[\sfu_n]}$, not
$T_{[n]}$, is the segment quantity that approximates $T_n$.

We begin with the deterministic approximation that replaces the original graph by a graph cut at regularly spaced interfaces. The cut graph has an exact sum decomposition into independent block-crossing times. 
The penalty for imposing this independence is strictly local: whenever the original geodesic jumps over an interface, we reroute it through the interface vertex. The required rerouting cost is bounded by two flooding times in windows of size $\ell$.
This approximation argument does not use the CMJ front-speed asymptotics and does not attempt to identify the optimal local fluctuation scale. Its role is to show that the centered passage time can be approximated in $L^2$ by a triangular array of independent block variables with an explicit error bound. With an appropriate variance lower bound, one can obtain a Gaussian CLT.

Fix an integer block length $b\in(2\ell,n/2)$ and put
\begin{align*}
 q:=\left\lfloor n/b\right\rfloor\ge 2,
 \qquad
 r:=n-qb\in\{0,1,\ldots,b-1\}.
\end{align*}
Define the deterministic interfaces
\begin{align*}
 x_i:=ib\ \text{ for } 0\le i\le q-1,
 \qquad
 x_q:=n.
\end{align*}
Thus, the blocks $[x_0,x_1],[x_1,x_2],\ldots,[x_{q-2},x_{q-1}]$ have length $b$, while the last block $[x_{q-1},x_q]$ has length $b+r\in[b,2b)$. The block-cut graph is obtained from the $\ell$-spread-out finite segment by deleting all edges whose endpoints do not lie in the same block. 
Let 
\begin{align}\label{eq:block-cut-sum}
 S_n^{(b)}:=
 \sum_{i=1}^{q} T^{(\ell)}_{[x_{i-1},x_i]}(x_{i-1},x_i)
\end{align}
denote the first-passage time from $0$ to $n$ in the block-cut graph. Since the block variables depend on disjoint edge sets, the variables $T^{(\ell)}_{[x_{i-1},x_i]}(x_{i-1},x_i)$, $i=1,2,\ldots,q-1$, are i.i.d.~with the law of $T_{[b]}$, while $T^{(\ell)}_{[x_{q-1},x_q]}(x_{q-1},x_q)$ has the law of $T_{[b+r]}$ and is independent of the other summands. Figure~\ref{fig:block-decomposition} illustrates this.

\begin{figure}[htbp]
\centering
\begin{tikzpicture}[x=1.3cm,y=1cm,every node/.style={font=\small}]
 % main segment
 \draw[thick] (0,0) -- (11.4,0);

 % key points
 \coordinate (A0) at (0,0);
 \coordinate (A1) at (2.0,0);
 \coordinate (A2) at (4.0,0);
 \coordinate (A3) at (6.0,0);
 \coordinate (A4) at (8.5,0);
 \coordinate (A5) at (10.2,0);
 \coordinate (A6) at (11.4,0);

 % points and labels
 \foreach \P/\lab in {
 A0/$0$,
 A1/$b$,
 A2/$2b$,
 A3/$3b$,
 A4/$(q-1)b$,
 A5/$qb$,
 A6/$n$
 }{
 \fill (\P) circle (1.5pt);
 \draw (\P)+(0,0.14) -- +(0,-0.14);
 \node[below=5pt] at (\P) {\lab};
 }

 % ellipses
 \node at (7.5,0.28) {$\cdots$};
 \node at (7.5,-0.90) {$\cdots$};

 % highlighted representative cut/interface only at x_3=3b
 \draw[thick] (A3)+(0,-0.32) -- +(0,0.85);

 % label for the representative deleted edges
 \node[align=center] at (6.0,1.15)
 {dotted arcs = representative deleted edges};

 % representative deleted crossing edges at the highlighted cut x_3=3b
 \draw[densely dotted,thick]
 (5.00,0) to[out=60,in=120] (6.90,0);
 \draw[densely dotted,thick]
 (4.70,0) to[out=66,in=114] (6.85,0);
 \draw[densely dotted,thick]
 (5.15,0) to[out=52,in=128] (7.15,0);
 \draw[densely dotted,thick]
 (5.45,0) to[out=42,in=138] (7.35,0);

 % block brackets below
 \draw (0,-0.72) -- (2.0,-0.72);
 \draw (0,-0.62) -- (0,-0.82);
 \draw (2.0,-0.62) -- (2.0,-0.82);
 \node[below=3pt] at (1.0,-0.72) {$[x_0,x_1]$};

 \draw (2.0,-0.72) -- (4.0,-0.72);
 \draw (2.0,-0.62) -- (2.0,-0.82);
 \draw (4.0,-0.62) -- (4.0,-0.82);
 \node[below=3pt] at (3.0,-0.72) {$[x_1,x_2]$};

 \draw (4.0,-0.72) -- (6.0,-0.72);
 \draw (4.0,-0.62) -- (4.0,-0.82);
 \draw (6.0,-0.62) -- (6.0,-0.82);
 \node[below=3pt] at (5.0,-0.72) {$[x_2,x_3]$};

 \draw (8.5,-0.72) -- (11.4,-0.72);
 \draw (8.5,-0.62) -- (8.5,-0.82);
 \draw (11.4,-0.62) -- (11.4,-0.82);
 \node[below=3pt] at (9.95,-0.72) {$[x_{q-1},x_q]$};

 % remainder marker inside the last block
 \draw (10.2,0.28) -- (11.4,0.28);
 \draw (10.2,0.20) -- (10.2,0.36);
 \draw (11.4,0.20) -- (11.4,0.36);
 \node[above=3pt] at (10.8,0.28)
 {remainder $r=n-qb$};
\end{tikzpicture}
\caption{Block decomposition on the finite segment $[0,n]$. In this graph, every edge crossing an interface is deleted; the dotted arcs show representative deleted edges at $x_3=3b$.}
\label{fig:block-decomposition}
\end{figure}
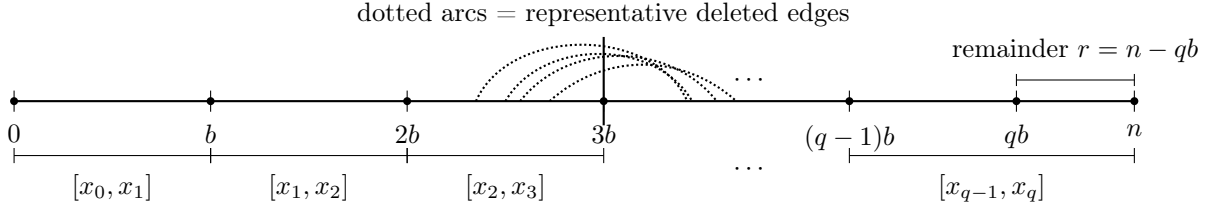

For each internal interface $x_i$, $1\le i\le q-1$, define the local left/right flooding times as
\begin{align*}
Y_{i,-}^{(\ell)}
&:=
\max_{1\le j< \ell}
T^{(\ell)}_{(x_i-\ell,x_i]}(x_i-j,x_i),\qquad
Y_{i,+}^{(\ell)}
:=
\max_{1\le j< \ell}
T^{(\ell)}_{[x_i,x_i+\ell)}(x_i,x_i+j).
\end{align*} 
These variables control the local detour cost. In the $\ell$-spread-out graph, a geodesic might bypass the interface $x_i$ by jumping directly over it. To force the path to visit $x_i$ and decouple the adjacent blocks, we must reroute the crossing edge through $x_i$. The sum $Y_{i,-}^{(\ell)} + Y_{i,+}^{(\ell)}$ provides a uniform upper bound for this rerouting cost. We denote by $Y^{(\ell)}$ a generic copy of $Y_{i,+}^{(\ell)}$.

\begin{lem}\label{lem:block-cut}
Assume that $b>2\ell$. Then, with $S_n^{(b)}$ as defined in~\eqref{eq:block-cut-sum}, we have
\begin{align}\label{eq:block-decomp-comparison}
 0\le S_n^{(b)}-T_{[n]} \le \sum_{i=1}^{q-1}(Y_{i,-}^{(\ell)} + Y_{i,+}^{(\ell)}).
\end{align}
\end{lem}

\begin{proof}[Proof of Lemma~\ref{lem:block-cut}]
Deleting edges can only increase passage times, which yields $T_{[n]}\le S_n^{(b)}$.
 In the block decomposition graph, any path from $0$ to $n$ must visit the interfaces $x_1,\ldots,x_{q-1}$ in order, which decouples the passage time into the sum $S_n^{(b)}$ given in~\eqref{eq:block-cut-sum}.

To prove the upper bound, let $\pi^{(0)}$ be a geodesic from $0$ to $n$ in the original segment. We recursively construct paths $\pi^{(1)},\ldots,\pi^{(q-1)}$ such that $\pi^{(i)}$ does not jump over any of the interfaces $x_1,\ldots,x_i$.
Given $\pi^{(i-1)}$, let $a_i$ be the last vertex on this path to the left of $x_i$, and $b_i$ be the first vertex to the right of $x_i$ ($a_i<x_i\le b_i$). If the edge $\{a_i, b_i\}$ jumps over $x_i$ (i.e., $a_i < x_i < b_i$), then $a_i \in (x_i-\ell, x_i)$ and $b_i \in (x_i, x_i+\ell)$. 
To eliminate this jump, we replace the single edge $\{a_i, b_i\}$ with the concatenation of two shortest paths: $\tilde\pi_{i,-}$ from $a_i$ to $x_i$ inside $(x_i-\ell, x_i]$, and $\tilde\pi_{i,+}$ from $x_i$ to $b_i$ inside $[x_i, x_i+\ell)$.

Let $\pi^{(i)}$ be the loop-erased version of this modified path. Since $\tilde\pi_{i,-}$ and $\tilde\pi_{i,+}$ have weights bounded by $Y_{i,-}^{(\ell)}$ and $Y_{i,+}^{(\ell)}$ respectively, and loop-erasure decreases total weight, we obtain
\begin{align*}
    W(\pi^{(i)}) \le W(\pi^{(i-1)}) + Y_{i,-}^{(\ell)} + Y_{i,+}^{(\ell)},
\end{align*}
where $W(\cdot)$ denotes the total weight of a path. Since $b>2\ell$, the rerouting at $x_i$ is restricted to $(x_{i-1}, x_{i+1})$, preventing new jumps over earlier interfaces. Thus, $\pi^{(q-1)}$ is a valid path in the block decomposition graph. Iterating the weight bound yields
\begin{align*}
    S_n^{(b)} \le W(\pi^{(q-1)}) \le T_{[n]} + \sum_{i=1}^{q-1} \bigl( Y_{i,-}^{(\ell)} + Y_{i,+}^{(\ell)} \bigr),
\end{align*}
proving~\eqref{eq:block-decomp-comparison}.
\end{proof}

Note that an interval of length at most $\ell$ induces a complete graph on its vertex set. Thus, we can control the local detour error using the following lemma.

\begin{lem}
\label{lem:complete-flooding}
Let the complete graph $K_m$ have i.i.d.~edge weights with law $\Exp(1)^{1/\theta}$. For distinct vertices $u,v$, let $T(u,v)$ be their first-passage time, and let
\begin{align*}
 Y^{(m)} := \max_{w\in V(K_m)}T(u,w)
\end{align*}
be the flooding time from $u$. For every fixed $p\ge1$, there is $C_p<\infty$ such that, for all $m\ge2$,
\begin{align}\label{eq:complete-flooding}
 \sup_{u\ne v} \norm{T(u,v)}_p \le C_p\frac{\log m}{m^{1/\theta}},
 \text{ and }
 \norm{Y^{(m)}}_p \le C_p\frac{(\log m)^{\gth}}{m^{1/\theta}},
 \text{ where } \gth=\max\{{1,\theta^{-1}}\}.
\end{align}
\end{lem}

\begin{rem}
    For $s=1/\theta$, Bhamidi and van der Hofstad~\cite[Theorem~1.1]{bh12} derive the limiting law for the distance between two fixed vertices, confirming the point-to-point order in the first estimate of~\eqref{eq:complete-flooding}. Their theorem does not provide the uniform flooding bound in the second estimate of~\eqref{eq:complete-flooding}; this is the additional estimate supplied by Lemma~\ref{lem:complete-flooding}.
\end{rem}

\begin{proof}[Proof of Lemma~\ref{lem:complete-flooding}]
It suffices to prove when $m\ge16$, adjusting $C_p$ for smaller values. A standard Dijkstra exploration reveals information about unexplored edges, losing their independence. We bypass this difficulty by building a tree step-by-step, restricted to previously unobserved edges.

Fix a root $u$. For the two-point bound, fix a target $v\ne u$ and exclude it from the tree growth; for the flooding bound, the excluded set is empty. Let $F_0=\{u\}$, $A_0=F_0$, and $k_0=1$. Inductively, suppose the current frontier $F_j$ has size $k_j=2^j$, and let $A_j=\bigcup_{i=0}^j F_i$. We double the frontier size at each step ($k_j = 2^j$) to ensure the tree reaches a macroscopic size in $O(\log m)$ steps, while maintaining a uniform bound on the connection cost at each layer. As long as $k_j<m/8$, let $U_j$ be the set of unreached vertices outside $A_j$ and the excluded set. 

We restrict our attention to the edges between $F_j$ and $U_j$. Since these edges have not been previously examined, their weights are unconditioned and independent of the exploration history. Given this history, the minimum connection weights
\begin{align*}
 X_{j,w} := \min_{z\in F_j}\go_{z,w}, \quad w\in U_j,
\end{align*}
are independent and satisfy $X_{j,w}^\theta \sim \Exp(k_j)$. Since $|A_j|=2k_j-1<m/4$, the number of available vertices is $N_j:=|U_j|\ge m/2$.

Let $r_j:=2k_j$. We form the next frontier $F_{j+1}$ by selecting the $r_j$ vertices in $U_j$ that minimize $X_{j,w}$. The cost to complete this layer is the $r_j$-th order statistic $X_j^\star$ of the sequence $(X_{j,w})_{w\in U_j}$.
 By standard properties of exponential order statistics, we can write $(X_j^\star)^\theta \stackrel{\mathrm d}{=} \sum_{i=0}^{r_j-1} \frac{E_{j,i}}{k_j(N_j-i)}$, where $E_{j,i}$ are i.i.d.~$\Exp(1)$. Using the bounds $N_j \ge m/2$ and $i < 2k_j < m/4$, the denominator satisfies $k_j(N_j-i) > k_j m/4$. Consequently, $(X_j^\star)^\theta$ is stochastically dominated by $\frac{4}{k_j m} G_j$, where $G_j \sim \mathrm{Gamma}(r_j, 1)$. 

For any $q>0$, the moments of the Gamma distribution satisfy $\E(G_j^q) \le C_q r_j^q$. By choosing $q=p/\theta$ and substituting $r_j=2k_j$, the factor $k_j$ cancels out, yielding a uniform bound in $j$,
\begin{align}\label{eq:fresh-frontier-stage}
 \norm{X_j^\star}_p \le C_p m^{-1/\theta}.
\end{align}
Consequently, every vertex in $F_{j+1}$ can be reached from $u$ within time $S_{j+1}:=\sum_{i=0}^j X_i^\star$.

Let $J$ be the first index such that $k_J\ge m/8$. Then $J\le C\log m$. By Minkowski's inequality and~\eqref{eq:fresh-frontier-stage},
\begin{align}\label{eq:fresh-frontier-growth-cost}
 \norm{S_J}_p \le C_p\frac{\log m}{m^{1/\theta}}.
\end{align}

To establish the two-point bound, recall that the target $v$ was intentionally excluded from the tree growth. Thus, the edges between $F_J$ and $v$ remain unconditioned. The minimum connection cost $M_v := \min_{z\in F_J}\go_{z,v}$ satisfies $M_v^\theta \sim \Exp(k_J)$. Since $k_J \ge m/8$, it follows that $\norm{M_v}_p \le C_p m^{-1/\theta}$. Using the triangle inequality $T(u,v) \le S_J + M_v$, we obtain the first bound in~\eqref{eq:complete-flooding}.

For the flooding bound, we reveal the fresh edges from $F_J$ to all unreached vertices $w \notin A_J$. Given the exploration history, the minimal connection costs $M_w := \min_{z\in F_J}\go_{z,w}$ are mutually independent and satisfy $M_w^\theta \sim \Exp(k_J)$. A union bound gives
\begin{align*}
 \pr\left( \max_w M_w>t \;\middle|\; F_0,F_1,\ldots,F_J \right) \le m\exp\{-k_J t^\theta\}.
\end{align*}
Integrating this tail probability and using the fact that $k_J \ge m/8$, we obtain $\norm{\max_w M_w}_p \le C_p\left(\frac{\log m}{m}\right)^{1/\theta}$. 
Every vertex is reached by time $S_J+\max_w M_w$. Since $\log m + (\log m)^{1/\theta} \le 2(\log m)^{\gth}$, combining this with~\eqref{eq:fresh-frontier-growth-cost} proves the second bound in~\eqref{eq:complete-flooding}.
\end{proof}

%%%%%%%%%%%%%%%%%%%%%%%%%%%%%%%%%%%%%%%%%%%%%%%%%%%
\section{Mesoscopic moment control: proof of Theorem~\ref{thm:meso-main}~\eqref{meso-var}}
\label{sec:meso-moment}

The block decomposition reduces the fluctuation problem to local moment estimates. We control the flooding cost at a single interface using Lemma~\ref{lem:complete-flooding}, and use a dyadic splitting argument to obtain variance and fourth-moment upper bounds on arbitrary line segments. We next prove the variance lower bound by combining~\cite[Lemma 1.2]{cha19} with a Dijkstra perturbation idea. 
Finally, we obtain the variance estimate for the passage time on the cycle.

%%%%%%%%%%%%%%%%%%%%%%%%%%%%%%%%%%%%%%%%%%%%%%%%%%%
\subsection{Moment upper bounds}
\label{ssec:recursive-moment-bounds}

We next use the complete-graph estimate to obtain the variance and fourth-moment bounds for passage times.

\begin{lem}\label{lem:block-moments}
There exists $C<\infty$ such that, for all $\ell\ge2$ and all $m\ge2\ell$, with $T_{[m]}:=T_{[0,m]}^{(\ell)}(0,m)$, we have
\begin{align}
 \var(T_{[m]})&\le C\frac{m(\log\ell)^{2\gth}}{\ell^{1+2/\theta}},
 \label{eq:block-var-bound}\\
 \E\abs{T_{[m]}-\E T_{[m]}}^4
 &\le C\frac{m^2(\log\ell)^{4\gth}}{\ell^{2+4/\theta}}.
 \label{eq:block-fourth-bound}
\end{align}
\end{lem}

\begin{proof}
Set
\begin{align*}
 \eta_{\ell,\theta}:=\frac{(\log\ell)^{\gth}}{\ell^{1/\theta}},
\end{align*}
and choose $k_0=k_0(\ell)$ with $2^{k_0-1}\le4\ell<2^{k_0}$. For
$r\ge k_0$, define
\begin{align*}
 M_{2,r}:=\max_{2\ell\le m\le2^r}\norm{T_{[m]}-\E T_{[m]}}_2,
 \qquad
 M_{4,r}:=\max_{2\ell\le m\le2^r}\norm{T_{[m]}-\E T_{[m]}}_4.
\end{align*}
For the base case $2\ell\le m\le2^{k_0}<8\ell$, partition $[0,m]$ into at most eight contiguous subintervals, each of length between $\ell/2$ and $\ell$. If such a subinterval has length $d$, then its induced graph is $K_{d+1}$, where $d\asymp\ell$. Lemma~\ref{lem:complete-flooding}, applied with $m=d+1$, gives for $p=2,4$
\begin{align*}
 \norm{T^{(\ell)}_{[a,a+d]}(a,a+d)}_p
 \le C\eta_{\ell,\theta}.
\end{align*}
Concatenating these subinterval crossings gives an admissible path from $0$ to $m$. By Minkowski's inequality and $\norm{X-\E X}_p\le2\norm{X}_p$,
\begin{align}
 M_{2,k_0}\le C\eta_{\ell,\theta},
 \qquad
 M_{4,k_0}\le C\eta_{\ell,\theta}.
 \label{eq:base-block}
\end{align}

For the recursive step, consider an interval length $m$ such that $2^{r-1}<m\le2^r$ for some $r>k_0$, and split $[0,m]$ at the midpoint $v=\lfloor m/2\rfloor$. 
Let
\begin{align*}
 T^L:=T^{(\ell)}_{[0,v]}(0,v),
 \qquad
 T^R:=T^{(\ell)}_{[v,m]}(v,m),
 \qquad
 T^I:=T^{(\ell)}_{[0,m]}(0,m).
\end{align*}
The two-block version of Lemma~\ref{lem:block-cut} gives
\begin{align*}
 0\le T^L+T^R-T^I\le Y_{v,-}^{(\ell)}+Y_{v,+}^{(\ell)}.
\end{align*}
After centering, the half-block terms $T^L-\E T^L$ and $T^R-\E T^R$ are independent, and the centered cut-error has an $L^2$ norm bounded by $C\eta_{\ell,\theta}$. Since the variance of the sum of independent variables is the sum of their variances, applying the triangle inequality in $L^2$ gives
\begin{align*}
 M_{2,r}\le \sqrt2\,M_{2,r-1}+C\eta_{\ell,\theta}.
\end{align*}
Iterating from $k_0$ gives, for every $r\ge k_0$,
\begin{align*}
 M_{2,r}\le C2^{(r-k_0)/2}\eta_{\ell,\theta}.
\end{align*}
If $2^{r-1}<m\le2^r$, then $2^r\le2m$ and $2^{k_0}>4\ell$, so
$
 2^{(r-k_0)/2}\le C\sqrt{m/\ell}.
$
Therefore
\begin{align*}
 \norm{T_{[m]}-\E T_{[m]}}_2
 \le C\sqrt{\frac m\ell}\frac{(\log\ell)^{\gth}}{\ell^{1/\theta}},
\end{align*}
which proves~\eqref{eq:block-var-bound}.

For the fourth moment, the same split gives
\begin{align*}
 T^I-\E T^I=X+X'+E,
\end{align*}
where $X:=T^L-\E T^L$, $X':=T^R-\E T^R$ are independent centered half-block terms and $\norm{E}_4\le C\eta_{\ell,\theta}$. Since $X$ and $X'$ are independent and centered,
\begin{align*}
 \norm{X+X'}_4
 &=\left(\E X^4+\E (X')^4+6\E X^2\E (X')^2\right)^{1/4}
 \le 2^{1/4}M_{4,r-1}+C M_{2,r-1}.
\end{align*}
Consequently,
\begin{align*}
 M_{4,r}\le 2^{1/4}M_{4,r-1}+C M_{2,r-1}+C\eta_{\ell,\theta}.
\end{align*}
Setting $d=r-k_0$ and iterating this inequality, while substituting the established $L^2$ bound $M_{2,r-1-j}\le C2^{(d-1-j)/2}\eta_{\ell,\theta}$, yields
\begin{align*}
 M_{4,r}
 &\le 2^{d/4}M_{4,k_0}
 +C\sum_{j=0}^{d-1}2^{j/4}M_{2,r-1-j}
 +C\eta_{\ell,\theta}\sum_{j=0}^{d-1}2^{j/4}\\
 &\le C2^{d/4}\eta_{\ell,\theta}
 +C\eta_{\ell,\theta}\sum_{j=0}^{d-1}2^{j/4}2^{(d-1-j)/2}
 +C2^{d/4}\eta_{\ell,\theta}
 \le C2^{d/2}\eta_{\ell,\theta}.
\end{align*}
Thus, if $2^{r-1}<m\le2^r$,
\begin{align*}
 \norm{T_{[m]}-\E T_{[m]}}_4
 \le C\sqrt{\frac m\ell}\frac{(\log\ell)^{\gth}}{\ell^{1/\theta}}.
\end{align*}
Raising to the fourth power gives~\eqref{eq:block-fourth-bound}.
\end{proof}

Lemma~\ref{lem:complete-flooding} controls the centered block-cut error, while~\eqref{eq:block-fourth-bound} is the Lyapunov input for the independent block sum. The remaining ingredient is a lower bound on $\var(T_{[n]})$.

%%%%%%%%%%%%%%%%%%%%%%%%%%%%%%%%%%%%%%%%%%%%%%%%%%%
\subsection{Variance lower bound}
\label{ssec:meso-var-lower}

We now prove a variance lower bound needed for the Gaussian fluctuation argument using the perturbative idea from Chatterjee~\cite{cha19}, see Lemma~\ref{lem:var-cc} below. This roughly says that if a random variable and a nearby perturbation have close laws but are usually separated by a deterministic amount, then the original variable cannot be too concentrated.

\begin{lem}[{\cite[Lemma 1.2]{cha19}}]\label{lem:var-cc}
Let $X\in L^2$ and let $X'$ be another real-valued random variable. Suppose that there exist $\Delta>0$ and $p,\eta\in[0,1)$ with $p+\eta<1$ such that
$\pr(\abs{X-X'}\le\Delta)\le p$, and 
$\dTV(\cL(X),\cL(X'))\le\eta$. Then
\begin{align*}
\var(X)\ge \frac14\gD^2\pr(\abs{X-\E X}\ge \gD/2)\ge\frac{1-p-\eta}{8}\Delta^2.
\end{align*}
\end{lem}

However, a direct perturbation of all $\Theta(n\ell)$ edge variables is too costly.
Instead, in the cumulative-hazard scale $H(t)=t^\theta$, the unrevealed Dijkstra residuals are conditionally independent exponentials. Comparing one vertex-discovery kernel at a time leads to a total variation cost of order $\eps\sqrt n$, uniformly in $\ell$.
There is also an unavoidable local contribution from the first edge leaving an endpoint. Indeed, after conditioning on all edge weights not incident to the source, the passage time is the minimum of independent variables with deterministic shifts.
Applied at the source, the next lemma gives an endpoint contribution of order $\ell^{-2/\theta}$, which becomes dominant when $\ell$ is at least of order $\sqrt n$.

\begin{lem}
\label{lem:shifted-var-min}
Fix $\theta>0$. There exists a constant $c_\theta>0$ such that the following
holds for every integer $d\ge1$. Let $a_1,\ldots,a_d\in[0,\infty)$ be
deterministic, let $\go_1,\ldots,\go_d$ be i.i.d.~random variables satisfying
$\go_i\equald\Exp(1)^{1/\theta}$. Set
$
 Y:=\min_{1\le i\le d}(a_i+\go_i).
$
Then
\begin{align}\label{eq:shifted-var-min}
 \var(Y)\ge c_\theta d^{-2/\theta}.
\end{align}
\end{lem}

\begin{proof}
For $t\ge0$, define $g(t):=\sum_{i=1}^d \max\{t-a_i,0\}^\theta$.
By independence and the exact Weibull tail $\pr(\go_i>x)=e^{-x^\theta}$, the survival function of $Y$ is
\begin{align}\label{eq:shifted-surv}
 \pr(Y>t)
 =\prod_{i=1}^d\exp\bigl(-\max\{t-a_i,0\}^\theta\bigr)
 =e^{-g(t)}.
\end{align}
The function $g$ is continuous, $g(t)=0$ for $t\le \min_i a_i$, strictly increasing for $t>\min_i a_i$, and $g(t)\to\infty$ as $t\to\infty$.
Hence there is a unique $t_0\ge0$ such that $g(t_0)=1$.
We claim that there exists $\delta\in(0,1]$, depending only on $\theta$, such that with $s:=\delta d^{-1/\theta}$, we have
\begin{align}\label{eq:shifted-hazard-inc}
 g(t_0+s)-g(t_0)\le1.
\end{align}

If $0<\theta\le1$, the subadditivity of $x\mapsto x^\theta$ on $[0,\infty)$ gives
$\max\{t_0+s-a_i,0\}^\theta-\max\{t_0-a_i,0\}^\theta\le s^\theta$.
Consequently,
\begin{align*}
 g(t_0+s)-g(t_0)\le ds^\theta = \delta^\theta,
\end{align*}
and~\eqref{eq:shifted-hazard-inc} follows by choosing any $\delta\le1$.

If $\theta>1$, there is a constant $C_\theta<\infty$ such
that $(x+y)^\theta-x^\theta\le C_\theta(x^{\theta-1}y+y^\theta)$ for all $x,y\ge0$.
Setting $x_i:=\max\{t_0-a_i,0\}$, the bound $\max\{t_0+s-a_i,0\}\le x_i+s$ gives
\begin{align*}
 g(t_0+s)-g(t_0)
 \le C_\theta s\sum_{i=1}^d x_i^{\theta-1}
 +C_\theta ds^\theta.
\end{align*}
By H\"older's inequality and $g(t_0)=1$,
\begin{align*}
 \sum_{i=1}^d x_i^{\theta-1}
 \le d^{1/\theta}\left(\sum_{i=1}^d x_i^\theta\right)^{(\theta-1)/\theta}
 =d^{1/\theta}.
\end{align*}
Therefore, $g(t_0+s)-g(t_0)\le C_\theta(\delta+\delta^\theta)$. Choosing $\delta>0$ sufficiently small proves~\eqref{eq:shifted-hazard-inc} for $\theta>1$.

Let $Y'$ be an independent copy of $Y$. By~\eqref{eq:shifted-surv} and~\eqref{eq:shifted-hazard-inc},
\begin{align*}
 \pr(Y\le t_0)=1-e^{-1}
 \text{ and }
 \pr(Y\ge t_0+s)=e^{-g(t_0+s)}\ge e^{-2}.
\end{align*}
It follows that
\begin{align*}
 \var(Y)
 =\frac12\E (Y-Y')^2
 &\ge s^2\pr(Y'\le t_0)\pr(Y\ge t_0+s)
 \ge (1-e^{-1})e^{-2}s^2.
\end{align*}
Since $s^2=\delta^2d^{-2/\theta}$, this proves~\eqref{eq:shifted-var-min}.
\end{proof}

We now apply this criterion to the passage time $T_{[n]}$. The proof consists of two components. First, we establish a lower-tail bound to ensure that $T_{[n]}$ is large with high probability. Second, we execute the dynamic Dijkstra hazard-rescaling to bound the total variation distance between the original and perturbed passage times.

\begin{prop}[Variance lower bound]\label{prop:var-lb}
Assume $n/\ell\to\infty$ and $\go_e\equald\Exp(1)^{1/\theta}$. 
Then there exists $c>0$ such that, for all sufficiently large $n$, we have
\begin{align}\label{eq:line-var-lb}
 \var(T_{[n]})\ge c \left( \frac{n}{\ell^{2+2/\theta}} \vee \frac{1}{\ell^{2/\theta}} \right).
\end{align}
Similarly, we have
\begin{align}\label{eq:cycle-var-lb}
 \var\bigl(T_n\bigr)
 \ge c\left( \frac n{\ell^{2+2/\theta}} \vee \frac{1}{\ell^{2/\theta}} \right).
\end{align}
\end{prop}

\begin{proof}
We first treat the finite segment and begin with the lower-tail estimate. Define the threshold
\begin{align*}
 t=t_n:=t_0\frac n{\ell^{1+1/\theta}},
\end{align*}
where $t_0>0$ will be chosen later. Since all edge weights are positive, on the event $\{T_{[n]}\le t\}$ there is a simple path from $0$ to $n$ with total weight at most $t$. Any such path has at most $n$ edges and, since every edge has Euclidean length at most $\ell$, it has at least $\lceil n/\ell\rceil$ edges.

Let $\go_1,\ldots,\go_k$ be i.i.d.~copies of $\go_e$. Since the density of $\go_e$ is $\theta x^{\theta-1}e^{-x^\theta}\le\theta x^{\theta-1}$, we have
\begin{align*}
 \pr\left(\sum_{i=1}^k\go_i\le t\right)
 \le \int_{\substack{x_i\ge0\\ x_1+\cdots+x_k\le t}}
 \prod_{i=1}^k \theta x_i^{\theta-1}\,\dd x_1\cdots\dd x_k 
 = \frac{\Gamma(\theta+1)^k}{\Gamma(k\theta+1)}t^{k\theta}
 \le \left(C_\theta\frac{t^\theta}{k^\theta}\right)^k.
\end{align*}
The number of $k$-step paths starting from $0$ is at most $(2\ell)^k$.
Therefore by a union bound,
\begin{align*}
 \pr(T_{[n]}\le t)
 \le
 \sum_{k=\lceil n/\ell\rceil}^{n}
 (2\ell)^k
 \left(C_\theta\frac{t^\theta}{k^\theta}\right)^k
 =
 \sum_{k=\lceil n/\ell\rceil}^{n}
 \left(
 2C_\theta\ell
 \left(\frac{t_0n/\ell^{1+1/\theta}}{k}\right)^\theta
 \right)^k
 \le
 \sum_{k=\lceil n/\ell\rceil}^{n}
 \left(2C_\theta t_0^\theta\right)^k.
\end{align*}
Choose $t_0>0$ so small that $C_\theta t_0^\theta\le1/4$. Then there exists
$c_0>0$ such that
\begin{align}\label{eq:lower-tail-var-lb}
 \pr\left(T_{[n]}\le t_0\frac n{\ell^{1+1/\theta}}\right)
 \le 2\exp\left(-c_0\frac n\ell\right).
\end{align}

Next we prove the Dijkstra hazard-rescaling comparison. Run Dijkstra's algorithm from $0$ until all vertices have been discovered. Let $A_j$ be the discovered set after $j$ discoveries beyond the initial vertex, so $A_0=\{0\}$ and $\abs{A_j}=j+1$. Let $D_j$ be the $j$-th discovery time, with $D_0=0$, and write
\begin{align*}
 \partial A_j:= \bigl\{\{u,v\}\in\cE_{[0,n]}^{(\ell)}:\ u\in A_j,\ v\notin A_j\bigr\}.
\end{align*}
For an active edge $e=\{u_e,v_e\}\in\partial A_j$ with $u_e\in A_j$ and
$v_e\notin A_j$, let $D(u_e)$ denote the discovery time of $u_e$. The current lower bound on this edge weight is $\go_e>D_j-D(u_e)$.
Since $H(t)=t^\theta$, Lemma~\ref{lem:hazard-residual} implies that, conditional on the exploration history, the residual hazards
\begin{align*}
 H(\go_e)-H(D_j-D(u_e)),
 \qquad e\in\partial A_j,
\end{align*}
are independent $\Exp(1)$ random variables.
For $q\ge D_j$, define
\begin{align*}
 A_{j,e}(q)
 :=H(q-D(u_e))-H(D_j-D(u_e)),
 \qquad
 A_j(q):=\sum_{e\in\partial A_j}A_{j,e}(q).
\end{align*}
In this hazard scale, the conditional law of the next winning edge and its discovery time is given by the Stieltjes measure
\begin{align*}
 e^{-A_j(q)}\,\dd A_{j,e}(q),
 \qquad e\in\partial A_j,
 \quad q\ge D_j.
\end{align*}
Expressing the conditional law as a Stieltjes measure avoids the singularity at $q=D_j$, where the standard density diverges for $\theta<1$.

For $0<\eps\le1/2$, let $\sfP_\eps$ be the probability measure obtained by scaling the cumulative hazard to
\begin{align*}
 H_\eps(t):=(1+\eps)H(t).
\end{align*}
Equivalently, the edge weights under $\sfP_\eps$ have the same distribution as $(1+\eps)^{-1/\theta}\go_e$. Under $\sfP_\eps$, the corresponding one-step measure becomes
\begin{align*}
 (1+\eps)e^{-(1+\eps)A_j(q)}\,\dd A_{j,e}(q).
\end{align*}
Conditional on the exploration history, the Hellinger affinity of a single Dijkstra step is therefore
\begin{align*}
\alpha_\eps
&:=
\sum_{e\in\partial A_j}\int_{D_j}^{\infty}
\sqrt{e^{-A_j(q)}(1+\eps)e^{-(1+\eps)A_j(q)}}\,\dd A_{j,e}(q)\\
&=\sqrt{1+\eps}\int_{D_j}^{\infty}
 e^{-(2+\eps)A_j(q)/2}\,\dd A_j(q)
=\sqrt{1+\eps}\int_0^\infty e^{-(2+\eps)a/2}\,\dd a
=\frac{2\sqrt{1+\eps}}{2+\eps}.
\end{align*}
This quantity is independent of the past.

Let $\sfP$ and $\sfP_\eps$ denote the laws of the full Dijkstra process over $n$ successive discoveries, and define
\begin{align*}
\gr(P,Q)
:=\int
\sqrt{
\frac{\dd P}{\dd\lambda}
\frac{\dd Q}{\dd\lambda}
}
\,\dd\lambda,
\end{align*}
where $\lambda$ is any common dominating measure. 
 If $\gr_j$ denotes the affinity of the exploration laws up to the $j$-th discovery, conditioning on the common history and applying the one-step estimate above yields
$
 \gr_{j+1}\ge\alpha_\eps\gr_j.
$
Induction therefore yields
\begin{align*}
 \gr(\sfP,\sfP_\eps)\ge\alpha_\eps^n.
\end{align*}
Moreover,
$
 1-\alpha_\eps^2
 =\frac{\eps^2}{(2+\eps)^2}
 \le C\eps^2.
$
Using $1-x^n\le n(1-x)$ for $x\in[0,1]$ and the standard relation between total variation and Hellinger affinity,
\begin{align*}
 \dTV(\sfP,\sfP_\eps)
 \le\sqrt{1-\gr(\sfP,\sfP_\eps)^2}
 \le\sqrt{1-\alpha_\eps^{2n}}
 \le C\eps\sqrt n.
\end{align*}

Let $\tau_n:=\min\{j:\ n\in A_j\}$. The map from the full Dijkstra exploration to the hitting time of $n$ sends $\sfP$ to $\cL(T_{[n]})$ and sends $\sfP_\eps$ to the law of $(1+\eps)^{-1/\theta}T_{[n]}$, because $H_\eps(t)=(1+\eps)t^\theta$ is exactly the law obtained by multiplying every edge weight by $(1+\eps)^{-1/\theta}$. Therefore
\begin{align}\label{eq:tv-var-lb-T}
 \dTV\left(
 \cL(T_{[n]}),
 \cL\left((1+\eps)^{-1/\theta}T_{[n]}\right)
 \right)
 \le C\eps\sqrt n.
\end{align}

Recall that $t_n=t_0n/\ell^{1+1/\theta}$. Since $n/\ell\to\infty$,
\eqref{eq:lower-tail-var-lb} implies that, for all sufficiently large $n$,
\begin{align}
 \pr(T_{[n]}<t_n)\le \frac18.
 \label{eq:tail-18-var-lb}
\end{align}
Let $\CTV$ be the constant in~\eqref{eq:tv-var-lb-T}. Choose
a fixed $\eps_0>0$ so small that
\begin{align*}
 \CTV\eps_0\le1/8,
 \text{ and set }
 \eps:=\eps_0 n^{-1/2}.
\end{align*}
For all sufficiently large $n$, $0<\eps\le1/2$. Define, on the original
probability space,
\begin{align*}
 T_{[n]}^{(\eps)}:=(1+\eps)^{-1/\theta}T_{[n]}.
\end{align*}
By~\eqref{eq:tv-var-lb-T},
\begin{align*}
 \dTV(\cL(T_{[n]}),\cL(T_{[n]}^{(\eps)}))
 \le \CTV\eps\sqrt n
 =\CTV\eps_0
 \le\frac18.
\end{align*}
For $0<\eps\le1/2$, we have $1-(1+\eps)^{-1/\theta}\ge c_\theta\eps$
for some $c_\theta>0$. Set
\begin{align*}
 \Delta_n:=\frac{c_\theta\eps t_n}{2}
 =c'_\theta\frac{\sqrt n}{\ell^{1+1/\theta}}.
\end{align*}
On the event $\{T_{[n]}\ge t_n\}$,
\begin{align*}
 \bigl|T_{[n]}-T_{[n]}^{(\eps)}\bigr|
 =\bigl(1-(1+\eps)^{-1/\theta}\bigr)T_{[n]}
 \ge c_\theta\eps t_n
 >\Delta_n.
\end{align*}
Therefore, by~\eqref{eq:tail-18-var-lb},
\begin{align*}
 \pr\bigl(\bigl|T_{[n]}-T_{[n]}^{(\eps)}\bigr|\le\Delta_n\bigr)
 \le \pr(T_{[n]}<t_n)
 \le\frac18.
\end{align*}
Applying Lemma~\ref{lem:var-cc} with $X=T_{[n]}$, $X'=T_{[n]}^{(\eps)}$, $\gD=\gD_n$, and
$p=\eta=1/8$, we get
\begin{align}\label{eq:line-bulk-var-lb}
 \var(T_{[n]})\ge c\Delta_n^2
 \ge c\frac n{\ell^{2+2/\theta}}.
\end{align}

It remains to prove the endpoint contribution. Let
\begin{align*}
 \cF_0
 &:=
 \gs\bigl(
 \go_e:
 e\in\cE_{[0,n]}^{(\ell)},\
 0\notin e
 \bigr),\text{ and }
 A_i
 := T_{[1,n]}(i,n),
 \quad
 1\le i\le\ell.
\end{align*}
The induced graph on $[1,n]$ is connected, so each $A_i$ is finite and $\cF_0$-measurable. Since all edge weights are strictly positive, every geodesic from $0$ to $n$ is simple. Thus,
\begin{align}\label{eq:line-endpoint-min}
 T_{[n]}
 =
 \min_{1\le i\le\ell}
 \bigl(
 A_i+\go_{0,i}
 \bigr).
\end{align}
Conditionally on $\cF_0$, the shifts $(A_i)_{1\le i\le\ell}$ are deterministic, while the variables $(\go_{\{0,i\}})_{1\le i\le\ell}$ are independent with law $\Exp(1)^{1/\theta}$. Lemma~\ref{lem:shifted-var-min} and the law of total variance therefore give
\begin{align}\label{eq:line-endpoint-var-lb}
 \var(T_{[n]})
 \ge
 \E\left[
 \var(T_{[n]}\mid\cF_0)
 \right]
 \ge c_\theta\ell^{-2/\theta}.
\end{align}
Combining~\eqref{eq:line-bulk-var-lb} and
\eqref{eq:line-endpoint-var-lb} proves~\eqref{eq:line-var-lb}.

The variance bound for the cycle passage time follows from the same argument.
For fixed $\sfu\in(0,1/2)$, any path from $0$ to $\sfu_n$ in $\dT_n^{(\ell)}$ uses at least $\sfu_n/\ell \asymp n/\ell$ edges. By the same path-counting argument (now summing over both orientations of the cycle, which only doubles the union bound), we obtain the analogous lower-tail bound
\begin{align}\label{eq:cycle-lower-tail}
	\pr\left(T_n\le c_1n\ell^{-1-1/\theta}\right)
 \le 2\exp\left(-c_2{n}/{\ell}\right),
\end{align}
for some constants $c_1, c_2>0$. Since Dijkstra's algorithm on the cycle still requires at most $n-1$ discoveries, the total variation estimate $\dTV(\sfP,\sfP_\eps) \le C\eps\sqrt n$ holds verbatim. Taking $\eps=\eps_0n^{-1/2}$ and applying Lemma~\ref{lem:var-cc} exactly as before gives the bulk bound
\begin{align}\label{eq:cycle-bulk-var-lb}
 \var(T_n)\ge c\,\frac{n}{\ell^{2+2/\theta}}.
\end{align}
Also, applying the endpoint argument at the source, where the degree is now $2\ell$, we obtain
\begin{align}\label{eq:cycle-endpoint-var-lb}
 \var(T_n)
 \ge
 c_\theta(2\ell)^{-2/\theta}
 \ge
 c\ell^{-2/\theta}.
\end{align}
Combining~\eqref{eq:cycle-bulk-var-lb} and~\eqref{eq:cycle-endpoint-var-lb} proves~\eqref{eq:cycle-var-lb}.
\end{proof}

%%%%%%%%%%%%%%%%%%%%%%%%%%%%%%%%%%%%%%%%%%%%%%%%%%%
\subsection{Variance upper bound for the cycle passage time}
\label{ssec:cycle-var-upper}
For the upper bound, we cut the cycle at the endpoints $0$ and $\sfu_n$, bounding $T_n$ by the minimum of the two restricted arc passage times,
\begin{align}\label{eq:two-arc-def}
 T_{n,\mathrm{short}}:=T^{(\ell)}_{[0,\sfu_n]}(0,\sfu_n)\equald T_{[\sfu_n]},
 \qquad
 T_{n,\mathrm{long}} := T^{(\ell)}_{[\sfu_n,n]}(\sfu_n,n)\equald T_{[n-\sfu_n]},
\end{align}
where we recall from~\eqref{eq:segment-Tm} that $T_{[m]}$ is the segment passage time of spatial length $m$, and where the vertex $n$ of the arc is identified with the vertex $0$ of the cycle.
Let $M_n := \min\{T_{n,\mathrm{short}}, T_{n,\mathrm{long}}\}$. The two arcs use disjoint edge sets, hence $T_{n,\mathrm{short}}$ and $T_{n,\mathrm{long}}$ are independent.

Applying Lemma~\ref{lem:block-cut} to control the boundary effects at $0$ and $\sfu_n$ yields
\begin{align}
0\le M_n-T_n
\le Z_0^{(\ell)}+Z_{\sfu_n}^{(\ell)},
\label{eq:cycle-two-arc-comparison}
\end{align}
where, for an interface $x$, we write
\begin{align}\label{eq:Z-detour-def}
 Z_x^{(\ell)}:=Y_{x,-}^{(\ell)}+Y_{x,+}^{(\ell)}
\end{align}
for the local detour time at $x$, with $Y_{x,\pm}^{(\ell)}$ as in the display preceding Lemma~\ref{lem:block-cut}. Each detour term is bounded by the sum of two flooding times in windows of size $\ell$, \ie\ by two copies of $Y^{(\ell)}$. Thus, Lemma~\ref{lem:complete-flooding} implies
\begin{align}\label{eq:cycle-endpoint-detour-L2}
\norm{Z_0^{(\ell)}+Z_{\sfu_n}^{(\ell)}}_2
\le C\frac{(\log\ell)^\gth}{\ell^{1/\theta}}.
\end{align}

Since $(a,b)\mapsto\min\{a,b\}$ is $1$-Lipschitz in each coordinate and the two arc times are independent, the Efron--Stein inequality gives $\var(M_n) \le \var(T_{n,\mathrm{short}}) + \var(T_{n,\mathrm{long}})$. Combining this with Lemma~\ref{lem:block-moments}, applied to segments of lengths $\sfu_n\le n$ and $n-\sfu_n\le n$, and with~\eqref{eq:cycle-endpoint-detour-L2}, gives
\begin{align*}
\sqrt{\var(T_n)}
\le \sqrt{\var(M_n)}
+2\norm{M_n-T_n}_2
&\le
\sqrt{\var(T_{[\sfu_n]})
+\var(T_{[n-\sfu_n]})}
+C\frac{(\log\ell)^\gth}{\ell^{1/\theta}}\\
&\le C\sqrt{
\frac{n(\log\ell)^{2\gth}}
{\ell^{1+2/\theta}}
}.
\end{align*}
Here the first inequality is the triangle inequality in $L^2$ applied to $T_n-\E T_n=(M_n-\E M_n)-\bigl((M_n-T_n)-\E(M_n-T_n)\bigr)$.
Together with~\eqref{eq:cycle-var-lb}, this proves~\eqref{eq:meso-var-ublb} in Theorem~\ref{thm:meso-main}.

The block decomposition, the fourth-moment estimate, and the variance lower bound are now in place. We combine them in the next section to prove the Gaussian limit on the finite segment.

%%%%%%%%%%%%%%%%%%%%%%%%%%%%%%%%%%%%%%%%%%%%%%%%%%%
\section{Mesoscopic CLT: proof of Theorem~\ref{thm:meso-main}~\eqref{meso-clt}}
\label{sec:meso-fluct}

We now combine the block decomposition in Subsection~\ref{ssec:meso-block}, the moment estimates in Lemma~\ref{lem:block-moments}, the variance lower bound in Proposition~\ref{prop:var-lb} and Lyapunov's central limit theorem applied to a triangular array of independent block passage times.
We begin by clarifying exactly how the Gaussian fluctuations rely on the variance lower bound. The two conditions on the block length below serve distinct purposes: the first ensures that the centered block-cut error is negligible, while the second guarantees Lyapunov's condition for the sum of independent blocks.

%%%%%%%%%%%%%%%%%%%%%%%%%%%%%%%%%%%%%%%%%%%%%%%%%%%
\begin{thm}\label{thm:clt-from-var-lb}
Here $n$ denotes the length of the segment; the statement will be applied with $n$ replaced by other lengths $m$ with $m\asymp n$.
Assume $1\ll\ell\ll n$ and
$\go_e\equald\Exp(1)^{1/\theta}$. Suppose that,
for some constants $c>0$, $\ga\ge0$, and $A\in\dR$,
\begin{align}
 \var(T_{[n]})\ge c\frac{n}{\ell^\ga(\log\ell)^A}
 \label{eq:var-lb}
\end{align}
for all sufficiently large $n$. Suppose further that there exists an integer sequence $b=b_{n,\ell}$
such that $b>2\ell$, $b=o(n)$, and
\begin{align}
 \frac{\ell^{\ga-2/\theta}(\log\ell)^{A+2\gth}}{b}\to0,
 \qquad
 \frac{b\ell^{2\ga-2-4/\theta}(\log\ell)^{2A+4\gth}}{n}\to0.
 \label{eq:b-conditions}
\end{align}
Then
\begin{align*}
 \frac{T_{[n]}-\E T_{[n]}}{\sqrt{\var(T_{[n]})}}
 \Rightarrow\N(0,1).
\end{align*}
\end{thm}

The first condition in~\eqref{eq:b-conditions} makes the centered block-cut error $S_n^{(b)}-T_{[n]}$ small compared to $\sqrt{\var(T_{[n]})}$. The second condition is Lyapunov's condition for the independent block variables, using the fourth-moment estimate in Lemma~\ref{lem:block-moments}. This is the point at which the sharpness of the variance lower bound directly determines the range of $\ell$ for which the present proof gives a CLT.

\begin{proof}[Proof of Theorem~\ref{thm:clt-from-var-lb}]
The proof has four steps. We first approximate $\oT_{[n]}$ in $L^2$ by the centered block-cut sum $\oS_n^{(b)}$. 
We then compare their variances, verify Lyapunov's condition for the independent block variables, and transfer the resulting Gaussian limit back to $\oT_{[n]}$.

\smallskip
\noindent\textbf{Step 1. Block decomposition and $L^2$ approximation.}
Choose $b$ satisfying~\eqref{eq:b-conditions}, and define $q,r,x_i$ as in
Subsection~\ref{ssec:meso-block}. For $0\le s<t\le q$, set
\begin{align*}
 \cT_{s,t}:=T^{(\ell)}_{[x_s,x_t]}(x_s,x_t),
 \qquad
 \cS_{s,t}:=
 \sum_{j=s+1}^{t}T^{(\ell)}_{[x_{j-1},x_j]}(x_{j-1},x_j),
\end{align*}
and define the centered block-cut error
\begin{align*}
 \Delta_{s,t}
 :=(\cS_{s,t}-\cT_{s,t})-
 \E(\cS_{s,t}-\cT_{s,t}).
\end{align*}
For $1\le N\le q$, let
\begin{align*}
 \eps_N:=\max_{\substack{0\le s<t\le q\\ t-s=N}}
 \norm{\Delta_{s,t}}_2.
\end{align*}
Since $\cS_{s,t}=\cT_{s,t}$ when $t-s=1$, we have $\eps_1=0$.

Fix $0\le s<t\le q$ with $N:=t-s\ge2$, and put
$
 h:=s+\lfloor N/2\rfloor
$.
Since $\cS_{s,t}=\cS_{s,h}+\cS_{h,t}$, we have 
\begin{align*}
 \cS_{s,t}-\cT_{s,t}
 = (\cS_{s,h}-\cT_{s,h})+(\cS_{h,t}-\cT_{h,t})+U_{s,h,t},
\end{align*}
where
$U_{s,h,t}:=\cT_{s,h}+\cT_{h,t}-\cT_{s,t}$.
By the two-block version of Lemma~\ref{lem:block-cut} at the interface
$x_h$, $0\le U_{s,h,t}\le Z_h^{(\ell)}$. Hence
\begin{align*}
 \norm{U_{s,h,t}-\E U_{s,h,t}}_2
 \le C\norm{Y^{(\ell)}}_2.
\end{align*}
The variables $\Delta_{s,h}$ and $\Delta_{h,t}$ are independent, since they
depend on disjoint edge sets. Therefore
\begin{align*}
 \norm{\Delta_{s,t}}_2
 \le
 \left(\norm{\Delta_{s,h}}_2^2+\norm{\Delta_{h,t}}_2^2\right)^{1/2}
 +C\norm{Y^{(\ell)}}_2.
\end{align*}
Taking the maximum over all $s<t$ with $t-s=N$ gives
\begin{align}
 \eps_N
 \le
 \left(\eps_{\lfloor N/2\rfloor}^2+
 \eps_{N-\lfloor N/2\rfloor}^2\right)^{1/2}
 +C\norm{Y^{(\ell)}}_2.
 \label{eq:eps-recur}
\end{align}

Put
\begin{align*}
 \mathsf e_r:=\max_{1\le N\le\min\{2^r,q\}}\eps_N.
\end{align*}
Then $\mathsf e_0=0$, and~\eqref{eq:eps-recur} implies, for $r\ge1$,
\begin{align*}
 \mathsf e_r\le \sqrt2\,\mathsf e_{r-1}+C\norm{Y^{(\ell)}}_2.
\end{align*}
Iterating, we get
\begin{align*}
 \mathsf e_r\le C2^{r/2}\norm{Y^{(\ell)}}_2.
\end{align*}
Choosing $r$ such that $2^{r-1}<q\le2^r$, we obtain
\begin{align*}
 \eps_q\le C\sqrt q\,\norm{Y^{(\ell)}}_2.
\end{align*}

Since $\cT_{0,q}=T_{[n]}$ and $\cS_{0,q}=S_n^{(b)}$, and since $\norm{Y^{(\ell)}}_2\le C(\log\ell)^{\gth}/\ell^{1/\theta}$ by Lemma~\ref{lem:complete-flooding}, we get
\begin{align}
 \norm{\oT_{[n]}-\oS_n^{(b)}}_2
 &=\norm{(S_n^{(b)}-T_{[n]})-\E(S_n^{(b)}-T_{[n]})}_2 \notag\\
 &=\norm{\Delta_{0,q}}_2
 \le C\sqrt q\,\norm{Y^{(\ell)}}_2
 \le C\sqrt{\frac nb}\frac{(\log\ell)^{\gth}}{\ell^{1/\theta}}.
 \label{eq:block-L2-approx}
\end{align}

\smallskip
\noindent\textbf{Step 2. Variance comparison.}
Using the variance lower bound~\eqref{eq:var-lb}, we get
\begin{align}
 \frac{\norm{\oT_{[n]}-\oS_n^{(b)}}_2}{\sqrt{\var(T_{[n]})}}
 \le C\sqrt{
 \frac{\ell^{\ga-2/\theta}(\log\ell)^{A+2\gth}}{b}}
 =o(1).
 \label{eq:block-error-negligible}
\end{align}
Thus,
\begin{align*}
 \abs{\sqrt{\var(S_n^{(b)})}-\sqrt{\var(T_{[n]})}}
 &=\abs{\norm{\oS_n^{(b)}}_2-\norm{\oT_{[n]}}_2}
 \le \norm{\oS_n^{(b)}-\oT_{[n]}}_2
 =o\left(\sqrt{\var(T_{[n]})}\right),
\end{align*}
and consequently
\begin{align}
 \var(S_n^{(b)})=\var(T_{[n]})(1+o(1)).
 \label{eq:block-var-asymp}
\end{align}

\smallskip
\noindent\textbf{Step 3. Lyapunov CLT for the block sum.}
We now check Lyapunov's condition. Write
\begin{align*}
 Q_i^{(b)}:=T^{(\ell)}_{[x_{i-1},x_i]}(x_{i-1},x_i),
 \quad 1\le i\le q-1,
 \qquad
 R_n^{(b)}:=T^{(\ell)}_{[x_{q-1},x_q]}(x_{q-1},x_q),
\end{align*}
for the $q-1$ full blocks and the last (longer) block, and define
\begin{align*}
 X_{i,n}:=Q_i^{(b)}-\E Q_i^{(b)},
 \qquad 1\le i\le q-1,
 \qquad
 X_{q,n}:=R_n^{(b)}-\E R_n^{(b)}.
\end{align*}
By Lemma~\ref{lem:block-cut}, $Q_1^{(b)},\ldots,Q_{q-1}^{(b)},R_n^{(b)}$ are independent, $Q_i^{(b)}\equald T_{[b]}$ and $R_n^{(b)}\equald T_{[b+r]}$.
By~\eqref{eq:block-cut-sum},
\begin{align*}
 \oS_n^{(b)}=\sum_{i=1}^q X_{i,n},
 \qquad
 s_n^2:=\var(S_n^{(b)})=\sum_{i=1}^q \E X_{i,n}^2.
\end{align*}
For all sufficiently large $n$, $q\ge2$ and $b+r\in[b,2b)$. Hence
Lemma~\ref{lem:block-moments} yields
\begin{align*}
 \sum_{i=1}^q \E\abs{X_{i,n}}^4
 &=(q-1)\E\abs{T_{[b]}-\E T_{[b]}}^4+
 \E\abs{T_{[b+r]}-\E T_{[b+r]}}^4\\
 &\le C\{(q-1)b^2+(b+r)^2\}
 \frac{(\log\ell)^{4\gth}}{\ell^{2+4/\theta}}
 \le Cnb\frac{(\log\ell)^{4\gth}}{\ell^{2+4/\theta}}.
\end{align*}
On the other hand,~\eqref{eq:block-var-asymp} and~\eqref{eq:var-lb} imply
\begin{align*}
 s_n^2\ge \frac c2\frac{n}{\ell^\ga(\log\ell)^A}
\end{align*}
for all large $n$. Therefore
\begin{align*}
 \frac1{s_n^4}\sum_{i=1}^q\E\abs{X_{i,n}}^4
 &\le C
 \frac{nb(\log\ell)^{4\gth}/\ell^{2+4/\theta}}
 {n^2/\{\ell^{2\ga}(\log\ell)^{2A}\}}
 =C\frac{b\ell^{2\ga-2-4/\theta}(\log\ell)^{2A+4\gth}}{n}
 =o(1),
\end{align*}
by the second condition in~\eqref{eq:b-conditions}. Lyapunov's theorem for
triangular arrays gives
\begin{align}
 {\oS_n^{(b)}}/{\sqrt{\var(S_n^{(b)})}}
 \Rightarrow\N(0,1).
 \label{eq:block-sum-clt}
\end{align}

\smallskip
\noindent\textbf{Step 4. Transfer from $S_n^{(b)}$ to $T_{[n]}$.}
Finally,~\eqref{eq:block-error-negligible} gives
\begin{align*}
 \frac{\oT_{[n]}-\oS_n^{(b)}}{\sqrt{\var(T_{[n]})}}\to0
\end{align*}
in $L^2$, hence in probability. Combining this with
\eqref{eq:block-var-asymp} and~\eqref{eq:block-sum-clt},
\begin{align*}
 \frac{\oT_{[n]}}{\sqrt{\var(T_{[n]})}}
 =
 \frac{\oS_n^{(b)}}{\sqrt{\var(S_n^{(b)})}}
 \sqrt{\frac{\var(S_n^{(b)})}{\var(T_{[n]})}}
 +\frac{\oT_{[n]}-\oS_n^{(b)}}{\sqrt{\var(T_{[n]})}}
 \Rightarrow\N(0,1).
\end{align*}
This proves the theorem.
\end{proof}

%%%%%%%%%%%%%%%%%%%%%%%%%%%%%%%%%%%%%%%%%%%%%%%%%%%
\noindent\textbf{Completion of the proof of Theorem~\ref{thm:meso-main}~\eqref{meso-clt}}.
We now apply Theorem~\ref{thm:clt-from-var-lb} to the finite-segment passage time $T_{[m]}$ for a segment length $m$ with $m\asymp n$. Setting $\ga=2+2/\theta$ and $A=0$, and using the lower bound from Proposition~\ref{prop:var-lb}, the block-size conditions~\eqref{eq:b-conditions} reduce to
\begin{align*}
 \ell^2(\log\ell)^{2\gth}\ll b
 \ll \frac{m}{\ell^2(\log\ell)^{4\gth}}.
\end{align*}
Such a sequence $b=b_{m,\ell}$ exists precisely when $\ell^4(\log\ell)^{6\gth}\ll m$; since $\ell\ll n^{1/4}/(\log n)^{3\gth/2}$ and $m\asymp n$, this holds. For this choice of $b$, Theorem~\ref{thm:clt-from-var-lb} yields the Gaussian limit
\begin{align}\label{eq:segment-clt}
 \frac{T_{[m]}-\E T_{[m]}}{\sqrt{\var(T_{[m]})}}
 \Rightarrow\N(0,1).
\end{align}

Applying~\eqref{eq:segment-clt} with $m=\sfu_n$ establishes the CLT for the short-arc passage time $T_{[\sfu_n]}$. Completing the proof of the cycle CLT stated in Theorem~\ref{thm:meso-main}~\eqref{meso-clt} additionally requires showing that the long arc contributes negligibly, \ie\ that its expected passage time is strictly larger than that of the short arc by an amount that dominates the fluctuations. Since this uses the first-order mean asymptotics, the proof is deferred to the end of Section~\ref{sec:cmj-mean}.

%%%%%%%%%%%%%%%%%%%%%%%%%%%%%%%%%%%%%%%%%%%%%%%%%%%%%%%%%%%%
\section{Mesoscopic mean analysis: proof of Theorem~\ref{thm:meso-main}~\eqref{meso-mean}}\label{sec:cmj-mean}

This section establishes the mean part of Theorem~\ref{thm:meso-main} in the range $\ell\log\ell \ll n$. While the fluctuation analysis in Section~\ref{sec:meso-fluct} relied on deterministic block cutting and a global Dijkstra perturbation, our focus here is on identifying the first-order time constant. In this regime, the local exploration behaves like a spatial CMJ branching random walk, whose front speed $\vstar$ naturally determines the mesoscopic time constant. 
The argument is divided into four main parts.

\begin{enumerate}[M1.]
 \item \emph{Limiting front speed.}
Proposition~\ref{prop:cmj-speed} identifies $\vstar$ as the right-front speed of the limiting spatial CMJ branching random walk and gives the corresponding inverse hitting-time asymptotic.
\item \emph{From the limiting CMJ process to the spread-out line.}
By restricting the limiting CMJ process to a finite set of directions, Lemma~\ref{lem:finite-direc-one-side-speed} produces a fast path moving along the $\ell$-spread-out tree. Proposition~\ref{prop:auxiliary-one-sided-line-crossing} then embeds this path onto the physical line by coupling their edge weights before the optimal path is identified.

\item\emph{Short windows and chaining.}
Proposition~\ref{prop:finite-graph-window-tools} collects the short-window mean estimate, showing that the passage time across a window of spatial length $h_\ell\ell$ with $h_\ell = o(\log \ell)$ scales precisely as $h_\ell/\vstar$ on the $\fa_\ell$ scale, and matches the ideal speed $\vstar$. To determine the total passage time over a macroscopic distance, we chain these short crossings together. The crucial geometric point is that we stop the exploration exactly upon first entrance into the next window; by the spatial Markov property the unrevealed edges ahead of the entrance point remain untouched (``fresh''), so the next crossing restarts independently. This precise local control bypasses the coarser global flooding bound and removes the extra logarithmic factor when $\theta<1$.

\item\emph{Lower bound and transferring to the cycle.}
To establish the lower bound, we compare the actual graph to an idealized tree model that ignores path overlaps. The physical intuition is simple: the tree provides an overly optimistic scenario, so if a remarkably fast path were to exist on the actual graph, an impossibly fast path would have to exist in the tree. Since the latter has negligible probability, the actual passage time cannot be faster than our estimate. Finally, since our original model is defined on a ring, we transfer this segment result to the fixed-target cycle passage time.
\end{enumerate}

All normalizations and limiting CMJ objects used below, including the normalization $\fa_\ell$ in~\eqref{eq:aell-def}, the offspring measure $\mu_\theta$ in~\eqref{eq:limiting-mu}, and the speed $\vstar$ in~\eqref{def:vstar}, are formally defined in Section~\ref{sec:local-cmj}.

%%%%%%%%%%%%%%%%%%%%%%%%%%%%%%%%%%%%%%%%%%%%%%%%%%%
\subsection{CMJ branching-random-walk estimates}\label{ssec:cmj-estimates}

We work with the spatially marked CMJ process defined in Section~\ref{sec:local-cmj}. If $v=(v_1,\ldots,v_k)\in\dN^k$, write $|v|=k$ and set
\begin{align*}
 v|_0:=\varnothing,
 \qquad
 v|_j:=(v_1,\ldots,v_j),
 \quad 1\le j\le k.
\end{align*}
For $v\ne\varnothing$, write $v^-$ for its parent. For $f\ge0$, define the front hitting time
\begin{align*}
 \tau(f):=\inf\{t\ge0:\sfM_t\ge f\}.
\end{align*}

Recall from~\eqref{def:vstar} that
\begin{align*}
 \gk_\theta(\beta)=\gf(\beta)^{1/\theta},
 \qquad
 \vstar=\vstar_\theta
 =
 \inf_{\beta>0}\frac{\gk_\theta(\beta)}{\beta}.
\end{align*}
The following lemma justifies the phrase ``the unique minimizer'' used in~\eqref{eq:beta_c}.

\begin{lem}[Existence and uniqueness of $\gbc$]\label{lem:betac-unique}
Let $\gf(\beta)=(\sinh\beta)/\beta=\E e^{\beta\xi}$ with $\xi\sim\Unif[-1,1]$ and
$\gk_\theta(\beta)=\gf(\beta)^{1/\theta}$. Then $\beta\mapsto\gk_\theta(\beta)/\beta$
attains its infimum over $(0,\infty)$ at a unique point $\gbc\in(0,\infty)$,
$\vstar=\gk_\theta(\gbc)/\gbc\in(0,\infty)$, and
\begin{align}\label{eq:betac-first-order}
 \gbc\gf'(\gbc)=\theta\gf(\gbc),
 \qquad
 \vstar=\gk_\theta'(\gbc).
\end{align}
\end{lem}

\begin{proof}
The function $\log\gf$ is the logarithmic moment generating function of a
non-degenerate bounded random variable, hence finite, smooth and strictly convex
on $\dR$. Therefore $\gk_\theta=\exp\{(\log\gf)/\theta\}$ is smooth, positive and
strictly convex on $\dR$, with $\gk_\theta(0)=1>0$.
Put $\psi(\beta):=\beta\gk_\theta'(\beta)-\gk_\theta(\beta)$, so that
$(\gk_\theta(\beta)/\beta)'=\psi(\beta)/\beta^2$ for $\beta>0$. Since
$\psi'(\beta)=\beta\gk_\theta''(\beta)>0$ for $\beta>0$, the function $\psi$ is
strictly increasing on $(0,\infty)$; moreover $\psi(0^+)=-\gk_\theta(0)=-1<0$.
As $\beta\to\infty$ we have $\gf(\beta)\sim e^\beta/(2\beta)$, so
$\gk_\theta(\beta)/\beta\to\infty$, while $\gk_\theta(\beta)/\beta\to\infty$ as
$\beta\downarrow0$ because $\gk_\theta(0)=1$. Hence the infimum is attained in
$(0,\infty)$, and at any minimizer $\psi$ vanishes. Strict monotonicity of $\psi$
gives at most one zero, so the minimizer $\gbc$ is unique and
$\gbc\gk_\theta'(\gbc)=\gk_\theta(\gbc)$, \ie\ $\vstar=\gk_\theta(\gbc)/\gbc=\gk_\theta'(\gbc)$.
Finally, $\gk_\theta'=\frac1\theta\gf^{1/\theta-1}\gf'$, so
$\gbc\gk_\theta'(\gbc)=\gk_\theta(\gbc)$ is equivalent to $\gbc\gf'(\gbc)=\theta\gf(\gbc)$.
\end{proof}

For the $\ell$-discrete tree approximation, let
\begin{align}
 \cD_\ell:=\{-\ell,\ldots,-1,1,\ldots,\ell\}
\label{eq:Dell-def}
\end{align}
be the set of possible jump directions from a vertex in the $\ell$-spread-out line graph, so that $|\cD_\ell|=2\ell$. After division by $\ell$, it is the uniform grid in $[-1,1]$ with the origin removed. Birth times in this discrete tree are scaled by the normalization $\fa_\ell$ from~\eqref{eq:aell-def}. By Lemma~\ref{lem:local-offspring-pp}, the corresponding offspring-age point process converges to the CMJ reproduction point process with intensity $\mu_\theta$. (The normalization $\fa_\ell$ is built from the non-backtracking degree $2\ell-1$ while a tree vertex here has $2\ell$ children; since $2\ell/(2\ell-1)\to1$, this discrepancy is asymptotically immaterial and is tracked explicitly wherever it matters, e.g.\ in~\eqref{eq:bar-mu-density}.)

We now define the $\ell$-spread-out tree. Its vertices are finite
direction words
\begin{align*}
	 v=(z_1,\ldots,z_k), \quad z_i\in\cD_\ell.
\end{align*}
 For each particle $u$ and each direction $z\in\cD_\ell$, attach an independent edge weight
\begin{align*}
 \go_{u,z}\equald E_{u,z}^{1/\theta},
 \qquad E_{u,z}\sim\Exp(1),
\end{align*}
and declare the child $uz$ to be born at scaled time increment
$\fa_\ell\go_{u,z}$ and to have scaled displacement $z/\ell$.
Thus
\begin{align*}
\begin{aligned}
 \sft_\ell(\varnothing)&:=0,
 \qquad
 \sft_\ell(uz):=\sft_\ell(u)+\fa_\ell\go_{u,z},\\
 \sfp_\ell(\varnothing)&:=0,
 \qquad
 \sfp_\ell(uz):=\sfp_\ell(u)+{z}/{\ell}.
\end{aligned}
\end{align*}
Finally set
\begin{align*}
 \cN_t^{(\ell)}:=\{v:\sft_\ell(v)\le t\},
 \qquad
 \sfM_t^{(\ell)}:=\max_{v\in\cN_t^{(\ell)}}\sfp_\ell(v).
\end{align*}

%%%%%%%%%%%
\begin{prop}\label{prop:cmj-speed}
For the continuous CMJ branching random walk,
\begin{align}
\frac{\sfM_t}{t}\to \vstar
\text{ a.s.}
\label{eq:cmj-speed}
\end{align}
Consequently,
\begin{align}
\frac{\tau(f)}{f}\to \frac1{\vstar}
\text{ a.s. and in }L^1,
\text{ and }
\E\tau(f)=\frac{f}{\vstar}+o(f).
\label{eq:cmj-inverse-speed}
\end{align}
\end{prop}

\begin{proof}[Proof of Propostion~\ref{prop:cmj-speed}]
The reproduction point process is locally finite because
\[
\mu_\theta(0,t]={t^\theta}/{\Gamma(\theta+1)}<\infty
\text{ for all }t>0.
\]For each spatial tilt $\beta>0$ and time tilt $\gk>0$, the Laplace transform is
\begin{align*}
 \int_{-1}^{1}\int_0^\infty
 e^{\beta x-\gk s}\,\frac{dx}{2}\,\mu_\theta(\dd s)
 = \gf(\beta)\gk^{-\theta} < \infty.
\end{align*}
This ensures the standard integrability conditions for the general branching random walk. Thus, by the spreading speed theorems in~\cite[Corollary 2, Proposition 1]{biggins95}, the front speed is given by
\begin{align*}
 \inf_{\beta>0}\frac{\gk_\theta(\beta)}{\beta} = \vstar,
\end{align*}
where $\gk_\theta(\beta)=\gf(\beta)^{1/\theta}$. This matches exactly our definition in~\eqref{def:vstar}.
The almost sure convergence in~\eqref{eq:cmj-inverse-speed} is a direct consequence of~\eqref{eq:cmj-speed} and the monotonicity of $\sfM_t$. To establish the $L^1$ convergence, it suffices to prove that the family $\{\tau(f)/f\}_{f \ge 1}$ is uniformly integrable. 
Fix $\gd\in(0,1)$ and put $p_\gd:=(1-\gd)/2$. Starting from
the root, at each selected particle follow the first child whose displacement is at least $\gd$. If $V$ is the waiting time for this child, Poisson thinning gives
\begin{align*}
 \pr(V>t)
 =
 \exp(-p_\gd\mu_\theta(0,t])
 =
 \exp\left(
 -\tfrac{1-\gd}{2} \cdot
 \tfrac{t^\theta}{\Gamma(\theta+1)}
 \right),
 \quad t\ge0.
\end{align*}
Thus $V$ has moments of every order. The reproduction processes of distinct particles are independent, so the successive waiting times $V_1,V_2,\ldots$ along this ray are i.i.d.~If $m_f:=\lceil f/\gd\rceil$, then the $m_f$-th selected particle has position at least $f$, and hence $\tau(f) \le \sum_{j=1}^{m_f}V_j$.
Therefore, for $f\ge1$,
\begin{align*}
 \E\tau(f)^2
 &\le
 m_f\E V^2
 +
 m_f(m_f-1)(\E V)^2
 \le
 C f^2.
\end{align*}
Thus
$
 \sup_{f\ge1}\E\left({\tau(f)}/f\right)^2 <\infty
$. This gives uniform integrability and proves the $L^1$ convergence.
\end{proof}

To transfer the continuous CMJ speed limit to our discrete setting, we need to control the number of paths in the $\ell$-spread-out tree. The following lemma provides the required exponential moment bound via a weighted renewal argument. A proof is presented in Appendix~\ref{pf:finite-direction-renewal-bound}.

%%%%%%%%%%%
\begin{lem}[Finite-direction weighted renewal bound]
\label{lem:finite-direction-renewal-bound}
For every fixed $\beta\ge0$ and $\eta>0$, there are constants
$C_{\beta,\eta}<\infty$ and $\ell_0<\infty$ such that, for all
$\ell\ge\ell_0$ and all $t\ge0$,
\begin{align}
\E\sum_{v\in\cN_t^{(\ell)}}e^{\beta \sfp_\ell(v)}
\le
C_{\beta,\eta}
\exp\{(\gk_\theta(\beta)+\eta)t\}.
\label{eq:finite-tree-exp-moment}
\end{align}
Here $\gf(0)=1$ and hence $\gk_\theta(0)=1$.
\end{lem}

By Lemma~\ref{lem:finite-direction-renewal-bound}, the expected population of the $\ell$-spread-out tree up to time $t_\ell^\star = o(\log \ell)$ is subpolynomial in $\ell$. This permits a high-probability coupling between the limiting CMJ tree and the $\ell$-spread-out tree via a union bound over the births. The following lemma establishes this coupling: it exactly preserves unmarked genealogies and birth times, uniformly controls the spatial discretization error, and keeps the spatial positions bounded away from the crossing boundaries.

\begin{lem}[CMJ to $\ell$-spread-out coupling]
\label{lem:cmj-ell-tree-coupling}
Let $h_\ell\to\infty$, $h_\ell=o(\log\ell)$, and $t_\ell^\star=O(h_\ell)$.
Fix $0<\gd<1/2$, and set
\begin{align*}
 N_\ell:=\ell^\gd,
 \qquad
 r_\ell:={N_\ell}/{\ell}.
\end{align*}
Then the limiting CMJ tree and the $\ell$-spread-out tree can be coupled up to
time $t_\ell^\star$ so that there is a high probability event $\cE_\ell$ satisfying the following properties.
On $\cE_\ell$, there is an injective ancestry-preserving map
\begin{align*}
 \Phi_\ell:\cN_{t_\ell^\star}\to \cN_{t_\ell^\star}^{(\ell)}
\end{align*}
such that $\Phi_\ell(\varnothing)=\varnothing$ and, for every $u\in\cN_{t_\ell^\star}$,
$
 \sft_\ell(\Phi_\ell(u))=\sft(u),
$
and 
$
 |\sfp_\ell(\Phi_\ell(u))-\sfp(u)|\le r_\ell .
$
Moreover, on $\cE_\ell$,
\begin{align*}
 \operatorname{dist}
 \left(
 \sfp(u),\{0,h_\ell,h_\ell+1\}
 \right)>3r_\ell
\end{align*}
for every non-root $u\in\cN_{t_\ell^\star}$.
\end{lem}
A proof is presented in Appendix~\ref{pf:cmj-ell-tree-coupling}.
We now apply the local coupling to establish a one-sided crossing estimate. We first prove this speed bound for the limiting CMJ branching random walk by embedding a supercritical Galton--Watson tree of fast-advancing lineages. Lemma~\ref{lem:cmj-ell-tree-coupling} then maps this continuous path to the $\ell$-spread-out tree, preserving the necessary spatial bounds.

\begin{lem}\label{lem:finite-direc-one-side-speed}
Let
\begin{align*}
 \sfM_t^{(\ell),+}
 :=
 \max\left\{
 \sfp_\ell(v):
 v\in\cN_t^{(\ell)},\
 \sfp_\ell(v|_j)\ge0 \text{ for all }0\le j\le |v|
 \right\}.
\end{align*}
If $h_\ell\to\infty$ and $h_\ell=o(\log\ell)$, then, for every $\eps>0$,
\begin{align*}
 \pr\left(
 \sfM_{(1/\vstar+\eps)h_\ell}^{(\ell),+}\ge h_\ell
 \right)
 \to 1.
\end{align*}
Moreover, the hitting particle may be chosen as the first particle on its ancestral line entering
$[h_\ell,h_\ell+1]$.
\end{lem}

The preceding lemma establishes a bound for the unrestricted tree, a proof of which is presented in Appendix~\ref{pf:finite-direc-one-side-speed}. To use it for FPP on the actual $\ell$-spread-out line, we must embed the successful tree lineage into the line geometry. Since the explored tree population up to time $O(h_\ell)$ is subpolynomial in $\ell$, spatial collisions are asymptotically negligible.
On this high-probability collision-free event, the standard covering-tree coupling exactly matches the tree weights to distinct edge weights on the line, embedding the lineage as a valid simple path.

\begin{prop}[Auxiliary one-sided line crossing]
\label{prop:auxiliary-one-sided-line-crossing}
Let $h_\ell\to\infty$ and $h_\ell=o(\log\ell)$, and put
\begin{align*}
 H_\ell:=\lfloor h_\ell\ell\rfloor .
\end{align*}
For every $\eps>0$,
\begin{align}\label{eq:aux-one-sided-line-crossing}
 \pr\left(
 \fa_\ell
 T^{(\ell)}_{[0,H_\ell+\ell]}
 \bigl(0,[H_\ell,H_\ell+\ell]\bigr)
 \le
 \left(\frac1{\vstar}+\eps\right)h_\ell
 \right)
 \to 1.
\end{align}
\end{prop}
A proof is presented in Appendix~\ref{pf:auxiliary-one-sided-line-crossing}. Next, we establish the finite-window estimates necessary to transfer the CMJ front speed to the $\ell$-spread-out line. By analyzing the passage time across a window of size $O(h_\ell \ell)=o(n)$, we construct the fundamental building blocks required for the upper bound. For this purpose, we define the auxiliary scale
\begin{align}\label{eq:h-ell-def}
h_\ell:=\left\lfloor(\log\ell)^{1/3}\right\rfloor,
\end{align}
so that $\ll h_\ell\ll \log\ell$.

\begin{prop}[$\ell$-spread-out graph window tools]\label{prop:finite-graph-window-tools}
The following estimates hold.

\emph{(i) Greedy bounds.} For every fixed $p\ge1$, there is $C_p<\infty$ such
that, for all $\ell\ge2$ and all integers $r\ge0$,
\begin{align}\label{eq:greedy-upper-cmj}
\left\|
 \fa_\ell
 T_{[0,r+\ell]}(0,[r,r+\ell])
\right\|_p
&\le
C_p\left(\frac r\ell+1\right),
\end{align}
and, for all integers $r\ge\ell$,
\begin{align}\label{eq:greedy-endpoint-cmj}
\left\|
 \fa_\ell
 T_{[0,r]}(0,r)
\right\|_p
&\le
C_p\left(\frac r\ell+\log\ell+1\right).
\end{align}

\emph{(ii) One-sided short-window mean.} With $h_\ell$ as in
\eqref{eq:h-ell-def},
\begin{align}\label{eq:window-mean-upper}
 \fa_\ell
 \E \Big( T_{[0,\lfloor h_\ell\ell\rfloor+\ell]}
 \left(0,[\lfloor h_\ell\ell\rfloor,\lfloor h_\ell\ell\rfloor+\ell]\right) \Big)
 \le \frac{h_\ell}{\vstar}+o(h_\ell).
\end{align}

\emph{(iii) Freshness after a stopped window crossing.} Fix $x,H\in\dZ$ with $H\ge1$ and set
\begin{align*}
 I_x:=[x+H,x+H+\ell]\cap\dZ .
\end{align*}
Let $\pi^\star=\pi^\star(x,I_x)$ be the shortest path from $x$ to $I_x$ strictly inside $[x, x+H+\ell]$, stopped upon its first entrance at a terminal point $Y$. Then, conditionally on $\pi^\star$, its passage time, and $Y=y$,
\begin{align*}
 \bigl\{\go_{\{u,v\}}:\ 1\le |u-v|\le \ell,\ u\ge y,\ v\ge y\bigr\}
\end{align*}
are fresh, \ie~conditionally independent and identically distributed as $\Exp(1)^{1/\theta}$.
\end{prop}
A proof is given in Appendix~\ref{pf:finite-graph-window-tools}. Now, we are ready to prove the asymptotic mean of the first-passage time $T_{[n]}$ result presented in Theorem~\ref{thm:meso-main}~\eqref{meso-mean}.

%%%%%%%%%%%%%%%%%%%%%%%%%%%%%%%%%%%%%%%%%%%%%%%%%%%
\subsection{Proof of Theorem~\ref{thm:meso-main}~\eqref{meso-mean}}

Recall that $T_{[n]} = T_{[0,n]}(0,n)$ is restricted to the finite segment $[0,n]$, and $\ell \to \infty$ and $\ell \log \ell = o(n)$. Since $\ell\log\ell=o(n)$, we have $n/\ell\to\infty$ and $\log\ell=o(n/\ell)$. We first prove the lower bound, upper bound and convergence in probability for $T_{[n]}$ separately, and then transfer the result to the cycle case.

\smallskip
\noindent\textbf{Lower bound.}
Fix $\eps\in(0,1/\vstar)$ and set
\[
 t_n^-:=\left(\frac1{\vstar}-\eps\right)\frac n\ell .
\]
On the event $\{\fa_\ell T_{[n]}\le t_n^-\}$, there exists a simple path from $0$ to $n$ inside $[0,n]$ with scaled cost at most $t_n^-$. Every such path is determined by its sequence of increments $(z_1,\ldots,z_k)\in\cD_\ell^k$ and is therefore identified with a unique vertex $v=(z_1,\ldots,z_k)$ of the $\ell$-spread-out tree, with $\sfp_\ell(v)=n/\ell$. Since the path is simple, its $k$ edges are distinct, so their weights are i.i.d.~copies of $\go$; consequently the scaled cost of the path has exactly the law of the birth time $\sft_\ell(v)$ of the corresponding tree vertex. (The two families of weights are \emph{not} coupled — different paths in the line share edges, whereas different tree lineages do not — but only the one-dimensional marginals enter the union bound below.) Hence, by a union bound over tree vertices,
\begin{align*}
\pr\left(\fa_\ell T_{[n]}\le t_n^-\right)
&\le
\E\sum_{v\in\cN_{t_n^-}^{(\ell)}}
\ind\left\{\sfp_\ell(v)\ge \frac n\ell\right\}.
\end{align*}
By Markov's inequality and Lemma~\ref{lem:finite-direction-renewal-bound}, for
every $\beta>0$ and $\eta>0$,
\begin{align*}
\E\sum_{v\in\cN_{t_n^-}^{(\ell)}}
\ind\left\{\sfp_\ell(v)\ge \frac n\ell\right\}
&\le
e^{-\beta n/\ell}
\E\sum_{v\in\cN_{t_n^-}^{(\ell)}}e^{\beta \sfp_\ell(v)}
\le
C_{\beta,\eta}
\exp\left\{
(\gk_\theta(\beta)+\eta)t_n^-
-\beta\frac n\ell
\right\}.
\end{align*}
Choose $\beta=\gbc$ and then choose $\eta>0$ small enough.
Since
$
 \gk_\theta(\gbc) = \gbc\vstar,
$
there is $c_\eps>0$ such that, for all large $\ell$,
\begin{align}
\pr\left(\fa_\ell T_{[n]}\le t_n^-\right)
\le
C e^{-c_\eps n/\ell}
=o(1).
\label{eq:global-meso-lower-tail}
\end{align}
Therefore
\begin{align*}
\fa_\ell\E T_{[n]}
&\ge
t_n^-\,
\pr\left(\fa_\ell T_{[n]}> t_n^-\right)
\ge
\left(\frac1{\vstar}-\eps\right)\frac n\ell(1-o(1)).
\end{align*}
Letting $\eps\downarrow0$ gives
\[
 \liminf_{n\to\infty}
 \frac{\fa_\ell\ell}{n}\E T_{[n]}
 \ge
 \frac1{\vstar}.
\]

\smallskip
\noindent\textbf{Upper bound.}
Let
\[
 H_\ell:=\lfloor h_\ell\ell\rfloor,
 \qquad
 J_\ell
 :=
 \left\lfloor\frac{n-\ell}{H_\ell+\ell}\right\rfloor .
\]
Since $\ell\log\ell=o(n)$ and $h_\ell=(\log\ell)^{1/3}+O(1)$, we have $h_\ell\to\infty$ and $h_\ell=o(n/\ell)$.

We construct the global path by concatenating $J_\ell$ independent short-window crossings. Initialize $X_0=0$. For each $0\le j < J_\ell$, conditional on the current position $X_j$, the path explores the interval $[X_j, X_j+H_\ell+\ell]$ until its first entrance into the target window $[X_j+H_\ell, X_j+H_\ell+\ell]$. We designate this initial entrance vertex as $X_{j+1}$, and we let $W_{j+1}$ denote the passage time of this $(j+1)$-st segment. Since $J_\ell(H_\ell+\ell)\le n-\ell$, all these intermediate crossings are strictly contained within $[0,n]$.

Let $\cG_j$ be the sigma-field generated by the first $j$ segments, including their terminal points and passage times. By Proposition~\ref{prop:finite-graph-window-tools} \textup{(iii)}, conditionally on $\cG_j$, the unexplored environment to the right of $X_j$ is entirely fresh. Hence, by translation invariance and the bound~\eqref{eq:window-mean-upper}, there exists a deterministic sequence $\gd_\ell\downarrow0$ such that, for every $0\le j<J_\ell$,
\begin{align}
\fa_\ell
\E\left[
 W_{j+1}\mid\cG_j
\right]
\le
\frac{h_\ell}{\vstar}+\gd_\ell h_\ell .
\label{eq:stopped-window-conditional-mean}
\end{align}
Summing over all $J_\ell$ segments yields
\begin{align*}
\fa_\ell
\E\left[\sum_{j=1}^{J_\ell}W_j\right]
\le
J_\ell\left(\frac{h_\ell}{\vstar}+\gd_\ell h_\ell\right)
\le
\frac{n}{\ell\vstar}+o(n/\ell),
\end{align*}
where we used $J_\ell h_\ell\le n/\ell$ and $\gd_\ell\to0$.

It remains to bound the passage time across the final segment of length $n-X_{J_\ell}$.
Since $J_\ell H_\ell\le X_{J_\ell}\le J_\ell(H_\ell+\ell)$, the remaining distance satisfies
\begin{align}
 \frac{n-X_{J_\ell}}{\ell}
 \le
 \frac{n-J_\ell H_\ell}{\ell}
 \le
 C\left(
 \frac{n}{h_\ell\ell}+h_\ell+1
 \right).
\label{eq:remaining-distance-global-meso}
\end{align}
Conditionally on $\cG_{J_\ell}$, the edge weights to the right of $X_{J_\ell}$ are fresh. Applying the endpoint bound~\eqref{eq:greedy-endpoint-cmj} to this terminal interval gives
\begin{align*}
\fa_\ell
\E\left[
 T_{[X_{J_\ell},n]}(X_{J_\ell},n)
 \,\middle|\,\cG_{J_\ell}
\right]
\le
C\left(\frac{n-X_{J_\ell}}{\ell}+\log\ell+1\right)
\le
C\left(
 \frac{n}{h_\ell\ell}+h_\ell+\log\ell+1
\right).
\end{align*}
The right-hand side is $o(n/\ell)$ as
$
 h_\ell\to\infty,
 h_\ell=o(n/\ell),
 \log\ell=o(n/\ell).
$
Taking expectations, we conclude that the terminal connection cost is asymptotically negligible:
\begin{align}
\fa_\ell
\E T_{[X_{J_\ell},n]}(X_{J_\ell},n)
=o(n/\ell).
\label{eq:terminal-global-meso-small}
\end{align}

Bounding the total passage time by the sum of the intermediate crossings and the final segment,
\[
 T_{[n]}
 \le
 \sum_{j=1}^{J_\ell}W_j
 +
 T_{[X_{J_\ell},n]}(X_{J_\ell},n).
\]
Consequently,
\[
 \limsup_{n\to\infty}
 \frac{\fa_\ell\ell}{n}\E T_{[n]}
 \le
 \frac1{\vstar}.
\]
Combining this with the previously established lower bound, we obtain
\[
 \E T_{[n]}
 =
 \frac{n}{\fa_\ell\ell\vstar}(1+o(1)).
\]

\smallskip
\noindent\textbf{Convergence in probability.}
 By the greedy $L^2$ estimate in Proposition~\ref{prop:finite-graph-window-tools}~\textup{(i)},
\begin{align*}
 \norm{\fa_\ell W_j}_2\le Ch_\ell.
\end{align*}
By Proposition~\ref{prop:finite-graph-window-tools}~\textup{(iii)} and translation
invariance of the weight field, conditionally on $\cG_j$ the pair consisting of the
$(j+1)$-st window crossing time and its terminal displacement has a law that does not
depend on $\cG_j$; hence the successive pairs are i.i.d., and thus
\begin{align*}
 \var\left(
 \fa_\ell\sum_{j=1}^{J_\ell}W_j
 \right)
 &\le
 C J_\ell h_\ell^2
 \le
 C\frac{n h_\ell}{\ell}.
\end{align*}
Since $\ell h_\ell=o(n)$,
\begin{align*}
 \frac{\fa_\ell\sum_{j=1}^{J_\ell}W_j
 -\E\bigl[\fa_\ell\sum_{j=1}^{J_\ell}W_j\bigr]}
 {n/\ell}
 \stackrel{\pr}{\to}0.
\end{align*}
Moreover,~\eqref{eq:terminal-global-meso-small} and Markov's inequality give
\begin{align*}
 \frac{\fa_\ell T_{[X_{J_\ell},n]}(X_{J_\ell},n)}{n/\ell}
 \stackrel{\pr}{\to}0.
\end{align*}
By the minimality of the first-passage time, for every $\eps>0$,
\begin{align*}
 \pr\left(
 \frac{\fa_\ell\ell}{n}T_{[n]}
 >
 \frac1{\vstar}+\eps
 \right)
 \to0.
\end{align*}
Together with the lower-tail estimate~\eqref{eq:global-meso-lower-tail}, this proves
\begin{align}\label{eq:meso-seg-lln-proof}
 \frac{\fa_\ell\ell}{n}T_{[n]}
 \stackrel{\pr}{\to}
 \frac1{\vstar}.
\end{align}

Finally, we record that the argument above used only $\ell\to\infty$ and $\ell\log\ell=o(n)$, where $n$ was the \emph{length of the segment}. Hence, for any sequence of segment lengths $m=m_n\to\infty$ with $\ell\log\ell=o(m)$,
\begin{align}\label{eq:meso-seg-lln-general}
 \frac{\fa_\ell\ell}{m}T_{[m]}
 \stackrel{\pr}{\to}
 \frac1{\vstar},
 \qquad
 \E T_{[m]}
 =
 \frac{m}{\fa_\ell\ell\vstar}(1+o(1)).
\end{align}
We will use~\eqref{eq:meso-seg-lln-general} with $m=\sfu_n$ and with $m=n-\sfu_n$.\qed

\begin{proof}[\textbf{Transfer from line to cycle in Theorem~\ref{thm:meso-main}}]
We first transfer the mean result. Since the $\ell$-spread-out graph induced on the arc $[0,\sfu_n]=\{0,1,\ldots,\sfu_n\}$ is a subgraph of $\dT_n^{(\ell)}$ (recall $\sfu_n\in\dN$ and $\sfu<1/2$, so $\sfu_n\le n/2$ and the arc carries no wrap-around edges),
\begin{align}\label{eq:cycle-seg-upper}
	T_n\le T_{n,\mathrm{short}}\equald T_{[\sfu_n]},
\end{align}
with $T_{n,\mathrm{short}}$ as in~\eqref{eq:two-arc-def}.
For the reverse lower-tail estimate, fix $\eps\in(0,1/\vstar)$ and let
\begin{align*}
 t_{n,\sfu}^-:=
 \left(\frac1{\vstar}-\eps\right)\frac{\sfu_n}{\ell}.
\end{align*}
Every simple cycle path from $0$ to $\sfu_n$ has a lift to $\dZ$ whose increments belong to $\cD_\ell=\{-\ell,\ldots,-1,1,\ldots,\ell\}$ and whose endpoint is $\sfu_n+kn$ for some $k\in\dZ$. Since $\sfu\in(0,1/2)$,
\begin{align*}
 |\sfu_n+kn|\ge \sfu_n
 \text{ for every }k\in\dZ
\end{align*}
for all large $n$. The same tree-overcounting argument as in the proof of~\eqref{eq:global-meso-lower-tail}, now applied to both spatial signs, gives
\begin{align}
 \pr\left(
 \fa_\ell T_n
 \le t_{n,\sfu}^-
 \right)
 &\le
 \E\sum_{v\in\cN_{t_{n,\sfu}^-}^{(\ell)}}
 \left[
 \ind\left\{\sfp_\ell(v)\ge\frac{\sfu_n}{\ell}\right\}
 +
 \ind\left\{\sfp_\ell(v)\le-\frac{\sfu_n}{\ell}\right\}
 \right]
 \notag\\
 &\le
 C\exp\left\{-c_\eps\frac{\sfu_n}{\ell}\right\}
 =o(1).
 \label{eq:cycle-meso-lower-tail}
\end{align}
Combining~\eqref{eq:cycle-seg-upper},~\eqref{eq:meso-seg-lln-general} with $m=\sfu_n$, and~\eqref{eq:cycle-meso-lower-tail} proves the cycle law of large numbers in~\eqref{eq:meso-mean}. The same lower tail gives
\begin{align*}
 \liminf_{n\to\infty}
 \frac{\fa_\ell\ell}{\sfu_n}\E T_n
 \ge\frac1{\vstar},
\end{align*}
whereas~\eqref{eq:cycle-seg-upper} and~\eqref{eq:meso-seg-lln-general} give the matching upper bound. This proves the cycle mean assertion in~\eqref{eq:meso-mean}.

We next transfer the CLT. Recall $T_{n,\mathrm{short}}\equald T_{[\sfu_n]}$ and $T_{n,\mathrm{long}}\equald T_{[n-\sfu_n]}$ from~\eqref{eq:two-arc-def}; they are independent, since they depend on disjoint edge sets. By cutting the cycle at the interfaces $0$ and $\sfu_n$, the same local rerouting argument as in Lemma~\ref{lem:block-cut} gives the deterministic bound
\begin{align*}
 0 \le \min\{T_{n,\mathrm{short}}, T_{n,\mathrm{long}}\} - T_n \le Z_0^{(\ell)} + Z_{\sfu_n}^{(\ell)},
\end{align*}
with $Z_x^{(\ell)}$ the local detour cost~\eqref{eq:Z-detour-def}. By Lemma~\ref{lem:complete-flooding}, $\norm{Z_0^{(\ell)} + Z_{\sfu_n}^{(\ell)}}_2 \le C(\log\ell)^{\gth}/\ell^{1/\theta}$.

Let $X_n := T_{n,\mathrm{long}} - T_{n,\mathrm{short}}$ and let $\bar{X}_n := X_n - \E X_n$ be its centred version. Using $\fa_\ell \asymp \ell^{1/\theta}$ and~\eqref{eq:meso-seg-lln-general} applied with $m=\sfu_n$ and with $m=n-\sfu_n$, we have
\begin{align*}
 \E X_n = \frac{(n-2\sfu_n)}{\fa_\ell\ell\vstar}(1+o(1))
 = \frac{(1-2\sfu)n}{\fa_\ell\ell\vstar}(1+o(1)) \ge c\frac{n}{\ell^{1+1/\theta}},
\end{align*}
where $c>0$ because $\sfu<1/2$.
Let $R_n := (T_{n,\mathrm{short}} - T_{n,\mathrm{long}})_+ = (-X_n)_+$. Since $R_n > 0$ requires $-\bar{X}_n > \E X_n$, we have $R_n \le |\bar{X}_n| \ind\{|\bar{X}_n| > \E X_n\}$, so that, by Cauchy--Schwarz and Markov,
$\norm{R_n}_2^2\le\bigl(\E\bar X_n^4\bigr)^{1/2}\pr(|\bar X_n|>\E X_n)^{1/2}\le \E\bar X_n^4/(\E X_n)^2$.
Applying the fourth-moment bound~\eqref{eq:block-fourth-bound} to both arcs therefore yields
\begin{align*}
 \norm{R_n}_2
 \le \frac{\left( \E|\bar X_n|^4 \right)^{1/2}}{\E X_n}
 \le C\frac{(\log\ell)^{2\gth}}{\ell^{1/\theta}}
 \to 0.
\end{align*}
Since $0\le T_{n,\mathrm{short}} - T_n \le R_n + Z_0^{(\ell)} + Z_{\sfu_n}^{(\ell)}$, combining the above bounds gives
\begin{align*}
 \norm{ \left(T_n-\E T_n\right) - \left(T_{[\sfu_n]}-\E T_{[\sfu_n]}\right) }_2
 \le C\frac{(\log\ell)^{2\gth}}{\ell^{1/\theta}}.
\end{align*}
Under the condition $\ell^4(\log\ell)^{6\gth} \ll n$, the segment variance lower bound $\var(T_{[\sfu_n]})^{1/2} \ge c \sqrt{n}/\ell^{1+1/\theta}$ ensures that this $L^2$-difference is $o\bigl(\sqrt{\var(T_{[\sfu_n]})}\bigr)$; indeed the ratio is $C\ell(\log\ell)^{2\gth}/\sqrt n\to0$. Slutsky's theorem and the $L^2$ reverse triangle inequality then transfer the variance equivalence and the CLT~\eqref{eq:segment-clt} from $T_{[\sfu_n]}$ to $T_n$.
\end{proof}
%%%%%%%%%%%%%%%%%%%%%%%%%%%%%%%%%%%%%%%%%%%%%%%%%%%

%%%%%%%%%%%%%%%%%%%%%%%%%%%%%%%%%%%%%%%%%%%%%%%%%%%%%%%%%
\section{Macroscopic exploration and Cox bridge hazards}\label{sec:macro-cox}

\subsection{Setup and proof roadmap}\label{sec:macro-setup}

In this and the next section, we prove Theorem~\ref{thm:Macro}. We first establish the limit for an independent uniform target and then transfer it to the deterministic target $\sfu_n$ using Lemma~\ref{lem:macro-target-uniformity}. The time normalization $\fa_\ell$, the limiting offspring measure $\mu_\theta$, and the CMJ martingale limit $M_\infty$ were defined in Section~\ref{sec:local-cmj}. The proof unfolds in four parts.

\begin{enumerate}[C1.]
 \item\emph{Exact two-source collision reduction.}
We simultaneously grow two Dijkstra explorations, one from the source and one from the target. Proposition~\ref{prop:half} identifies the precise time at which the two clusters collide
$
 \sfS_{\col}=\frac12\fa_\ell T_n^{\unif}.
$
Proposition~\ref{prop:censor} then shows that, strictly before this collision, all potential bridge edges connecting the two clusters can be censored without altering either one-source exploration. This geometric reduction yields two disjoint, growing clusters while leaving the connecting bridge weights completely unrevealed.

\item\emph{Exact Cox representation.}
The unrevealed bridge weights are handled via deferred decisions in the cumulative hazard scale $H(t)=t^\theta$. Proposition~\ref{prop:cox} leverages their conditionally independent residual laws to establish the exact survival identity
\begin{align*}
 \pr(\sfS_{\col}>s)=\E e^{-\gL_s},
\end{align*}
where $\gL_s$ is the total hazard accumulated by all active bridges. This step effectively reduces the passage-time problem to determining the asymptotic behavior of $\gL_s$ at the critical scale $s=\frac12\log n+O(1)$.

\item\emph{CMJ approximation and removal of the spatial geometry.}
Proposition~\ref{prop:cmj-package} couples each censored exploration to an unrestricted finite-$n$ CMJ proposal process, showing that the rejected ``ghost'' particles (which attempt to visit already-occupied sites) are stochastically negligible. This yields the uniform population growth, stable-age, and Fourier decay estimates valid throughout the collision window. However, population sizes alone are insufficient to determine $\gL_s$, as bridges only form between valid spatial pairs. Proposition~\ref{prop:pair} uses the Fourier decay to prove that the two clusters become uniformly mixed in space. Consequently, the actual bridge density and the endpoint ages of the bridges are perfectly captured by the corresponding non-spatial CMJ limits.

\item\emph{Integrated hazard and theorem assembly.}
Lemma~\ref{lem:integrated} consolidates these CMJ growth, spatial mixing, and bridge-age estimates to show that, for every fixed $y\in\dR$,
\begin{align*}
 \gL_{\frac12\log n+y}
 \Rightarrow
 \theta M_\infty^{(1)}M_\infty^{(2)}e^{2y}.
\end{align*}
Combining this limit with the exact Cox identity from~\emph{(C2)} and the half-distance identity from~\emph{(C1)} directly yields the macroscopic fluctuation limit.
\end{enumerate}

Throughout this section, we assume
\begin{align}\label{eq:macro-ell-assumption}
 \frac{\ell}{n}\to \gl\in(0,1/2).
\end{align}
When $\gl=1/2$, the cycle $\dT_n^{(\ell)}$ is essentially the complete graph $K_n$; we omit this boundary case.

The assumption~\eqref{eq:macro-ell-assumption} is used only when the exact bridge-hazard representation is evaluated asymptotically at the $O(1)$ fluctuation scale.
 However, the deterministic half-distance identity and the Cox representation derived below are non-asymptotic statements that hold for any fixed $n$ as long as $1\le \ell\le n/2$. We shall exploit these identities again in Subsection~\ref{subsec:Macro-mean} to establish the right-crossover first-order law in the wider range $n/\log n\ll\ell\le n/2$.
Set
\begin{align*}
 \cD_n:=\{z\in\dT_n\setminus\{0\}:d_n(z,0)\le \ell\}\subseteq\dT_n.
\end{align*}
This is the set of all possible single-step displacements on the cycle; it is the image of $\cD_\ell$ from~\eqref{eq:Dell-def} under the projection $\dZ\to\dT_n$, and, because $\gl<1/2$, this projection is injective on $\cD_\ell$ for all large $n$. By translation invariance, $\cD_n$ determines both the vertex degree and the spatial increment distribution for the exploration process on the graph of fixed size $n$. To distinguish this pre-limit growth from the limiting Crump--Mode--Jagers process as $n \to \infty$, we refer to it as the \emph{finite-$n$ CMJ exploration}. Since $\gl < 1/2$, for all sufficiently large $n$ we have
\begin{align*}
 |\cD_n|=2\ell,
 \qquad
 \frac{|\cD_n|}{2\gl n}\to1,
 \qquad
 \frac{|\cD_n|}{2\ell-1}\to1.
\end{align*}
The edge-weight distribution follows~\eqref{eq:edge-law}.
Thus, the cumulative hazard and hazard density are
\begin{align*}
 H(t):=-\log\bar F(t)=t^\theta,
 \qquad
 h(t):=H'(t)=\theta t^{\theta-1},
 \quad t>0.
\end{align*}
Let the target $U_n$ be drawn uniformly from $\dT_n \setminus \{0\}$, independently of the edge weights, and define
\begin{align}\label{eq:macro-source-target}
 T_n^{\unif}
 :=T_{\dT_n}^{(\ell)}(0,U_n).
\end{align}
For the deterministic target in Theorem~\ref{thm:Macro}, write
\begin{align*}
 T_n^{\mathrm{fix}}:=T_n
 :=T_{\dT_n}^{(\ell)}(0,\sfu_n).
\end{align*}
As our estimates bound the passage time uniformly over all possible choices of $U_n$, the conditioning on $U_n$ is dropped from the notation.

The limiting reproduction measure is
\begin{align*}
 \mu(\dd s):=\mu_\theta(\dd s)
 =
 \frac{s^{\theta-1}}{\Gamma(\theta)}\,\dd s,
 \quad s>0,
\end{align*}
as in~\eqref{eq:limiting-mu}. We write $W_1,W_2$ for independent copies of the CMJ martingale limit $M_\infty$ defined in~\eqref{eq:Minfty-def}.

\begin{rem}[The assumptions $\ell/n\to\gl$ and $\ell=\lfloor\gl n\rfloor$]
The proof only uses the following facts:
\begin{align*}
 \frac{\ell}{n}\to\gl,
 \qquad
 \frac{|\cD_n|}{2\gl n}\to1,
 \qquad
 \frac{|\cD_n|}{2\ell-1}\to1,
 \qquad
 \frac{2\gl n}{2\ell-1}\to1.
\end{align*}
Thus $\ell=\lfloor\gl n\rfloor$ is only a special case. Theorem~\ref{thm:Macro} is stated under the weaker assumption $\ell/n\to\gl$.
\end{rem}

\begin{rem}[Random and deterministic targets]
Theorem~\ref{thm:Macro} is stated for the deterministic target $\sfu_n$. In the proof below, we first work with the independent uniform target $U_n$, since this removes irrelevant arithmetic dependence on a fixed endpoint.
Lemma~\ref{lem:macro-target-uniformity} shows that the estimates are uniform over the second root and transfers the random-target convergence to $\sfu_n$.
Notably, the limiting distribution $S_\infty$ is independent of the position $\sfu$.
\end{rem}

\begin{rem}[Asymptotic negligibility of fixed-hop shortcuts]
Finite-hop paths are asymptotically negligible at the $(\log n)/\fa_\ell $ leading passage-time scale. For any fixed $m \ge 1$ and $u \ne 0$, there are $O(n^{m-1})$ simple paths of length $m$. Since the sum of $m$ independent weights satisfies $\pr(\go_1 + \dots + \go_m \le t) \le C_m t^{m\theta}$, a union bound limits the probability of any $m$-hop shortcut completing within $(\log n+y)/\fa_\ell$ by
\begin{align*}
C_m n^{m-1} \left(\frac{\log n + |y|}{\fa_\ell}\right)^{m\theta} 
= O\left(\frac{(\log n)^{m\theta}}{n}\right) 
= o(1).
\end{align*}
Thus, macroscopic passage times are strictly driven by the collision of large growing clusters, not by any fixed-hop shortcuts.
\end{rem}

%%%%%%%%%%%%%%%%%%%%%%%%%%%%%%%%%%%%%%%%%%%%%%%%%%%%%%%%%%%%%%%%%%%%%%%%%%%%%%%%%%%%%%%%%%%%%%%%%%%%%%%%%%
\subsection{Two-source collision time and censored exploration}\label{sec:collision}

The first step reduces the point-to-point passage time to a two-source collision problem. We simultaneously grow two exploration clusters: cluster $1$ rooted at the source $0$, and cluster $2$ rooted at the random target $U_n$.

\begin{defn}[Two-source balls and collision time]
For $v\in\dT_n$ and $t\ge0$, define the one-source passage times and
passage-time balls by
\begin{align*}
 \sfT_1(v)&:=T_{\dT_n}^{(\ell)}(0,v),
 &
 \cA_1(t)&:=\{v\in\dT_n:\sfT_1(v)\le t\},\\
 \sfT_2(v)&:=T_{\dT_n}^{(\ell)}(U_n,v),
 &
 \cA_2(t)&:=\{v\in\dT_n:\sfT_2(v)\le t\}.
\end{align*}
The \emph{original-time collision time} is
\begin{align}\label{eq:tau-col}
\tau_{\col}:=\inf\Big\{t\ge0:
&\text{ there exist }u\in\cA_1(t),\ v\in\cA_2(t)
\text{ with }\{u,v\}\in\cE_n^{(\ell)}
\text{ and} \notag\\
&\go_{u,v}
\le (t-\sfT_1(u))+(t-\sfT_2(v))
\Big\}.
\end{align}
Its rescaled version is
\begin{align*}
 \sfS_{\col}:=\fa_\ell\tau_{\col}.
\end{align*}
\end{defn}

For each fixed $n$, let $\Omega_*$ be the full-measure event where all edge weights are strictly positive, all geodesics are unique, and no geodesic midpoint lands exactly on a vertex. Explicitly, for the unique geodesic $0=v_0, \ldots, v_m=U_n$, the partial passage times
\begin{align*}
 a_j:=\sum_{r=0}^{j-1}\go_{v_r,v_{r+1}},
 \quad 0\le j\le m,
\end{align*}
satisfy $a_j \neq T_n^{\unif}/2$ for all $j$. Since the edge weights are independent and continuous on a finite graph, any path ties or exact midpoint collisions correspond to a finite union of lower-dimensional hyperplanes, which has zero Lebesgue measure. Thus $\pr(\Omega_*)=1$, and all subsequent deterministic arguments are implicitly carried out on $\Omega_*$.

With the collision time and the tie-breaking event $\Omega_*$ properly set up, we can now formally connect the two-source exploration to the global passage time $T_n^{\unif}$.

%%%%%%%%%%%%%%%%%%%%%%%%%%%%%%%%%%%%%%%%%%%%%%%%%%%%%%%%%%%%%%%%
\begin{prop}[Half-distance identity]\label{prop:half}
Conditional on the full-measure event $\Omega_*$, the collision times satisfy
\begin{align*}
 \tau_{\col}=\frac{1}{2}T_n^{\unif},
 \text{ and }
 \sfS_{\col}=\frac{1}{2}\fa_\ell T_n^{\unif}.
\end{align*}
In particular, these identities hold almost surely.
\end{prop}

\begin{proof}
Work on $\Omega_*$. If $t < T_n^{\unif}/2$, any edge $\{u,v\}$ satisfying the collision condition~\eqref{eq:tau-col} would yield a path from $0$ to $U_n$ of total weight at most $\sfT_1(u) + \omega_{u,v} + \sfT_2(v) \le 2t < T_n^{\unif}$, which contradicts the definition of $T_n^{\unif}$. Thus, $\tau_{\col} \ge T_n^{\unif}/2$.

Conversely, let $0=v_0, \ldots, v_m=U_n$ be a geodesic. By the midpoint exclusion in $\Omega_*$, there exists a unique edge $\{v_j, v_{j+1}\}$ along this geodesic strictly containing the midpoint $T_n^{\unif}/2$. Setting $u=v_j$ and $v=v_{j+1}$, the subpath optimality guarantees that the passage times from $0$ to $u$ and from $v$ to $U_n$ along this geodesic are exactly $\sfT_1(u)$ and $\sfT_2(v)$, respectively. Consequently, the passage time from $0$ to $v$ along the geodesic is exactly $T_n^{\unif} - \sfT_2(v)$. Hence, 
\[
 \go_{u,v} = \left(\frac{T_n^{\unif}}{2} - \sfT_1(u)\right) + \left(\frac{T_n^{\unif}}{2} - \sfT_2(v)\right), 
 \]
which satisfies the collision condition exactly at $t = T_n^{\unif}/2$. Thus, $\tau_{\col} \le T_n^{\unif}/2$.
\end{proof}

\begin{rem}[Why the half-distance identity is useful]
Proposition~\ref{prop:half} is an exact statement for each fixed finite graph. It removes the global optimization over all source-to-target paths and replaces it with the first time at which two one-source explorations can be joined. The factor $1/2$ in the macroscopic centering enters only through this deterministic identity.
\end{rem}

To compute the exact distribution of the collision time, the unrevealed cross-edges between the two clusters should remain conditionally independent. To retain this independence, we replace the usual simultaneous Dijkstra exploration by the following censored exploration.

\begin{defn}[Censored two-source exploration]
Initially,
\begin{align*}
 \widehat\cA_1(0)=\{0\},
 \quad
 \widehat \sfT_1(0)=0,
 \quad
 \widehat\cA_2(0)=\{U_n\},
 \quad
 \widehat \sfT_2(U_n)=0,
\end{align*}
and every other vertex is unlabelled.

Suppose that $q$ is an exploration time immediately after a vertex addition. For $i\in\{1,2\}$, define the active oriented boundary of label $i$ by
\begin{align*}
 \widehat\partial_i(q)
 :=
 \big\{(a,b):
 a\in\widehat\cA_i(q),\
 b\notin\widehat\cA_1(q)\cup\widehat\cA_2(q),\
 \{a,b\}\in\cE_n^{(\ell)}
 \big\}.
\end{align*}
If an unlabelled vertex remains, the next vertex-addition time is
\begin{align*}
 q^+
 :=
 \min_{i\in\{1,2\}}
 \min_{(a,b)\in\widehat\partial_i(q)}
 \big\{\widehat\sfT_i(a)+\go_{a,b}\big\}.
\end{align*}
On $\Omega_*$, this minimum has a unique minimizing triple $(i,a,b)$.
At time $q^+$, assign label $i$ to $b$ and set
\begin{align*}
 \widehat\sfT_i(b):=q^+,
 \qquad
 \widehat\cA_i(q^+)
 :=\widehat\cA_i(q)\cup\{b\},
 \qquad
 \widehat\cA_{3-i}(q^+)
 :=\widehat\cA_{3-i}(q).
\end{align*}
The labelled sets remain unchanged between successive vertex-addition times.

An edge is called a \emph{bridge} once its endpoints carry different labels. From that time onward, the bridge is censored: its weight is not tested against any further threshold, and the edge is never used for a subsequent vertex addition. The exploration is run without stopping at the first collision. After all vertices are labelled, the process is held fixed.

In rescaled time, define for $s\ge0$
\begin{align*}
 \cC_i(s):=\widehat\cA_i(s/\fa_\ell ),
 \qquad
 N_i(s):=|\cC_i(s)|,
 \quad i=1,2,
\end{align*}
and let the bridge set and its cardinality be
\begin{align*}
	\cB_s:=\big\{\{u,v\}\in\cE_n^{(\ell)}:u\in\cC_1(s),\ v\in\cC_2(s)\big\},
 \qquad
 B_s:=|\cB_s|.
\end{align*}
\end{defn}

The following proposition verifies that this censoring mechanism perfectly couples with the original exploration strictly before the collision time.

%%%%%%%%%%%%%%%%%%%%%%%%%%%%%%%%%%%%%%%%%%%%%%%%%%%%%%%%%%%%%%%%%%%%%%%%
\begin{prop}[Censoring before collision]\label{prop:censor}
Fix $t\ge0$. On $\{\tau_{\col}>t\}$, the following statements hold.
\begin{enumeratei}
\item $\cA_1(t)\cap\cA_2(t)=\eset$.
\item If $u\in\cA_1(t)$, $v\in\cA_2(t)$, and
$\{u,v\}\in\cE_n^{(\ell)}$, then the edge $\{u,v\}$ lies on no geodesic
from $0$ to a vertex of $\cA_1(t)$ and on no geodesic from $U_n$ to a
vertex of $\cA_2(t)$.
\item For $i=1,2$,
\begin{align*}
 \widehat\cA_i(q)=\cA_i(q),
 \qquad 0\le q\le t.
\end{align*}
Consequently, $\widehat \sfT_i(v)=\sfT_i(v)$ for every $v\in\cA_i(t)$.
\end{enumeratei}
\end{prop}

\begin{proof}
If $w\in\cA_1(t)\cap\cA_2(t)$, then
$
 T_n^{\unif}\le \sfT_1(w)+\sfT_2(w)\le2t.
$
By Proposition~\ref{prop:half}, this gives $\tau_{\col}=T_n^{\unif}/2\le t$, a contradiction. This proves (i).

For (ii), suppose an edge $\{u,v\}$ with $u\in\cA_1(t)$ and $v\in\cA_2(t)$ lies on a geodesic from $0$ to some $z\in\cA_1(t)$. Since subpaths of a geodesic are themselves geodesics, we must have $\sfT_1(v) \le \sfT_1(z) \le t$. This implies $v \in \cA_1(t) \cap \cA_2(t)$, which contradicts (i). By symmetry, the same holds for geodesics from $U_n$.

It remains to prove (iii). Every censored path is a path in the original
graph, hence $\widehat\cA_i(q)\subseteq\cA_i(q)$ for $0\le q\le t$. Assume
that the reverse inclusion fails, and let $q_\ast\le t$ be the first
uncensored discovery time at which it fails. Then for some $i$ and $z$,
\begin{align*}
 \sfT_i(z)=q_\ast,
 \qquad
 z\notin\widehat\cA_i(q_\ast).
\end{align*}
Choose an uncensored geodesic from the source of label $i$ to $z$, and write
it as $v_0,v_1,\ldots,v_m=z$, where $v_0=0$ if $i=1$ and $v_0=U_n$ if
$i=2$. Set
\begin{align*}
 b_r:=\sum_{h=0}^{r-1}\go_{v_h,v_{h+1}},
 \qquad 0\le r\le m.
\end{align*}
Let $r$ be the first index such that $v_r\notin\widehat\cA_i(b_r)$. Then $r\ge1$ and $v_{r-1}\in\widehat\cA_i(b_{r-1})$. If $v_r$ had remained unlabelled strictly before time $b_r$, the active edge $(v_{r-1},v_r)$ would have offered a candidate arrival time of $b_r$. Consequently, the censored exploration would have assigned label $i$ to $v_r$ no later than time $b_r$. Since $v_r\notin\widehat\cA_i(b_r)$, it must have already received the opposite label $3-i$ by time $b_r$.

This implies that the edge $\{v_{r-1}, v_r\}$ is a cross-edge connecting $\cA_i(t)$ and $\cA_{3-i}(t)$. Since this edge lies on the minimal path to $z \in \cA_i(t)$, it is part of a geodesic from the source of label $i$ to a vertex of $\cA_i(t)$, which contradicts (ii). 
Thus, $\cA_i(q)\subseteq\widehat\cA_i(q)$ for all $q\le t$. Equality of hitting times follows from the equality of all level sets up to time $t$.
\end{proof}

\begin{rem}[Exactness of the censored exploration prior to collision]
The censored process is introduced strictly to preserve the freshness of bridge residuals. Proposition~\ref{prop:censor} shows that, conditionally on the event that no collision has occurred by time $t$, this modification does not change either one-source ball up to time $t$. Thus, the censoring yields a tractable filtration without modifying the underlying passage-time events.
\end{rem}

\subsection{Bridge residuals and Cox hazard representation}\label{sec:cox}

We now convert the censored bridge process into an exact conditional survival formula. The relevant memoryless variable is $H(\go_e)=\go_e^\theta$. Thus Lemma~\ref{lem:hazard-residual} applies in the cumulative-hazard scale, while the censoring ensures that no bridge residual is tested after its activation.

Let $\cF_t^{\rm orig}$ be the filtration generated by the censored exploration up to original time $t$, including the labelled sets, discovery times, selected-edge weights, and all lower bounds recorded for non-selected edges, but no information about a bridge weight beyond the lower bound recorded at its activation. For the rescaled time, we set 
\begin{align*}
 \cF_s:=\cF_{s/\fa_\ell}^{\rm orig}.
\end{align*}

\begin{defn}[Bridge thresholds, residuals, and cumulative hazard]
Let $e=\{u,v\}$ be a bridge, where $u$ has label $1$ and $v$ has label $2$. 
Its original activation time and the lower bound recorded at activation are
\begin{align*}
 \gt_e^{\act}
 &:=
 \max\{\widehat\sfT_1(u),\widehat\sfT_2(v)\},\\
 L_e^{\act}
 &:=
 \big(\gt_e^{\act}-\widehat\sfT_1(u)\big)
 +
 \big(\gt_e^{\act}-\widehat\sfT_2(v)\big).
\end{align*}
The corresponding residual in the cumulative-hazard scale is
\begin{align*}
 R_e
 := H(\go_e)-H(L_e^{\act}) 
 = \go_e^\theta-(L_e^{\act})^\theta.
\end{align*}

For $q\ge\gt_e^{\act}$, define the collision threshold
\begin{align*}
 L_e(q)
 :=
 \big(q-\widehat\sfT_1(u)\big)
 +
 \big(q-\widehat\sfT_2(v)\big).
\end{align*}
In particular, $L_e(\gt_e^{\act})=L_e^{\act}$.

For $s\ge\fa_\ell\gt_e^{\act}$, define the rescaled bridge age by
\begin{align}\label{eq:Ae-def}
 A_e(s)
 :=
 \big(s-\fa_\ell\widehat\sfT_1(u)\big)
 +
 \big(s-\fa_\ell\widehat\sfT_2(v)\big).
\end{align}
Then $ L_e(s/\fa_\ell)=\frac{A_e(s)}{\fa_\ell}$.

The cumulative bridge hazard at rescaled time $s$ is
\begin{align}\label{eq:Lambda-def}
 \gL_s
 &:= \sum_{e\in\cB_s}
 \left[ H\left(\frac{A_e(s)}{\fa_\ell}\right) - H(L_e^{\act}) \right]
 =\sum_{e\in\cB_s}\left[\left(\frac{A_e(s)}{\fa_\ell }\right)^\theta
 -(L_e^{\act})^\theta\right]
 \notag\\
 &=\int_0^s \frac{2\theta}{\fa_\ell ^\theta}
 \sum_{e\in\cB_q} A_e(q)^{\theta-1}\,\dd q.
\end{align}
\end{defn}

The identity in~\eqref{eq:Lambda-def} follows from the chain rule.
Away from the finitely many vertex-addition times, $\frac{\dd}{\dd q}A_e(q)=2$.

\begin{prop}[Bridge residuals and Cox hazard]\label{prop:cox}
The following statements hold.
\begin{enumeratei}
\item \label{cox:residuals} \textup{(Independent residuals.)} For every deterministic $t\ge0$, conditionally on $\cF_t^{\rm orig}$,
the family
\begin{align*}
 \{R_e:e\in\cB_{\fa_\ell t}\}
\end{align*}
is independent and each variable has the $\Exp(1)$ law.

\item \label{cox:event} \textup{(Collision event.)} For every deterministic $s\ge0$, with $t=s/\fa_\ell $, on $\Omega_*$,
\begin{align}\label{eq:background-collision-event}
 \{\sfS_{\col}>s\}
 =
 \bigcap_{e\in\cB_s}
 \left\{
 R_e> L_e(t)^\theta-(L_e^{\act})^\theta
 \right\}.
\end{align}

\item \label{cox:survival} \textup{(Cox survival.)} For every $s\ge0$,
\begin{align*}
 \pr(\sfS_{\col}>s\mid\cF_s)=e^{-\gL_s}.
\end{align*}
Consequently,
\begin{align*}
 \pr(\sfS_{\col}>s)=\E e^{-\gL_s}.
\end{align*}

\item \label{cox:marked} \textup{(Marked form.)}
Let $e_{\col}$ denote the a.s.\ unique bridge realising the collision, and put $\cB_\infty:=\bigcup_{s\ge0}\cB_s$.
For $e\in\cB_\infty$, define its individual cumulative hazard by
\begin{align*}
 \gL_e(s)
 :=
 \begin{cases}
 0,
 &s<\fa_\ell\gt_e^{\act},\\
 \left(A_e(s)/\fa_\ell\right)^\theta-(L_e^{\act})^\theta,
 &s\ge\fa_\ell\gt_e^{\act}.
 \end{cases}
\end{align*}
Then $\gL_s=\sum_{e\in\cB_\infty}\gL_e(s)$.
For every non-negative predictable field $(\Psi(e,s))_{e,s}$, \ie~$s\mapsto\Psi(e,s)$ is left-continuous and adapted to
$(\cF_s)_{s\ge0}$,
\begin{align}\label{eq:cox-marked}
 \E\bigl[\Psi(e_{\col},\sfS_{\col})\bigr]
 =
 \E\int_0^\infty
 e^{-\gL_s}\,
 \frac{2\theta}{\fa_\ell^\theta}
 \sum_{e\in\cB_s}A_e(s)^{\theta-1}\Psi(e,s)
 \,\dd s .
\end{align}
In words, conditionally on the censored exploration, the collision occurs on the
bridge $e$ in $[s,s+\dd s)$ with rate
$\frac{2\theta}{\fa_\ell^\theta}A_e(s)^{\theta-1}\,\dd s$, thinned by the
survival factor $e^{-\gL_s}$.
\end{enumeratei}
\end{prop}

\begin{proof}
We first prove~\eqref{cox:residuals} by induction on the vertex-addition steps of the censored exploration. Let $L_h$ denote the current lower bound of any unrevealed active edge $h$. We claim that, conditionally on the exploration history, the hazard residuals $H(\go_h) - H(L_h)$ across all unrevealed active edges and bridges are independent $\Exp(1)$ variables, while untouched edges remain completely fresh.

The base case at time zero is trivial. Assume the claim holds up to a given step, and condition on the next winning edge. By the memoryless property of the hazard scale Lemma~\ref{lem:hazard-residual}, updating the lower bounds for the unselected active edges preserves the independence and the $\text{Exp}(1)$ law of their residuals. The winning edge is then fully revealed and discarded from the active set, while any newly incident edges draw fresh, independent weights. Crucially, if an unselected active edge becomes a bridge (due to its endpoint receiving the opposite label), it is permanently censored; its residual is thus safely frozen at its activation threshold $L_e^{\act}$ without undergoing further tests.

This inductive step guarantees the conditional statement holds throughout the exploration. Evaluating this process up to time $t$ and restricting the result to the bridge set $\cB_{\fa_\ell t}$ proves~\eqref{cox:residuals}.

We next identify the bridge event with the original collision event. On $\Omega_*$, define
\begin{align*}
 \tau_{\mathrm{br}}
 :=
 \inf\{q\ge0:\exists e\in\cB_{\fa_\ell q}\text{ such that }\go_e\le L_e(q)\}.
\end{align*}
We claim that
\begin{align}\label{eq:proof-bridge-equals-collision}
 \tau_{\mathrm{br}}=\tau_{\col}
 \qquad\text{on }\Omega_* .
\end{align}
Indeed, if $\tau_{\mathrm{br}}<\tau_{\col}$, choose $q$ with $\tau_{\mathrm{br}}<q<\tau_{\col}$. By Proposition~\ref{prop:censor}, the censored and uncensored explorations agree up to time $q$. A bridge satisfying $\go_e\le L_e(q)$ would then satisfy the original collision inequality before $\tau_{\col}$, a contradiction. Thus $\tau_{\mathrm{br}}\ge\tau_{\col}$.

Conversely, consider the exact collision time $q=\tau_{\col}$. By applying Proposition~\ref{prop:censor} strictly before $q$, the edge producing the original collision is already present as a bridge in the censored exploration by time $q$, and it satisfies the bridge inequality at time $q$. Hence $\tau_{\mathrm{br}}\le q=\tau_{\col}$. This proves~\eqref{eq:proof-bridge-equals-collision}.

Now fix deterministic $s\ge0$ and put $t=s/\fa_\ell $. Since bridges remain in the bridge set after activation,
\begin{align*}
 \{\sfS_{\col}>s\}
 =
 \{\tau_{\mathrm{br}}>t\}
 =
 \bigcap_{e\in\cB_s}\{\go_e>L_e(t)\}.
\end{align*}
For an activated bridge $e$, this is equivalent to
\begin{align*}
 \go_e^\theta-(L_e^{\act})^\theta
 >
 L_e(t)^\theta-(L_e^{\act})^\theta,
\end{align*}
which is exactly~\eqref{eq:background-collision-event}. This proves~\eqref{cox:event}.

Finally, condition on $\cF_s$. By~\eqref{cox:residuals}, the variables $\{R_e:e\in\cB_s\}$ are independent $\Exp(1)$ variables, while by~\eqref{cox:event}
\begin{align*}
 \{\sfS_{\col}>s\}
 =
 \bigcap_{e\in\cB_s}
 \left\{
 R_e>L_e(t)^\theta-(L_e^{\act})^\theta
 \right\}.
\end{align*}
Therefore
\begin{align*}
 \pr(\sfS_{\col}>s\mid\cF_s)
 &=
 \prod_{e\in\cB_s}
 \exp\left\{
 -\left[L_e(t)^\theta-(L_e^{\act})^\theta\right]
 \right\} 
 =
 \exp\left\{
 -\sum_{e\in\cB_s}
 \left[L_e(t)^\theta-(L_e^{\act})^\theta\right]
 \right\}
 =
 e^{-\gL_s},
\end{align*}
where the last equality is the definition~\eqref{eq:Lambda-def}. Taking expectations gives
\begin{align*}
 \pr(\sfS_{\col}>s)=\E e^{-\gL_s}.
\end{align*}
This proves~\eqref{cox:survival}.

Finally, we prove~\eqref{eq:cox-marked}. Let $\cF_\infty :=\gs\left(\bigcup_{s\ge0}\cF_s\right)$.
Continue the residual-clock induction used in the proof of~\eqref{cox:residuals} through all vertex-addition steps of the censored exploration. Since bridge residuals are never tested after activation, conditionally on $\cF_\infty$, the finite family $(R_e)_{e\in\cB_\infty}$ is independent and each member has the $\Exp(1)$ law.

Since $A_e(\fa_\ell\gt_e^{\act}) = \fa_\ell L_e^{\act}$, the function $\gL_e$ is continuous, non-decreasing, and absolutely continuous on $[0,\infty)$, with $\gL_e(s)\to\infty$ as $s\to\infty$. Moreover, for almost every $s$,
\begin{align*}
 \dd\gL_e(s)
 =
 \ind_{\{e\in\cB_s\}}
 \frac{2\theta}{\fa_\ell^\theta}
 A_e(s)^{\theta-1}\,\dd s.
\end{align*}

For $e\in\cB_\infty$, put $\tau_e:=\inf\{s\ge0:R_e\le\gL_e(s)\}$.
Conditionally on $\cF_\infty$,
\begin{align*}
 \pr(\tau_e>s\mid\cF_\infty)
 =e^{-\gL_e(s)},
\end{align*}
so the conditional law of $\tau_e$ is the Stieltjes measure $e^{-\gL_e(s)}\,\dd\gL_e(s)$.
The variables $(\tau_e)_{e\in\cB_\infty}$ are conditionally independent, and by~\eqref{cox:event} their minimum is $\sfS_{\col}$ with mark $e_{\col}$. Hence,
\begin{align*}
 \pr\bigl(
 e_{\col}=e,\ \sfS_{\col}\in\dd s
 \bigm|\cF_\infty
 \bigr)
 =
 \prod_{f\ne e}e^{-\gL_f(s)}
 e^{-\gL_e(s)}\,\dd\gL_e(s)
 =
 e^{-\gL_s}\,\dd\gL_e(s).
\end{align*}
Multiplying by $\Psi(e,s)$, summing over $e$, integrating in $s$, and taking expectations yields~\eqref{eq:cox-marked}. 
\end{proof}

\begin{rem}[Reduction achieved by the Cox representation]
Up to this point, all identities hold exactly for any finite network size $n$. The exact survival formula 
\begin{align*}
 \pr(\sfS_{\col}>s)=\E e^{-\gL_s}
\end{align*}
rigorously reduces the main theorem to evaluating the asymptotic behavior of the random hazard $\gL_s$ at the critical scale $s = \frac{1}{2}\log n + y$. The remainder of the proof focuses entirely on the large-$n$ limit of the censored clusters and the resulting bridge-age sum in~\eqref{eq:Lambda-def}.
\end{rem}

%%%%%%%%%%%%%%%%%%%%%%%%%%%%%%%%%%%%%%%%%%%%%%%%%%%%%%%%%%%%%
\section{Macroscopic CMJ approximation: proof of Theorem~\ref{thm:Macro} and~\ref{thm:macro-cross}}
\label{sec:macro-cmj}

This section executes steps~\emph{(C3)} and~\emph{(C4)} of our macroscopic roadmap. Following the exact Cox reduction in Proposition~\ref{prop:cox}, our remaining task is to determine the asymptotic behavior of the cumulative bridge hazard $\gL_s$ at the critical collision times $s=\frac12\log n+O(1)$. The argument proceeds in three parts.

Recall the limiting objects defined in Section~\ref{sec:local-cmj}. We write $W_1, W_2$ for independent copies of the limiting CMJ martingale $M_\infty$, and let $W_{1,n}, W_{2,n}$ denote the corresponding finite-$n$ approximations constructed below.

For $\eps\in(0,1/2)$ and $y\in\dR$, write
\begin{align}\label{eq:s0-s1-def}
 \fs_0=\fs_0(n,\eps):=\eps\log n,
 \qquad
 \fs_1=\fs_1(n,y):=\frac12\log n+y.
\end{align}
We use this notation for the macroscopic collision window throughout this section.

\smallskip
\noindent\emph{(C3a) Proposal CMJ processes and ghost control.}
We first couple the exploration to two independent, unrestricted finite-$n$ CMJ branching random walks, which we call the \emph{proposal processes}. The censored Dijkstra clusters are recovered by retaining a proposed birth exactly when its parent is retained and its spatial label is unoccupied; every other proposal and all of its descendants are ghosts. Proposition~\ref{prop:cmj-package} demonstrates that these rejected proposals---which we track as ``ghost'' particles---are stochastically negligible. This allows us to transfer three crucial uniform estimates from the ideal proposal processes to the accepted physical clusters: exponential population growth, convergence to the stable-age distribution, and the uniform decay of non-zero Fourier modes. (See Subsection~\ref{ssec:finite-cmj}).

\smallskip
\noindent\emph{(C3b) Spatial mixing and bridge statistics.}
Population sizes alone cannot determine the collision hazard, as bridges only form between vertices separated by distance at most $\ell$. Proposition~\ref{prop:pair} translates the Fourier decay into an interval discrepancy bound, proving that the bridge density satisfies
\begin{align*}
 \frac{B_s}{N_1(s)N_2(s)}
 \stackrel{\pr}{\to}
 2\gl
\end{align*}
uniformly across the collision window. Applying the same spatial mixing argument to age-weighted empirical measures shows that the endpoint ages of available bridges are asymptotically independent and follow the stable-age law. In particular, on the collision window,
\begin{align*}
 \int_{\fs_0}^{\fs_1}\frac1n
 \abs{
 \sum_{e\in\cB_s}A_e(s)^{\theta-1}
 -\Gamma(\theta+1)B_s
 }\,\dd s
 \stackrel{\pr}{\to}0.
\end{align*}
When $0<\theta<1$, the singularity at age zero is carefully handled by an integrated truncation argument rather than a direct pointwise estimate. (See Subsection~\ref{ssec:mix-bridge}).

\smallskip
\noindent\emph{(C4) Integrated hazard and the collision limit.}
Lemma~\ref{lem:integrated} first shows that early-time contributions to the hazard are negligible. It then substitutes the asymptotic bridge density and age statistics into the exact Cox hazard formula to deduce that, for every fixed $y\in\dR$,
\begin{align*}
 \gL_{\frac12\log n+y}
 -
 \theta W_{1,n}W_{2,n}e^{2y}
 \stackrel{\pr}{\to}
 0.
\end{align*}
Combined with the joint weak convergence $(W_{1,n},W_{2,n})\Rightarrow(W_1,W_2)$, this establishes the limiting Cox intensity and completes the proof of Theorem~\ref{thm:Macro}. (See Subsections~\ref{ssec:intensity} and~\ref{ssec:proof-Macro-main}).

\subsection{Finite-$n$ CMJ process and proposal construction}\label{ssec:finite-cmj}

This subsection establishes the one-source estimates governing the individual cluster growths. We couple the exploration with an unrestricted finite-$n$ CMJ branching random walk, which we call the \emph{proposal process}. The censored Dijkstra clusters are recovered by retaining a proposed birth exactly when its parent is retained and its spatial label is unoccupied; all other proposals and their descendants are ghosts. We first record the finite-$n$ CMJ estimates needed for the proposal process, then add the spatial Fourier estimate, and finally show that the suppressed ``ghost'' particles are asymptotically negligible within the collision window. To maintain the flow, these auxiliary statements are recorded here where they are needed, while their proofs are deferred to Appendix~\ref{app:macro-cmj-inputs}.

We use the local scaling from Section~\ref{sec:local-cmj}. Recall the displacement set $\cD_n$ from Section~\ref{sec:macro-cox}. Define
\begin{align*}
 \bar\mu_n(\dd s):=|\cD_n|\pr(\fa_\ell \go\in\dd s)
\end{align*}
for the age marginal of one finite-$n$ reproduction point process. Since $\fa_\ell^\theta=(2\ell-1)\Gamma(\theta+1)$ and $\Gamma(\theta+1)=\theta\Gamma(\theta)$, its exact density is
\begin{align}\label{eq:bar-mu-density}
 \bar\mu_n(\dd s)
 =
 \frac{|\cD_n|}{2\ell-1}
 \frac{s^{\theta-1}}{\Gamma(\theta)}
 e^{-(s/\fa_\ell)^\theta}\,\dd s.
\end{align}
For the spatial proposal process, a particle at label $x\in\dT_n$ has one potential child at $x+z$ for each $z\in\cD_n$, with birth age $\fa_\ell\go_z$. Thus one marked reproduction point process is
\begin{align*}
 \Xi_n:=\sum_{z\in\cD_n}\gd_{(\fa_\ell\go_z,z)}.
\end{align*}
Its age marginal is $\bar\mu_n$.

Although $\bar\mu_n$ converges to the limiting measure $\mu_\theta$, it is not exactly equal to it. Since the process is observed up to times of order $\log n$, we use the exact finite-$n$ Malthusian normalization before passing to the limit. Define
\begin{align*}
 \fm_n(q):=\int_0^\infty e^{-qs}\bar\mu_n(\dd s),
 \quad q>0.
\end{align*}
Let 
\begin{align}\label{def:alpha_n}
    \alpha_n\text{ be the unique solution of the equation $\fm_n(x)=1$}.
\end{align}
Put
\begin{align}\label{eq:prelimit-spine-law}
 \nu_n(\dd s):=e^{-\alpha_ns}\bar\mu_n(\dd s),
 \qquad
 \theta_n:=\int_0^\infty s\,\nu_n(\dd s).
\end{align}
The equation $\fm_n(\alpha_n)=1$ makes $\nu_n$ a probability measure. It is the Malthusian tilted ancestral-age law, while $\theta_n$ is its mean. The factor $\alpha_n\theta_n$ below is the standard conversion between the mean-one intrinsic martingale and the population normalization.

Let $Z_n^{\op}$ be one proposal CMJ process. A proposal particle $u$ has birth time $\tau_u$, generation $|u|$, and spatial label $X_u\in\dT_n$. Define its generation martingale by
\begin{align}\label{eq:prelimit-generation-martingale}
 \cW_{n,r}:=\sum_{|u|=r}e^{-\alpha_n\tau_u},
 \quad r\ge0.
\end{align}

\begin{lem}[Finite-$n$ normalization and intrinsic martingale]
\label{lem:finite-n-cmj-martingale}
As $n\to\infty$, $\alpha_n=1+O(n^{-1})$ and $\theta_n=\theta+O(n^{-1})$. The densities of $\nu_n$ converge in $L^1(\dd s)$ to the $\Gamma(\theta,1)$ density. Moreover, for some $\rho>0$ and $n_0<\infty$,
$\sup_{n\ge n_0}\int_0^\infty e^{\rho s}\nu_n(\dd s)<\infty$.

The martingale in~\eqref{eq:prelimit-generation-martingale} converges in $L^2$ to a mean-one limit $\cW_{n,\infty}$. There exist $C<\infty$, $q\in(0,1)$, and $n_0<\infty$ such that
\begin{align*}
 \sup_{n\ge n_0}
 \E\abs{\cW_{n,\infty}-\cW_{n,r}}^2
 \le Cq^r,
 \quad r\ge0.
\end{align*}
Finally,
\begin{align*}
 \frac{\cW_{n,\infty}}{\alpha_n\theta_n}
 \Rightarrow M_\infty,
\end{align*}
and the same convergence holds jointly for finitely many independent copies.
\end{lem}

The preceding lemma identifies the random growth factor, but the population size and bridge hazards also depend on the empirical age distribution. The classical random-characteristic theory is insufficient here because the reproduction law varies with $n$ and the observation time grows with $n$. Therefore we establish a uniform random-characteristic estimate over the collision window. For a characteristic $\phi:[0,\infty)\to\dR$, write
\begin{align*}
 Z_n^\phi(t):=\sum_{u:\tau_u\le t}\phi(t-\tau_u).
\end{align*}

\begin{lem}[Uniform finite-$n$ random characteristics]
\label{lem:uniform-triangular-nerman}
Fix $\eps\in(0,1/2)$ and $y\in\dR$.
If $\phi$ is continuous and $|\phi(a)|\le Ce^{\eta a}$ for some $\eta<1$, or if $\phi$ is bounded and Riemann-integrable, then
\begin{align}\label{eq:uniform-triangular-nerman-window}
 \sup_{\fs_0\le t\le \fs_1}
 \abs{
 e^{-\alpha_nt}Z_n^\phi(t)
 -\frac{\cW_{n,\infty}}{\theta_n}
 \int_0^\infty e^{-\alpha_na}\phi(a)\,\dd a
 }
 \stackrel{\pr}{\to}0.
\end{align}
\end{lem}

The marginal age distribution alone is insufficient to count admissible bridges, so we also track spatial labels. For every fixed non-zero Fourier mode, the spatial multiplier has limiting modulus strictly smaller than one. The many-to-two estimate in Appendix~\ref{app:macro-cmj-inputs} therefore gives an $L^2$ growth exponent $\rho_k<1$, so the mode is negligible relative to the population scale $e^s$, uniformly over the collision window.

\begin{lem}[Proposal Fourier decay]
\label{lem:fourier-many-to-two}
Fix $k\in\dZ\setminus\{0\}$ and let $g$ be bounded and Riemann-integrable. For one proposal process, put
\begin{align*}
 Y_{n,k,g}^{\op}(s)
 :=\sum_{u:\tau_u\le s}g(s-\tau_u)e^{2\pi \mathrm{i}kX_u/n}.
\end{align*}
For every fixed $\eps\in(0,1/2)$ and $y\in\dR$,
\begin{align}\label{eq:proposal-fourier-decay}
 \sup_{\fs_0\le s\le \fs_1}
 e^{-s}\abs{Y_{n,k,g}^{\op}(s)}
 \stackrel{\pr}{\to}0.
\end{align}
\end{lem}

We couple the censored graph exploration with two independent unrestricted CMJ processes (the proposal processes). A proposal is retained if its parent is retained and its spatial label is unoccupied; otherwise, it is a ghost, and all its descendants are ghosts. Whenever a retained particle proposes along a physical edge that has not appeared before, we use the corresponding fresh graph weight. If the edge has already appeared, or if the parent is a ghost, we use an independent auxiliary weight instead.

Because the choice between a physical and an auxiliary weight is predictable from the exploration history, every proposed child receives a fresh weight with law $\go$, independently of the past. Thus, each unrestricted proposal tree has exact reproduction law $\Xi_n$. By induction on the chronological birth times, the retained labels, parent relations, and birth times precisely match those of the two-source censored Dijkstra exploration.

The auxiliary variables are not part of the censored filtration. Since this filtration only records the active exploration and the lower bounds on unrevealed edges, the residual weights of bridges remain perfectly concealed. This ensures the exact Cox representation required in Proposition~\ref{prop:cox}.

The ghost estimate is based on a predictable collision bound. Call a ghost \emph{primary} if its parent is retained but its proposed label is already occupied, and put $Z_{\tot,n}^{\op}:=Z_{1,n}^{\op}+Z_{2,n}^{\op}$. A one-child-deleted construction gives, for every non-negative measurable function $F$,
\begin{align*}
 \E\sum_{v\text{ primary}}F(\tau_v)
 \le
 \frac1{|\cD_n|}
 \E\int_{(0,\infty)}
 F(r)Z_{\tot,n}^{\op}(r-)
 \,\dd Z_{\tot,n}^{\op}(r).
\end{align*}
Combining this bound with the proposal-process moment estimates controls the descendants of all primary ghosts and yields the following lemma; see Appendix~\ref{app:macro-cmj-inputs} for the details.

\begin{lem}[Ghost suppression]
\label{lem:ghost-compensator}
Fix $\eps\in(0,1/2)$ and $y\in\dR$.
For a characteristic satisfying $|g(a)|\le Ce^{\eta a}$ for some $\eta<1$, let $G_{n,g}(\fs_1)$ denote the supremum, over $s\in[0,\fs_1]$, of the total absolute $g$-characteristic of all ghost particles in the two proposal processes at time $s$. Then
\begin{align}\label{eq:ghost-characteristic-negligible}
 \sup_{n\ge n_0}\E G_{n,g}(\fs_1)<\infty,
 \qquad
 e^{-\fs_0}G_{n,g}(\fs_1)
 \stackrel{\pr}{\to}0.
\end{align}
\end{lem}

\begin{prop}[CMJ approximation]\label{prop:cmj-package}
Fix $\eps\in(0,1/2)$ and $y\in\dR$.
There exist non-negative random variables $W_{1,n},W_{2,n}$, defined on the probability space carrying the coupling above, such that
\begin{align}\label{eq:Wn-joint-conv}
 (W_{1,n},W_{2,n})\Rightarrow(W_1,W_2),
\end{align}
and, for $i=1,2$,
\begin{align}\label{eq:cmj-growth-package}
 \sup_{\fs_0\le s\le \fs_1}
 \abs{e^{-s}N_i(s)-W_{i,n}}
 \stackrel{\pr}{\to}0.
\end{align}
Moreover, if
\begin{align*}
 A_v^{(i)}(s)
 :=s-\fa_\ell\widehat\sfT_i(v),
 \qquad
 v\in\cC_i(s),
\end{align*}
then, for every continuous $g$ satisfying $|g(a)|\le Ce^{\eta a}$ for some $\eta<1$, and also for every bounded Riemann-integrable $g$,
\begin{align}\label{eq:cmj-age-package}
 \sup_{\fs_0\le s\le \fs_1}
 \abs{
 \frac1{N_i(s)}
 \sum_{v\in\cC_i(s)}g\bigl(A_v^{(i)}(s)\bigr)
 -\int_0^\infty g(a)e^{-a}\,\dd a
 }
 \stackrel{\pr}{\to}0.
\end{align}
For every fixed $k\in\dZ\setminus\{0\}$ and every bounded Riemann-integrable $g$,
\begin{align}\label{eq:cmj-fourier-package}
 \sup_{\fs_0\le s\le \fs_1}
 \abs{
 \frac1{N_i(s)}
 \sum_{v\in\cC_i(s)}
 g\bigl(A_v^{(i)}(s)\bigr)e^{2\pi \mathrm{i}kv/n}
 }
 \stackrel{\pr}{\to}0.
\end{align}
\end{prop}

\begin{proof}
Take two independent copies $Z_{i,n}^{\op}$, $i=1,2$, of the proposal process. Let $\cW_{i,n,\infty}$ be their intrinsic martingale limits and define
\begin{align*}
 W_{i,n}^{\op}
 :=\frac{\cW_{i,n,\infty}}{\alpha_n\theta_n}.
\end{align*}
Applying Lemma~\ref{lem:uniform-triangular-nerman} with the counting characteristic $\phi\equiv1$ gives
\begin{align}\label{eq:proposal-growth-patch}
 \sup_{\fs_0\le s\le \fs_1}
 \abs{e^{-\alpha_ns}Z_{i,n}^{\op}(s)-W_{i,n}^{\op}}
 \stackrel{\pr}{\to}0.
\end{align}
Lemma~\ref{lem:finite-n-cmj-martingale} gives
\begin{align}\label{eq:proposal-W-convergence-patch}
 (W_{1,n}^{\op},W_{2,n}^{\op})
 \Rightarrow(W_1,W_2).
\end{align}
Since $M_\infty>0$ almost surely, the normalized proposal populations are bounded away from zero in probability, uniformly on $[\fs_0,\fs_1]$.

Applying Lemma~\ref{lem:uniform-triangular-nerman} to a general age characteristic $g$ and dividing by~\eqref{eq:proposal-growth-patch} yields
\begin{align*}
 \sup_{\fs_0\le s\le \fs_1}
 \abs{
 \frac1{Z_{i,n}^{\op}(s)}
 \sum_{\tau_u\le s}g(s-\tau_u)
 -\alpha_n\int_0^\infty e^{-\alpha_na}g(a)\,\dd a
 }
 \stackrel{\pr}{\to}0.
\end{align*}
The deterministic term converges to $\int_0^\infty g(a)e^{-a}\,\dd a$. Moreover,
\begin{align*}
 \sup_{s\le \fs_1}\abs{e^{(\alpha_n-1)s}-1}\to0
\end{align*}
because $\alpha_n=1+O(n^{-1})$ and $\fs_1=O(\log n)$. Thus, the proposal processes satisfy the required population and stable-age estimates under the $e^{-s}$ normalization. The corresponding Fourier estimate follows by combining Lemma~\ref{lem:fourier-many-to-two} with~\eqref{eq:proposal-growth-patch}.

By construction, the accepted processes coincide exactly with the censored Dijkstra clusters. For each characteristic used above, the discrepancy between the corresponding proposal and accepted sums, uniformly over $s\le\fs_1$, is bounded by the ghost characteristic $G_{n,|g|}(\fs_1)$. Since
\begin{align*}
 \sup_{\fs_0\le s\le\fs_1}
 e^{-s}G_{n,|g|}(\fs_1)
 \le
 e^{-\fs_0}G_{n,|g|}(\fs_1)
 \stackrel{\pr}{\to}0
\end{align*}
by Lemma~\ref{lem:ghost-compensator}, the population estimate transfers first. In particular, $\inf_{\fs_0\le s\le\fs_1}e^{-s}N_i(s)$ is bounded away from zero in probability, and
\begin{align*}
 \sup_{\fs_0\le s\le\fs_1}
 \abs{\frac{Z_{i,n}^{\op}(s)}{N_i(s)}-1}
 \stackrel{\pr}{\to}0.
\end{align*}
Dividing the corresponding numerator discrepancies by $N_i(s)$ and using this denominator comparison transfers the stable-age and Fourier estimates. Setting $W_{i,n}:=W_{i,n}^{\op}$ proves~\eqref{eq:Wn-joint-conv}--\eqref{eq:cmj-fourier-package}.
\end{proof}

While Proposition~\ref{prop:cmj-package} controls the growth and spatial distribution of each cluster separately, the collision hazard $\gL_s$ depends on the edges connecting them. The following subsection bridges this gap by combining these one-source estimates. We show that the number of available bridge pairs is asymptotic to $2\gl N_1(s)N_2(s)$, and their joint empirical age distribution factors into the product of two independent stable-age laws.

%%%%%%%%%%%%%%%%%%
\subsection{Spatial mixing and bridge statistics}\label{ssec:mix-bridge}

The bridge hazard depends on pairs of vertices rather than on individual clusters. We therefore embed the discrete cycle into the continuum torus $\dS^1=[0,1)$ and use the Fourier decay from Proposition~\ref{prop:cmj-package} to show that the two empirical cluster measures are spatially mixed at every fixed Fourier mode. This yields the asymptotic density of geometric bridge pairs and, after weighting by ages, the correct average contribution of the singular kernel $A_e(s)^{\theta-1}$.

Embed $\dT_n$ into $\dT=[0,1)$ by $v\mapsto v/n$. For $i=1,2$, let
\begin{align*}
 \pi_i^s:=\frac1{N_i(s)}\sum_{v\in\cC_i(s)}\gd_{v/n},
 \qquad
 \pihat_i^s(k):=\int_{\dT}e^{2\pi \mathrm{i}kz}\,\pi_i^s(\dd z).
\end{align*}
For $e=\{u,v\}\in\cB_s$ with $u\in\cC_1(s)$ and $v\in\cC_2(s)$, write
\begin{align*}
 A_e(s):=A_u^{(1)}(s)+A_v^{(2)}(s).
\end{align*}

\begin{prop}[Pair density]\label{prop:pair}
Fix $\eps\in(0,1/2)$ and $y\in\dR$. Then
\begin{align*}
 \sup_{\fs_0\le s\le \fs_1}
 \abs{\frac{B_s}{N_1(s)N_2(s)}-2\gl}
 \stackrel{\pr}{\to}0.
\end{align*}
Moreover,
\begin{align}\label{eq:B-positive}
 \pr\left(\inf_{\fs_0\le s\le \fs_1}B_s=0\right)\to0,
 \qquad
 \inf_{\fs_0\le s\le \fs_1}B_s\stackrel{\pr}{\to}\infty.
\end{align}
If $g_1,g_2$ are bounded Riemann-integrable functions on $[0,\infty)$, then
\begin{align*}
 \sup_{\fs_0\le s\le \fs_1}
 \abs{\frac1{B_s}\sum_{\{u,v\}\in\cB_s}
 g_1\bigl(A_u^{(1)}(s)\bigr)g_2\bigl(A_v^{(2)}(s)\bigr)
 -\prod_{i=1}^2\int_0^\infty g_i(a)e^{-a}\,\dd a}
 \stackrel{\pr}{\to}0,
\end{align*}
where the ratio may be defined arbitrarily on $\{B_s=0\}$. Finally,
\begin{align}\label{eq:age-kernel-average}
 \int_{\fs_0}^{\fs_1}\frac1n
 \abs{\sum_{e\in\cB_s}A_e(s)^{\theta-1}-\Gamma(\theta+1)B_s}\,\dd s
 \stackrel{\pr}{\to}0.
\end{align}
\end{prop}

\begin{proof}
Let
\begin{align*}
 I_n:=\{z\in\dT:d_{\dT}(z,0)\le\ell/n\}.
\end{align*}
Since $\gl<1/2$, the set $I_n$ is an interval in $\dT$ of length $|I_n|=2\gl+o(1)$.
If $\widetilde\pi_2^s(A):=\pi_2^s(-A)$, then
\begin{align*}
 \frac{B_s}{N_1(s)N_2(s)}=\pi_1^s*\widetilde\pi_2^s(I_n).
\end{align*}
Applying the Erd\H{o}s--Tur\'an inequality for interval discrepancy \cite[Theorem~2.5]{kn74} (after a rotation of $\dT$) gives an absolute constant $C_{\mathrm{ET}}>0$ such that, for every $m\ge1$,
\begin{align*}
 \sup_{\fs_0\le s\le \fs_1}\abs{\pi_1^s*\widetilde\pi_2^s(I_n)-|I_n|}
 \le
 C_{\mathrm{ET}}\left(
 \frac1{m+1}
 +
 \sum_{k=1}^m
 \frac{\sup_{\fs_0\le s\le \fs_1}|
 \pihat_1^s(k)\overline{\pihat_2^s(k)}|}{k}
 \right).
\end{align*}
Choose $m$ large and fixed. The estimate~\eqref{eq:cmj-fourier-package}, applied with $g\equiv1$, implies that the finite Fourier sum vanishes in probability. Taking $m\to\infty$ then yields the pair-density estimate. Since $2\gl>0$ and $N_1(s)N_2(s)\to\infty$ in probability uniformly on $[\fs_0,\fs_1]$,~\eqref{eq:B-positive} follows.

For bridge-age sampling, first take non-negative bounded Riemann-integrable $g_i$ and define
\begin{align*}
 \pi_{i,g_i}^s
 :=
 \frac1{N_i(s)}\sum_{v\in\cC_i(s)}
 g_i\bigl(A_v^{(i)}(s)\bigr)\gd_{v/n},
 \qquad i=1,2.
\end{align*}
Applying the finite-measure form of the Erd\H{o}s--Tur\'an argument (obtained by normalising the total mass of $\pi_{i,g_i}^s$, which by~\eqref{eq:cmj-age-package} converges in probability to $\int_0^\infty g_i(a)e^{-a}\dd a$, uniformly on $[\fs_0,\fs_1]$) to $\pi_{1,g_1}^s*\widetilde\pi_{2,g_2}^s$, the zero Fourier mode is controlled by~\eqref{eq:cmj-age-package} and the non-zero modes by~\eqref{eq:cmj-fourier-package}. Hence
\begin{align*}
\sup_{\fs_0\le s\le \fs_1}
\Bigg|
\frac1{N_1(s)N_2(s)}
\sum_{\{u,v\}\in\cB_s}
 g_1\bigl(A_u^{(1)}(s)\bigr)g_2\bigl(A_v^{(2)}(s)\bigr) -2\gl\prod_{i=1}^2\int_0^\infty g_i(a)e^{-a}\,\dd a
\Bigg|\stackrel{\pr}{\to}0.
\end{align*}
Dividing by $B_s/(N_1(s)N_2(s))$ and using the pair-density estimate together with~\eqref{eq:B-positive} proves the product statement. The result extends to general bounded Riemann-integrable $g_i$ by decomposing them into positive and negative parts.

Approximating from above and below by finite linear combinations of rectangle indicators extends this to any bounded Riemann-integrable function $\varphi$ with compact support on $[0,\infty)^2$, yielding
\begin{align}\label{eq:bridge-two-age-test}
 \sup_{\fs_0\le s\le \fs_1}
 \abs{\frac1{B_s}\sum_{\{u,v\}\in\cB_s}
 \varphi\bigl(A_u^{(1)}(s),A_v^{(2)}(s)\bigr)
 -\int_0^\infty\int_0^\infty\varphi(a,b)e^{-a}e^{-b}\,\dd a\,\dd b}
 \stackrel{\pr}{\to}0.
\end{align}
Also,
\begin{align}\label{eq:bridge-integral-tight}
 \int_{\fs_0}^{\fs_1}\frac{B_s}{n}\,\dd s=O_{\pr}(1),
 \qquad
 \frac{B_{\fs_1}}{n}=O_{\pr}(1),
\end{align}
by the pair-density and growth estimates.

Let $E_1,E_2$ be independent $\Exp(1)$ variables and put $X:=E_1+E_2$. For $0<r<R<\infty$, set
\begin{align*}
 \varphi_{r,R}(a,b):=(a+b)^{\theta-1}\ind\{r\le a+b\le R\}.
\end{align*}
Multiplying~\eqref{eq:bridge-two-age-test} by $B_s/n$ and integrating gives
\begin{align}\label{eq:truncated-age-kernel}
 \int_{\fs_0}^{\fs_1}\frac1n
 \abs{\sum_{e\in\cB_s}A_e(s)^{\theta-1}\ind\{r\le A_e(s)\le R\}
 -B_s\E\big[X^{\theta-1}\ind\{r\le X\le R\}\big]}
 \,\dd s
 \stackrel{\pr}{\to}0.
\end{align}

It remains to uniformly bound the tail contributions. For the upper tail, if $0<\theta<1$, then $x^{\theta-1}\ind\{x>R\}\le R^{\theta-1}$ and~\eqref{eq:bridge-integral-tight} applies. If $\theta\ge1$, fix $\eta\in(0,1)$ and use
\begin{align*}
 x^{\theta-1}\ind\{x>R\}
 \le C_{\theta,\eta}e^{-\eta R/2}e^{\eta x}.
\end{align*}
Since $\cB_s\subseteq\cC_1(s)\times\cC_2(s)$,
\begin{align*}
 \sum_{e\in\cB_s}e^{\eta A_e(s)}
 \le
 \Biggl(\sum_{u\in\cC_1(s)}e^{\eta A_u^{(1)}(s)}\Biggr)
 \Biggl(\sum_{v\in\cC_2(s)}e^{\eta A_v^{(2)}(s)}\Biggr).
\end{align*}
Applying~\eqref{eq:cmj-age-package} to $g(a)=e^{\eta a}$ together with~\eqref{eq:cmj-growth-package} shows that the integral of the right-hand side, scaled by $1/n$, is $O_{\pr}(1)$. Thus the upper tail is $O_{\pr}(e^{-\eta R/2})$.

For the lower tail, if $\theta>1$, then $A_e(s)^{\theta-1}\ind\{A_e(s)\le r\} \le r^{\theta-1}$, so~\eqref{eq:bridge-integral-tight} gives an $O_{\pr}(r^{\theta-1})$ contribution. If $\theta=1$, applying~\eqref{eq:bridge-two-age-test} to $\varphi_r(a,b):=\ind\{a+b\le r\}$ and using~\eqref{eq:bridge-integral-tight} gives
\begin{align*}
 \int_{\fs_0}^{\fs_1}\frac1n
 \sum_{e\in\cB_s}\ind\{A_e(s)\le r\}\,\dd s
 =
 O_{\pr}\bigl(\pr(X\le r)\bigr)+o_{\pr}(1).
\end{align*}
Both bounds vanish as $r\downarrow0$.

It remains to treat $0<\theta<1$. Let $\gs_e$ be the rescaled activation time of bridge $e$. For $s\ge\gs_e$, the bridge age satisfies $A_e(s)=A_e(\gs_e)+2(s-\gs_e)\ge 2(s-\gs_e)$. Since the map $x\mapsto x^{\theta-1}$ is strictly decreasing for $\theta<1$, substituting $u=s-\gs_e$ bounds the integral for each bridge uniformly for $0<r\le1$ by
\begin{align*}
 \int_{\max(\fs_0, \gs_e)}^{\fs_1}A_e(s)^{\theta-1}\ind\{A_e(s)\le r\}\,\dd s
 \le \int_0^{r/2}(2u)^{\theta-1}\,\dd u
 =\frac{r^\theta}{2\theta}.
\end{align*}
Hence, using the monotonicity of $\cB_s$ and~\eqref{eq:bridge-integral-tight}, extending the sum over the final set $\cB_{\fs_1}$ yields
\begin{align*}
 \int_{\fs_0}^{\fs_1}\frac1n\sum_{e\in\cB_s}
 A_e(s)^{\theta-1}\ind\{A_e(s)\le r\}\,\dd s
 \le \frac{r^\theta}{2\theta}\frac{B_{\fs_1}}{n} = O_{\pr}(r^\theta).
\end{align*}
Since $X^{\theta-1}$ is integrable, its omitted lower and upper tails vanish in expectation:
\begin{align*}
 \E\bigl[X^{\theta-1}\ind\{X<r\}\bigr]
 +
 \E\bigl[X^{\theta-1}\ind\{X>R\}\bigr]
 \to0
\end{align*}
as $r\downarrow0$ and $R\to\infty$. Letting first $n\to\infty$, then $R\to\infty$, and finally $r\downarrow0$ in~\eqref{eq:truncated-age-kernel} proves~\eqref{eq:age-kernel-average}, since $\E X^{\theta-1}=\Gamma(\theta+1)$.
\end{proof}

\begin{rem}[The singularity at age zero]
The last estimate in Proposition~\ref{prop:pair} is the only place where the possible singularity of $x^{\theta-1}$ for $0<\theta<1$ enters the macroscopic proof. The bridge-age average is first proved for truncated bounded kernels and then extended by controlling the small-age and large-age tails uniformly over the time window. With these tail bounds established, the integrated hazard can be approximated by assuming each bridge contributes the mean $\Gamma(\theta+1)$.
\end{rem}

%%%%%%%%%%%%%%%%%%%%
\subsection{Integrated bridge-intensity asymptotics and proof of the theorem}\label{ssec:intensity} 

We now combine the exact Cox formula with the CMJ and mixing estimates. The initial accumulation of the hazard, up to $\fs_0=\eps\log n$, is asymptotically negligible because the sizes of both exploration clusters remain $o(n^{1/2})$. On the remaining interval, Proposition~\ref{prop:pair} replaces the age-weighted bridge sum by $\Gamma(\theta+1)B_s$. This matches the stable-age constant from~\eqref{eq:age-pair-gamma}: the endpoint ages of a typical bridge asymptotically behave as independent $\Exp(1)$ variables, so their sum follows a $\Gamma(2,1)$ distribution, yielding
\begin{align*}
 \E(A_1+A_2)^{\theta-1}=\Gamma(\theta+1).
\end{align*}
The pair-density estimate then replaces $B_s$ by $2\gl N_1(s)N_2(s)$, and the CMJ growth estimate evaluates the integral of $N_1(s)N_2(s)$ to yield the explicit limit $\theta W_1W_2e^{2y}$.

\begin{lem}[Integrated bridge-hazard asymptotics]\label{lem:integrated}
Fix $\eps\in(0,1/2)$ and $y\in\dR$. Then
\begin{align}\label{eq:early-hazard-negligible}
 \gL_{\fs_0}\stackrel{\pr}{\to}0.
\end{align}
Moreover,
\begin{align}\label{eq:hazard-reduction}
 \gL_{\fs_1}-\gL_{\fs_0}
 -\frac{2\theta\Gamma(\theta+1)}{\fa_\ell ^\theta}
 \int_{\fs_0}^{\fs_1}B_s\,\dd s
 \stackrel{\pr}{\to}0.
\end{align}
Consequently, with the finite-$n$ limits $W_{1,n},W_{2,n}$ from Proposition~\ref{prop:cmj-package},
\begin{align}\label{eq:integrated-Wn}
 \gL_{\fs_1}-\theta W_{1,n}W_{2,n}e^{2y}
 \stackrel{\pr}{\to}0.
\end{align}
\end{lem}

\begin{proof}
Each bridge activated by rescaled time $\fs_0$ contributes at most $(2\fs_0/\fa_\ell )^\theta$ to the cumulative hazard. Thus
\begin{align*}
 0\le\gL_{\fs_0}
 \le N_1(\fs_0)N_2(\fs_0)\left(\frac{2\fs_0}{\fa_\ell }\right)^\theta.
\end{align*}
By Proposition~\ref{prop:cmj-package}, $N_i(\fs_0)=O_{\pr}(e^{\fs_0})$, while $\fa_\ell ^\theta\asymp n$. Hence
\begin{align*}
 \gL_{\fs_0}=O_{\pr}(n^{2\eps-1}\fs_0^\theta)=o_{\pr}(1),
\end{align*}
which proves~\eqref{eq:early-hazard-negligible}. By the hazard identity~\eqref{eq:Lambda-def},
\begin{align*}
 \gL_{\fs_1}-\gL_{\fs_0}
 =
 \frac{2\theta}{\fa_\ell^\theta}
 \int_{\fs_0}^{\fs_1}
 \sum_{e\in\cB_s}A_e(s)^{\theta-1}\,\dd s.
\end{align*}
Since $\fa_\ell^\theta\asymp n$, Proposition~\ref{prop:pair} gives
\begin{align*}
\abs{\gL_{\fs_1}-\gL_{\fs_0}
 -\frac{2\theta\Gamma(\theta+1)}{\fa_\ell^\theta}
 \int_{\fs_0}^{\fs_1}B_s\,\dd s}
\le
 \frac{C}{n}\int_{\fs_0}^{\fs_1}
 \abs{
 \sum_{e\in\cB_s}A_e(s)^{\theta-1}
 -\Gamma(\theta+1)B_s
 }\,\dd s
 \stackrel{\pr}{\to}0.
\end{align*}
This proves~\eqref{eq:hazard-reduction}. Combining this with~\eqref{eq:early-hazard-negligible}, we obtain
\begin{align}\label{eq:hazard-after-reduction}
 \gL_{\fs_1}
 =
 \frac{2\theta\Gamma(\theta+1)}{\fa_\ell^\theta}
 \int_{\fs_0}^{\fs_1}B_s\,\dd s
 +o_{\pr}(1).
\end{align}
By Propositions~\ref{prop:pair} and~\ref{prop:cmj-package},
\begin{align*}
&\frac1n\int_{\fs_0}^{\fs_1}
\left|B_s-2\gl\,N_1(s)N_2(s)\right|\,\dd s \\
&\quad\le
\sup_{\fs_0\le s\le\fs_1}
\left|\frac{B_s}{N_1(s)N_2(s)}-2\gl\right|
\cdot
\frac1n\int_{\fs_0}^{\fs_1}N_1(s)N_2(s)\,\dd s
=o_{\pr}(1).
\end{align*}
Since $2\gl n\Gamma(\theta+1)/\fa_\ell^\theta\to 1$, the right-hand side of~\eqref{eq:hazard-after-reduction} is
\begin{align*}
 \frac{2\theta}{n}(1+o(1))
 \int_{\fs_0}^{\fs_1}N_1(s)N_2(s)\,\dd s
 +o_{\pr}(1).
\end{align*}
The CMJ growth estimate gives, uniformly on $[\fs_0,\fs_1]$,
\begin{align*}
 e^{-2s}N_1(s)N_2(s)=W_{1,n}W_{2,n}+o_{\pr}(1).
\end{align*}
Since $e^{2\fs_1}=ne^{2y}$ and $e^{2\fs_0}=n^{2\eps}$, integrating this uniform approximation yields
\begin{align*}
 \frac1n\int_{\fs_0}^{\fs_1} N_1(s)N_2(s)\,\dd s 
 &= \frac{1}{n} \int_{\fs_0}^{\fs_1} \bigl( W_{1,n}W_{2,n} + o_{\pr}(1) \bigr) e^{2s}\,\dd s \\
 &= \frac{W_{1,n}W_{2,n}}{2n}\bigl(e^{2\fs_1}-e^{2\fs_0}\bigr) + o_{\pr}(1)\frac{e^{2\fs_1}}{2n} \\
 &= \frac12 W_{1,n}W_{2,n} \bigl(e^{2y} - n^{2\eps-1}\bigr) + o_{\pr}(1)
 = \frac12 W_{1,n}W_{2,n} e^{2y} + o_{\pr}(1).
\end{align*}
Substituting this back into~\eqref{eq:hazard-after-reduction}, we obtain
\begin{align*}
 \gL_{\fs_1}=\theta W_{1,n}W_{2,n}e^{2y}+o_{\pr}(1),
\end{align*}
which proves~\eqref{eq:integrated-Wn}.
\end{proof}

\begin{lem}[Deterministic target]
\label{lem:macro-target-uniformity}
For every fixed $y\in\dR$,
\begin{align*}
 \pr\left(
 \fa_\ell T_{\dT_n}^{(\ell)}(0,\sfu_n)-\log n>y
 \right)
 \to
 \E\exp\{-\theta W_1W_2e^y\}.
\end{align*}
\end{lem}

\begin{proof}
We run the two-source exploration with roots at $0$ and $\sfu_n$. The population growth, stable-age, and ghost estimates are translation-invariant. Rooting the second proposal process at $\sfu_n$ merely multiplies its $k$-th spatial Fourier characteristic by the phase factor $e^{2\pi i k \sfu_n/n}$. Because this factor has modulus one, the Fourier second-moment bounds are identical. Since the Erd\H{o}s--Tur\'an bound in Proposition~\ref{prop:pair} depends only on these moduli, the pair-density and integrated-hazard estimates (Proposition~\ref{prop:pair} and Lemma~\ref{lem:integrated}) apply unchanged. The exact half-distance and Cox identities remain valid, so evaluating the Cox formula yields the identical limit.
\end{proof}

We now transfer the collision-time limit to the original passage time. This requires no new estimates: the Cox representation yields the tail of $\sfS_{\col}$, Lemma~\ref{lem:integrated} identifies the limiting Laplace exponent, and Proposition~\ref{prop:half} relates $\sfS_{\col}$ to $\fa_\ell T_n/2$.

\subsection{Proof of Theorem~\ref{thm:Macro}}\label{ssec:proof-Macro-main}
\begin{proof}[Proof of Theorem~\ref{thm:Macro}]
Write $W_1,W_2$ for two independent copies of $M_\infty$. We first treat the independent uniform target. Fix $\eps\in(0,1/2)$ and $y\in\dR$. By Proposition~\ref{prop:cox},
\begin{align*}
 \pr\bigl(\sfS_{\col}>\fs_1\bigr)
 =\E e^{-\gL_{\fs_1}}.
\end{align*}
Lemma~\ref{lem:integrated} and~\eqref{eq:Wn-joint-conv} give
\begin{align*}
 \gL_{\fs_1}
 \Rightarrow
 \theta W_1W_2e^{2y}.
\end{align*}
Since the map $x\mapsto e^{-x}$ is bounded and continuous,
\begin{align}\label{eq:macro-half-scale-tail}
 \pr\Bigl(
 \sfS_{\col}-\frac12\log n>y
 \Bigr)
 \to
 \E\exp\{-\theta W_1W_2e^{2y}\}.
\end{align}
For the uniform target, Proposition~\ref{prop:half} identifies $\sfS_{\col}=\fa_\ell T_n^{\unif}/2$. Substituting $y/2$ for $y$ in~\eqref{eq:macro-half-scale-tail} yields
\begin{align*}
 \pr\Bigl(
 \fa_\ell T_n^{\unif}-\log n>y
 \Bigr)
 \to
 \E\exp\{-\theta W_1W_2e^y\}.
\end{align*}
Lemma~\ref{lem:macro-target-uniformity} extends this identical limit to the deterministic-target variable $T_n^{\mathrm{fix}}=T_n$.

Let $E_3\sim\Exp(1)$ be independent of $(W_1,W_2)$ and put $G_3:=-\log E_3$. Conditional on $(W_1,W_2)$,
\begin{align*}
&\pr\left(
 -\log W_1-\log W_2-G_3-\log\theta>y
 \,\middle|\,W_1,W_2
\right)
\\
&\qquad=
\pr\left(E_3>\theta W_1W_2e^y\,\middle|\,W_1,W_2\right)
=
\exp\{-\theta W_1W_2e^y\}.
\end{align*}
This proves~\eqref{eq:macro-limit-representation}.

The weak convergence established above implies that $\fa_\ell T_n^\star/\log n \stackrel{\pr}{\to}1$ for $\star\in\{\mathrm{fix},\unif\}$. Since $\fa_\ell = \bigl(2\gl\Gamma(\theta+1)n\bigr)^{1/\theta}(1+o(1))$, we obtain
\begin{align*}
 \frac{n^{1/\theta}}{\log n}T_n^\star
 \stackrel{\pr}{\to}
 \bigl(2\gl\Gamma(\theta+1)\bigr)^{-1/\theta}.
\end{align*}

When $\theta=1$, the limiting CMJ process is the Yule process and $W_1,W_2$ are independent $\Exp(1)$ variables. Thus, setting $G_i:=-\log W_i$ for $i=1,2$ yields $S_\infty\equald G_1+G_2-G_3$.
\end{proof}

%%%%%%%%%%%%%%%%%%%%%%%%%%%%%%%%%%%%%%%%%%%%%%%%%%%%%%%%%%
\subsection{Macroscopic first-order crossover regime}\label{subsec:Macro-mean}

Theorem~\ref{thm:Macro} identifies the exact macroscopic first-order constant when $\ell/n\to\gl\in(0,1/2)$. In the larger window $n/\log n\ll\ell\le n/2$, the argument below gives the logarithmic order and the universal lower constant. The proof uses only path counting and the moment estimate; no triangular-array spatial-mixing statement is required.

\begin{lem}[Mittag--Leffler bound]\label{lem:ML-bound}
For every $\theta>0$ and every $\eta>0$, there exists
$C_{\theta,\eta}<\infty$ such that, for all $r\ge0$,
\begin{align}\label{eq:ML-bound}
 \sum_{k\ge0}
 \frac{r^{k\theta}}{\Gamma(k\theta+1)}
 \le
 C_{\theta,\eta}e^{(1+\eta)r}.
\end{align}
\end{lem}

\begin{proof}
For $0\le r\le1$, the left-hand side is bounded by $\sum_{k\ge0}\frac1{\Gamma(k\theta+1)}<\infty,$ where convergence follows from Stirling's formula. 
Now let $r\ge1$ and put $m_k:=\lfloor k\theta\rfloor$. For each $m\ge0$, the number of indices $k$ satisfying $m_k=m$ is at most $N_\theta:=\lceil1/\theta \rceil+1$.
If $m_k\ge1$, then $k\theta+1\ge m_k+1\ge2$. Since the Gamma function is
increasing on $[2,\infty)$,
\begin{align*}
 \Gamma(k\theta+1)\ge\Gamma(m_k+1)=m_k!,
 \quad
 r^{k\theta}\le r^{m_k+1}.
\end{align*}
For the finitely many indices with $m_k=0$, we have $0\le k\theta<1$, and therefore their total contribution is at most $C_\theta r$. Grouping the remaining terms according to $m_k$ gives
\begin{align*}
 \sum_{k\ge0}
 \frac{r^{k\theta}}{\Gamma(k\theta+1)}
 \le
 C_\theta r
 +N_\theta r\sum_{m\ge1}\frac{r^m}{m!} 
 \le
 C_\theta(1+r)e^r.
\end{align*}
Finally, $(1+r)e^r\le C_\eta e^{(1+\eta)r}$ for $r\ge0$, proving
\eqref{eq:ML-bound}.
\end{proof}

\begin{lem}[Uniform $L^p$ bound]\label{lem:macro-cycle-Lp}
For every fixed $p>1$, there exists $C_p<\infty$ such that, for all
$2\le\ell\le n/2$ and all $x\in\dT_n\setminus\{0\}$,
\begin{align*}
 \norm{ \fa_\ell T_n^{(\ell)}(0,x)}_p
 \le
 C_p\left(
 \frac{d_n(0,x)}{\ell}
 +\log\ell+1
 \right)
 \le
 C_p\left(
 \frac{n}{\ell}
 +\log\ell+1
 \right).
\end{align*}
\end{lem}

\begin{proof}
Let $m=d_n(0,x)\le n/2$, and choose a shortest oriented arc from $0$ to $x$. If $m\ge\ell$, restrict the optimization to the induced $\ell$-spread-out line subgraph on this arc. 
The endpoint estimate~\eqref{eq:greedy-endpoint-cmj} gives
\begin{align*}
 \norm{\fa_\ell T_n^{(\ell)}(0,x)}_p
 \le
 C_p\left(
 \frac m\ell+\log\ell+1
 \right).
\end{align*}
If $m<\ell$, choose an interval of $\ell+1$ consecutive cycle vertices containing both $0$ and $x$, so the induced graph is the complete graph $K_{\ell+1}$. 
Restricting the optimization to this complete subgraph and applying~\eqref{eq:complete-flooding}, together
with $\fa_\ell\asymp\ell^{1/\theta}$, gives
$
 \norm{\fa_\ell T_n^{(\ell)}(0,x)}_p
 \le C_p(\log\ell+1)$.
\end{proof}

\begin{prop}[Uniform right-crossover bounds]
\label{prop:macro-order-uniform}
Assume Assumption~\ref{ass:weight} and
\begin{align*}
 \liminf_{n\to\infty} \frac{\ell\log n}{n}>0, 
 \qquad
 \ell\le\frac n2.
\end{align*}
Then for every $\eps>0$,
\begin{align}\label{eq:macro-order-lower}
\sup_{x\in\dT_n\setminus\{0\}}
\pr\left(
 \frac{\fa_\ell T_n^{(\ell)}(0,x)}{\log n}\le1-\eps
\right)
\to0.
\end{align}
Moreover, for every fixed $p>1$,
\begin{align}\label{eq:macro-order-Lp}
\limsup_{n\to\infty}
\sup_{x\in\dT_n\setminus\{0\}}
\E\left[
 \left(
 \frac{\fa_\ell T_n^{(\ell)}(0,x)}{\log n}
 \right)^p
\right]
<\infty .
\end{align}
In particular, $\fa_\ell T_n/\log n$ is tight and bounded away from zero in probability, uniformly in the target, and
\begin{align}\label{eq:macro-order-mean-bounds}
 1
 \le
 \liminf_{n\to\infty}
 \inf_{x\in\dT_n\setminus\{0\}}
 \frac{\fa_\ell\E T_n^{(\ell)}(0,x)}{\log n}
 \le
 \limsup_{n\to\infty}
 \sup_{x\in\dT_n\setminus\{0\}}
 \frac{\fa_\ell\E T_n^{(\ell)}(0,x)}{\log n}
 <\infty.
\end{align}
\end{prop}

\begin{proof}
Fix $\eps\in(0,1)$ since~\eqref{eq:macro-order-lower} is immediate for $\eps\ge1$. 
Fix $x\ne0$ and $t\ge0$, and put $s:=\fa_\ell t$. Since the edge weights are positive almost surely, the event $\{T_n^{(\ell)}(0,x)\le t\}$ is realized by a simple path. The number of simple $k$-edge paths from $0$ to the vertex $x$ is at most $(2\ell)^{k-1}$.

Let $\go_1,\ldots,\go_k$ be independent copies of $\Exp(1)^{1/\theta}$. Since their density is $\theta x^{\theta-1}e^{-x^\theta}\ind_{\{x>0\}}$, the bound $e^{-x^\theta}\le1$ yields
\begin{align*}
 \pr\left(\sum_{j=1}^k\go_j\le t\right)
 \le
 \int_{\substack{x_j\ge0\\x_1+\cdots+x_k\le t}}
 \prod_{j=1}^k\theta x_j^{\theta-1}\,
 \dd x_1\cdots\dd x_k 
 =
 \frac{\Gamma(\theta+1)^k}{\Gamma(k\theta+1)}t^{k\theta}.
\end{align*}
Thus, using $\fa_\ell^\theta=(2\ell-1)\Gamma(\theta+1)$,
\begin{align}\label{eq:lower-tail-ub}
 \pr\bigl(T_n^{(\ell)}(0,x)\le t\bigr)
 &\le
 \frac1{2\ell}
 \sum_{k\ge1}
 \frac{
 \left(
 \tfrac{2\ell}{2\ell-1}s^\theta
 \right)^k
 }{\Gamma(k\theta+1)}.
\end{align}
Choose $\eta>0$ so small that $(1+\eta)^2(1-\eps)<1$. Since $c_\ell:=\left(\frac{2\ell}{2\ell-1}\right)^{1/\theta}\to 1,$ we have $c_\ell\le1+\eta$ for all sufficiently large $n$. Applying Lemma~\ref{lem:ML-bound} with $r=c_\ell s$ to~\eqref{eq:lower-tail-ub} gives
\begin{align*}
 \sup_{x\ne0}\pr\bigl(T_n^{(\ell)}(0,x)\le t\bigr)
 \le
 \frac{C_{\theta,\eta}}{\ell}
 \exp\big((1+\eta)c_\ell s\big).
\end{align*}
Taking $s=(1-\eps)\log n$, we obtain uniformly in $x\ne0$,
\begin{align*}
 \pr\left(
 \frac{\fa_\ell T_n^{(\ell)}(0,x)}{\log n}\le1-\eps
 \right)
 \le
 C_{\theta,\eta}\frac{n^{(1+\eta)^2(1-\eps)}}{\ell}
 \to0,
\end{align*}
since $\ell\ge cn/\log n$ for all sufficiently large $n$. This proves~\eqref{eq:macro-order-lower}.

Since $n/\ell=O(\log n)$ by hypothesis, and $\log\ell\le\log n$, Lemma~\ref{lem:macro-cycle-Lp} yields~\eqref{eq:macro-order-Lp}. Tightness follows from Markov's inequality. The lower-tail estimate shows that the family is bounded away from zero in probability.
Finally, non-negativity and~\eqref{eq:macro-order-lower} imply, for every $\eps\in(0,1)$,
\begin{align*}
 \liminf_{n\to\infty}
 \inf_{x\ne0}
 \frac{\fa_\ell\E T_n^{(\ell)}(0,x)}{\log n}
 \ge
 (1-\eps)
 \liminf_{n\to\infty}
 \inf_{x\ne0}
 \pr\left(
 \frac{\fa_\ell T_n^{(\ell)}(0,x)}{\log n}>1-\eps
 \right)
 =1-\eps.
\end{align*}
Letting $\eps\downarrow0$ gives the first inequality in~\eqref{eq:macro-order-mean-bounds}; the last inequality follows from~\eqref{eq:macro-order-Lp} and H\"older's inequality.
\end{proof}

\subsection{Proof of Theorem~\ref{thm:macro-cross}}
Part~\eqref{cross:meso} is Theorem~\ref{thm:meso-main}\;\eqref{meso-mean}. Part~\eqref{cross:right}) follows from Proposition~\ref{prop:macro-order-uniform} and $\fa_\ell\asymp\ell^{1/\theta}$. For part~\eqref{cross:macro}, Theorem~\ref{thm:Macro} gives convergence in probability, while~\eqref{eq:macro-order-Lp} gives uniform integrability and hence convergence of the means.\qed

%%%%%%%%%%%%%%%%%%%%%%%%%%%%%%%%%%%%%%%%%%%%%%%%%%%
\section{Hop-count analysis: proof of Theorem~\ref{thm:hop-count-asymptotics}}\label{sec:hop-count-fluct}

In this section we study the number of edges in the geodesic. Since the
edge-weight distribution is continuous and the graph is finite, geodesics are
almost surely unique (see $\Omega_*$ in Section~\ref{sec:collision}).
For the relevant passage time $T_n$, we write
\begin{align*}
 H_n:=|\pi_n^\star|
\end{align*}
for the number of edges in its unique geodesic $\pi_n^\star$. The first-order behavior is different in the two regimes. In the mesoscopic regime, the geodesic follows the critical spatial CMJ direction and has order $\sfu_n/\ell\asymp n/\ell$. 
On the other hand, in the macroscopic regime, the geodesic is obtained by joining two CMJ ancestral lines at the first collision bridge and has order $\log n$. We identify the leading constants in both regimes. The corresponding second-order hop-count limits are left as open problems and discussed in Section~\ref{sec:open}.

%%%%%%%%%%%%%%%%%%%%%%%%%%%%%%%%%%%%%%%%%%%%%%%%%%%
\subsection{Mesoscopic expected hop-count}\label{ssec:meso-hop}

Recall that $\gf(\beta)=\E e^{\beta\xi}$ for $\xi\sim\Unif[-1,1]$ and
\begin{align*}
 \gk_\theta(\beta)=\gf(\beta)^{1/\theta},
 \qquad
 \vstar=\inf_{\beta>0}\frac{\gk_\theta(\beta)}{\beta},
\end{align*}
and that the unique minimizer $\gbc$ satisfies $ \gbc\gf'(\gbc)=\theta\gf(\gbc)$.

\begin{lem}
\label{lem:meso-ld-bound}
For every $\eps>0$, there exist $\gd,c>0$, $C<\infty$, and $\ell_0<\infty$ such that, for all $\ell\ge\ell_0$ and $f\ge1$,
\begin{align}\label{eq:bivariate-hop-tree-bound}
\E\sum_{v\in
\cN_{\left(1/\vstar+\gd\right)f}^{(\ell)}}
\ind\left\{
 |\sfp_\ell(v)|\ge f,\ 
 \left|
 \frac{|v|}{f}-\frac{\gbc}{\theta}
 \right|>\eps
\right\}
\le Ce^{-cf}.
\end{align}
\end{lem}

\begin{proof}
By symmetry, it suffices to consider lineages with $\sfp_\ell(v)\ge f$. For $q>0$, define the single-step moment generating function
\begin{align*}
 K_\ell(q) := \sum_{z\in\cD_\ell} e^{\gbc z/\ell}\E e^{-q\fa_\ell\go_z}.
\end{align*}
By~\eqref{eq:finite-weighted-laplace}, $K_\ell(q)\to \gf(\gbc)q^{-\theta}$ for each fixed $q>0$. Setting $q_\star:=\gk_\theta(\gbc)=\gbc\vstar$ and $r_\star:=\gbc/\theta$, we define the asymptotic rate function
\begin{align*}
 \Phi(q,r) := -\gbc+\frac q{\vstar} + r\log\bigl(\gf(\gbc)q^{-\theta}\bigr).
\end{align*}
Because $\gf(\gbc)q_\star^{-\theta}=1$, we have $\Phi(q_\star,r)=0$ for all $r$, with derivative $\partial_q\Phi(q_\star,r) = \frac{1-\theta r/\gbc}{\vstar}$.

For the upper-generation tail, since $\partial_q\Phi(q_\star,r_\star+\eps) < 0$, we may choose $q_+>q_\star$ such that $\Phi(q_+,r_\star+\eps)<0$ and $\gf(\gbc)q_+^{-\theta}<1$. By continuity, for $\gd>0$ chosen sufficiently small and all sufficiently large $\ell$, we guarantee $K_\ell(q_+)<1$ and 
\begin{align*}
 -\gbc+q_+\left(\frac1{\vstar}+\gd\right) +(r_\star+\eps)\log K_\ell(q_+) \le -2c
\end{align*}
for some $c>0$. The branching property yields $\E\sum_{|v|=m}e^{\gbc \sfp_\ell(v)-q\sft_\ell(v)} = K_\ell(q)^m$. Applying an exponential Markov inequality over the relevant generations gives
\begin{align*}
 \E\sum_{v\in \cN_{\left(1/\vstar+\gd\right)f}^{(\ell)}}
 \ind\left\{ \sfp_\ell(v)\ge f,\ |v|\ge (r_\star+\eps)f \right\}
 &\le
 e^{-\gbc f+q_+(1/\vstar+\gd)f}
 \sum_{m\ge (r_\star+\eps)f} K_\ell(q_+)^m \\
 &\le Ce^{-cf}.
\end{align*}

Similarly, for the lower-generation tail (assuming $r_\star-\eps>0$), $\partial_q\Phi(q_\star,r_\star-\eps) > 0$ allows us to select $q_-<q_\star$ such that $\Phi(q_-,r_\star-\eps)<0$ and $\gf(\gbc)q_-^{-\theta}>1$. For small $\gd>0$ and large $\ell$, we obtain $K_\ell(q_-)>1$ with an analogous rate strictly bounded by $-2c$. Summing the exponential moment $K_\ell(q_-)^m$ over $m \le (r_\star-\eps)f$ symmetrically yields the $Ce^{-cf}$ bound. If $r_\star-\eps\le0$, the lower-generation event is trivially empty.
Same bounds apply to the reflected event $\{\sfp_\ell(v)\le -f\}$ by symmetry, proving~\eqref{eq:bivariate-hop-tree-bound}.
\end{proof}

\begin{proof}[Proof of Theorem~\ref{thm:hop-count-asymptotics}~\eqref{hop:meso}]
Set $f_n:=\sfu_n/\ell$. Then $f_n\to\infty$ and $\log\ell=o(f_n)$. Fix $\eps>0$, and let $\gd$ be given by Lemma~\ref{lem:meso-ld-bound}. 
Equation~\eqref{eq:meso-mean} in Theorem~\ref{thm:meso-main} gives
\begin{align}\label{eq:ld-ub-Tn}
 \pr\left(
 \fa_\ell T_n>
 (1/\vstar+\gd)f_n
 \right)
 \to 0.
\end{align}

On the complementary event, we lift the simple cycle geodesic $\pi_n^\star$ to $\dZ$. As in the proof of~\eqref{eq:global-meso-lower-tail}, a lifted simple path with increments $(z_1,\dots,z_k)\in\cD_\ell^k$ is identified with the tree vertex $v=(z_1,\dots,z_k)$; the path is simple, so its edge weights are i.i.d.~and its scaled cost has exactly the law of $\sft_\ell(v)$, while its hop-count equals $|v|$ and its scaled displacement equals $\sfp_\ell(v)$. Combining a union bound over the tree vertices with~\eqref{eq:ld-ub-Tn} and Lemma~\ref{lem:meso-ld-bound}, we obtain
\begin{align*}
&\pr\left(
 \left|\frac{H_n}{f_n}-\frac{\gbc}{\theta}\right|>\eps
\right) \\
&\quad\le
\pr\left(
 \fa_\ell T_n>
 \left(\frac1{\vstar}+\gd\right)f_n
\right)
+
\E\sum_{v\in \cN_{\left(1/\vstar+\gd\right)f_n}^{(\ell)}}
\ind\left\{
 |\sfp_\ell(v)|\ge f_n,\ 
 \left|\frac{|v|}{f_n}-\frac{\gbc}{\theta}\right|>\eps
\right\} \\
&\quad\le o(1)+Ce^{-cf_n}
\to 0.
\end{align*}

It remains to pass to expectations. Choose $q>1$. Since $2\ell\E e^{-q\fa_\ell\go}\to q^{-\theta}<1,$ there are $c_0,c_1>0$ such that, for all large $\ell$,
\begin{align*}
 e^{qc_0}\,2\ell\E e^{-q\fa_\ell\go}\le e^{-c_1}.
\end{align*}
If $H_n\ge k$ and $\fa_\ell T_n\le c_0k$, the first $k$ edges of $\pi_n^\star$ form a simple path of scaled weight at most $c_0k$. There are at most $(2\ell)^k$ such sequences of directions; hence, a direct union bound and the exponential Markov inequality yield
\begin{align*}
 \pr(H_n\ge k)
 &\le
 e^{qc_0k}
 \left(2\ell\E e^{-q\fa_\ell\go}\right)^k
 +\pr(\fa_\ell T_n>c_0k) \\
 &\le
 e^{-c_1k}+\frac{Cf_n^2}{k^2}.
\end{align*}
Indeed, $T_n\le T_{[\sfu_n]}$ by~\eqref{eq:cycle-seg-upper}, and Proposition~\ref{prop:finite-graph-window-tools}~\textup{(i)}, applied with $r=\sfu_n$, gives
\begin{align*}
 \norm{\fa_\ell T_n}_2
 \le
 C(f_n+\log\ell+1)
 \le Cf_n.
\end{align*}
Thus, for $Y_n:=H_n/f_n$ and $M\ge1$, the identity $\E\bigl[Y_n\ind\{Y_n>M\}\bigr] = M\pr(Y_n>M) +\int_M^\infty\pr(Y_n>x)\,\dd x $
yields, for all large $n$,
\begin{align*}
 \E\bigl[Y_n\ind\{Y_n>M\}\bigr]
 &\le
 \left(M+\frac1{c_1f_n}\right)e^{-c_1Mf_n}
 +\frac{C}{M}.
\end{align*}
The right-hand side tends to zero as $M\to\infty$, uniformly for all large $n$. Hence $Y_n$ is uniformly integrable. Together with the convergence in probability proved above, this establishes the mean asymptotics in~\eqref{eq:meso-hop-mean}.
\end{proof}

\begin{rem}[Mesoscopic hop-count CLT]
\label{rem:meso-hop-clt-open}
Theorem~\ref{thm:hop-count-asymptotics}~\eqref{hop:meso} identifies the first-order constant, but we do not prove a mesoscopic hop-count CLT here. A natural conjectural form is
\begin{align*}
 \frac{H_n^{(\ell)}(0,\sfu_n)-\E H_n^{(\ell)}(0,\sfu_n)}{\sqrt{\sfu_n/\ell}}
 \Rightarrow
 \N(0,\gs_{\mathrm{meso\text{-}hop}}^2)
\end{align*}
for some $\gs_{\mathrm{meso\text{-}hop}}^2\in(0,\infty)$.
\end{rem}

%%%%%%%%%%%%%%%%%%%%%%%%%%%%%%%%%%%%%%%%%%%%%%%%%%%
\subsection{Macroscopic hop-count law of large numbers}
\label{ssec:macro-hop-lln}

Throughout this subsection, assume
\begin{align*}
 \frac{\ell}{n}\to\gl\in(0,1/2),
 \qquad
 \go_e^\theta\sim\Exp(1),
\end{align*}
and use the two-source censored exploration from Sections~\ref{sec:macro-cox} and~\ref{sec:macro-cmj}. For $i\in\{1,2\}$ and $v\in\cC_i(s)$, let $\Gen_i(v)$ be the generation of $v$ in the censored Dijkstra tree rooted at source $i$.

Under the change of measure, the successive birth ages along an ancestral line have law
\begin{align*}
 \nu(\dd a)
 :=e^{-a}\mu_\theta(\dd a)
 =\frac{a^{\theta-1}e^{-a}}{\Gamma(\theta)}\,\dd a.
\end{align*}
Thus, $\nu$ is the $\Gamma(\theta,1)$ distribution with mean $\theta$. The only additional estimate needed for the hop-count is that, at time $s$, almost all of the bridge hazard is carried by endpoints whose generations are tightly concentrated around $s/\theta$. 

\begin{lem}[Bridge-generation concentration]
\label{lem:macro-bridge-gen-conc}
Fix $\eps\in(0,1/2)$ and $y\in\dR$. Recall the definition of $\fs_0, \fs_1$ in~\eqref{eq:s0-s1-def}. For $\gd\in(0,1/\theta)$ and $i\in\{1,2\}$, define
\begin{align*}
 \Bad_{i,n}(s;\gd)
 :=
 \left\{
 v\in\cC_i(s):
 \left|\Gen_i(v)-\frac{s}{\theta}\right|>\gd s
 \right\}.
\end{align*}
Then
\begin{align}\label{eq:macro-bridge-gen-conc}
\int_{\fs_0}^{\fs_1}\frac1n
\sum_{\substack{e=\{u,v\}\in\cB_s\\
 u\in\cC_1(s),\ v\in\cC_2(s)}}
A_e(s)^{\theta-1}
\ind\big\{
 u\in\Bad_{1,n}(s;\gd)
\text{ or }
 v\in\Bad_{2,n}(s;\gd)
\big\}
\,\dd s
\stackrel{\pr}{\to}0.
\end{align}
\end{lem}

\begin{proof}
Recall the finite-$n$ offspring-age measure $\bar\mu_n$ and the law $\nu_n(\dd a):=e^{-\alpha_n a}\bar\mu_n(\dd a)$ from~\eqref{eq:prelimit-spine-law}. By Lemma~\ref{lem:finite-n-cmj-martingale}, $\alpha_n=1+O(n^{-1})$, $\theta_n:=\int_0^\infty a\,\nu_n(\dd a)\to\theta$, and the measures $\nu_n$ share a common exponential moment.

For the $i$-th proposal process, define the corresponding bad set by
\begin{align*}
 \Bad_{i,n}^{\op}(s;\gd)
 :=
 \left\{
 u\in\sT:\tau_u\le s,\
 \left||u|-\frac{s}{\theta}\right|>\gd s
 \right\}.
\end{align*}
The superscript $\op$ distinguishes the proposal process from the accepted cluster. Since the expected birth-time measure in generation $m$ is $\bar\mu_n^{*m}$, we have
\begin{align}
 e^{-\alpha_n s}\E\abs{\Bad_{i,n}^{\op}(s;\gd)}
 =
 \sum_{\substack{m\ge0:\\ |m-s/\theta|>\gd s}}
 \int_{[0,s]}
 e^{-\alpha_n(s-u)}\nu_n^{*m}(\dd u).
\label{eq:proposal-bad-generation-mean}
\end{align}
We uniformly bound the right-hand side using Chernoff's inequality. Let $X_{n,1}, X_{n,2}, \ldots$ be i.i.d.~with law $\nu_n$, and set $S_{n,m}:=\sum_{j=1}^mX_{n,j}$. 

For the upper-generation range, let $c_+:=1/\theta+\gd$. Since $\E X_{n,1} = \theta_n \to \theta > 1/c_+$, there exist $t_+>0$ and $q_+\in(0,1)$ such that $e^{t_+/c_+}\E e^{-t_+ X_{n,1}} \le q_+$ uniformly for all large $n$. Hence, for $m\ge c_+s$,
\begin{align*}
 \pr(S_{n,m}\le s)
 \le e^{t_+s}\bigl(\E e^{-t_+X_{n,1}}\bigr)^m
 \le q_+^m.
\end{align*}
Summing over $m\ge c_+s$ yields an $O(e^{-cs})$ bound.

Similarly, for the lower-generation range, let $c_-:=1/\theta-\gd$ and choose $\eta\in(0,\gd\theta)$. Since $(1-\eta)/c_- > \theta$, there exist $t_->0$ and $b>0$ such that $e^{-t_-(1-\eta)} (\E e^{t_- X_{n,1}})^{c_-} \le e^{-b}$ uniformly for all large $n$. Splitting the integration domain in~\eqref{eq:proposal-bad-generation-mean} at $u=(1-\eta)s$, we obtain
\begin{align*}
 \sum_{m\le c_-s}
 \int_{[0,(1-\eta)s]}
 e^{-\alpha_n(s-u)}\nu_n^{*m}(\dd u)
 &\le
 (c_-s+1)e^{-\alpha_n\eta s},
 \\
 \sum_{m\le c_-s}
 \int_{((1-\eta)s,s]}
 e^{-\alpha_n(s-u)}\nu_n^{*m}(\dd u)
 &\le
 \sum_{m\le c_-s}
 \pr\bigl(S_{n,m}\ge(1-\eta)s\bigr)
 \le
 (c_-s+1)e^{-bs}.
\end{align*}
Combining both generation ranges, there are constants $c,C>0$ such that
\begin{align}
 e^{-\alpha_n s}\E\abs{\Bad_{i,n}^{\op}(s;\gd)}
 \le Ce^{-cs},
 \qquad s\ge1,
\label{eq:proposal-bad-generation-bound}
\end{align}
uniformly for all large $n$.

Since every accepted vertex is a proposal particle with the identical generation, $\abs{\Bad_{i,n}(s;\gd)} \le \abs{\Bad_{i,n}^{\op}(s;\gd)}$.
Put
\begin{align*}
 K_{i,n}
 :=
 \left(
 \inf_{\fs_0\le s\le \fs_1}e^{-s}N_i(s)
 \right)^{-1}.
\end{align*}
By~\eqref{eq:cmj-growth-package},~\eqref{eq:Wn-joint-conv}, and the strict positivity of $M_\infty$, $K_{i,n}=O_{\pr}(1)$. Moreover, $\alpha_n=1+O(n^{-1})$ and $\fs_1=O(\log n)$ imply $\sup_{\fs_0\le s\le \fs_1}e^{(\alpha_n-1)s}=1+o(1)$. Therefore,
\begin{align*}
\int_{\fs_0}^{\fs_1}
 \frac{\abs{\Bad_{i,n}(s;\gd)}}{N_i(s)}
 \,\dd s 
\le
 \bigl(1+o(1)\bigr)K_{i,n}
 \int_{\fs_0}^{\fs_1}
 e^{-\alpha_ns}\abs{\Bad_{i,n}^{\op}(s;\gd)}
 \,\dd s.
\end{align*}
By Tonelli's theorem and~\eqref{eq:proposal-bad-generation-bound}, the expectation of the integral on the right is bounded by $C\int_{\fs_0}^{\infty}e^{-cs}\,\dd s \to 0$. Since $K_{i,n}=O_{\pr}(1)$, Markov's inequality implies
\begin{align}
 \int_{\fs_0}^{\fs_1}
 \frac{\abs{\Bad_{i,n}(s;\gd)}}{N_i(s)}
 \,\dd s
 \stackrel{\pr}{\to}0.
\label{eq:macro-generation-fraction}
\end{align}

Let $\fB_s^{\mathrm{bad}}(\gd)$ be the number of bridges with at least one bad endpoint. We can bound
\begin{align*}
 \fB_s^{\mathrm{bad}}(\gd)
 \le
 \abs{\Bad_{1,n}(s;\gd)}N_2(s)
 +
 \abs{\Bad_{2,n}(s;\gd)}N_1(s).
\end{align*}
Moreover,~\eqref{eq:cmj-growth-package} gives $\sup_{\fs_0\le s\le \fs_1} \frac{N_1(s)N_2(s)}{n} =O_{\pr}(1)$, so~\eqref{eq:macro-generation-fraction} yields
\begin{align}
 \int_{\fs_0}^{\fs_1}
 \frac{\fB_s^{\mathrm{bad}}(\gd)}{n}\,\dd s
 \stackrel{\pr}{\to}0.
\label{eq:macro-bad-bridge-count}
\end{align}

Fix $0<r<R<\infty$. On $\{r\le A_e(s)\le R\}$, the kernel $A_e(s)^{\theta-1}$ is bounded, so~\eqref{eq:macro-bad-bridge-count} guarantees
\begin{align*}
 \int_{\fs_0}^{\fs_1}\frac1n
 \sum_{\substack{e\in\cB_s:\\ e\text{ has a bad endpoint}}}
 A_e(s)^{\theta-1}
 \ind\{r\le A_e(s)\le R\}
 \,\dd s
 \stackrel{\pr}{\to}0.
\end{align*}
The small- and large-age tail components are controlled exactly as in the proof of Proposition~\ref{prop:pair}. Letting first $n\to\infty$, then $R\to\infty$, and finally $r\downarrow0$ completes the proof of~\eqref{eq:macro-bridge-gen-conc}.
\end{proof}

With this generation concentration, we are now ready to prove the macroscopic hop-count limit.

\begin{proof}[Proof of Theorem~\ref{thm:hop-count-asymptotics}\;\eqref{hop:macro}]
Throughout this proof, we write $H_n:=H_n^{(\ell)}(0,U_n)$.
Let
\begin{align*}
 e_{\col}=\{V_{1,n},V_{2,n}\},
 \qquad
 V_{i,n}\in\cC_i(\sfS_{\col}),
\end{align*}
be the first collision bridge. By Proposition~\ref{prop:half} and the identification~\eqref{eq:proof-bridge-equals-collision} in the proof of Proposition~\ref{prop:cox}, $e_{\col}$ is the unique edge of the geodesic containing its midpoint. On $\Omega_*$, both endpoints are therefore discovered strictly before the collision time. Proposition~\ref{prop:censor} then identifies the two censored-tree paths with the corresponding geodesic subpaths, yielding
\begin{align}
 H_n
 =
 \Gen_1(V_{1,n})+
 \Gen_2(V_{2,n})+1.
\label{eq:macro-hop-generation-sum}
\end{align}

Fix $\eps\in(0,1/2)$, $\gd\in(0,1/\theta)$, and $A>0$, and put
\begin{align*}
 I_A
 :=
 \left[
 \frac12\log n-A,
 \frac12\log n+A
 \right].
\end{align*}
For all sufficiently large $n$, $I_A\subseteq[\fs_0,\fs_1]$ when Lemma~\ref{lem:macro-bridge-gen-conc} is applied with $y=A$. For $e=\{u,v\}$ with $u\in\cC_1(s)$ and $v\in\cC_2(s)$, set
\begin{align*}
 \Psi(e,s)
 :=
 \ind\{s\in I_A\} \cdot
 \ind\big\{
 u\in\Bad_{1,n}(s;\gd)
 \text{ or }
 v\in\Bad_{2,n}(s;\gd)
 \big\}.
\end{align*}
Taking the left-continuous version of $\Psi$ ensures predictability without affecting the Lebesgue integral, since the collision time is almost surely not a vertex-addition time. Applying the marked Cox representation, Proposition~\ref{prop:cox}~\eqref{cox:marked}, to this $\Psi$ gives
\begin{align}
\notag
&\pr\left(
 \sfS_{\col}\in I_A,
 \ V_{1,n}\in\Bad_{1,n}(\sfS_{\col};\gd)
 \text{ or }
 V_{2,n}\in\Bad_{2,n}(\sfS_{\col};\gd)
\right)
\\
&\qquad=
\E\int_0^\infty
 e^{-\gL_s}
 \frac{2\theta}{\fa_\ell^\theta}
 \sum_{e\in\cB_s}
 A_e(s)^{\theta-1}\Psi(e,s)
 \,\dd s.
\label{eq:collision-endpoint-generation-good}
\end{align}
The random integral on the right converges to zero in probability by Lemma~\ref{lem:macro-bridge-gen-conc}, since $\fa_\ell^\theta\asymp n$ and $e^{-\gL_s}\le1$. Moreover, pathwise,
\begin{align*}
 0
 \le
 \int_0^\infty
 e^{-\gL_s}
 \frac{2\theta}{\fa_\ell^\theta}
 \sum_{e\in\cB_s}A_e(s)^{\theta-1}\Psi(e,s)\,\dd s
 \le
 \int_0^\infty e^{-\gL_s}\,\dd\gL_s
 \le1.
\end{align*}
Hence the random integrals are uniformly integrable, and their convergence to zero in probability implies convergence to zero in $L^1$. Therefore the probability in~\eqref{eq:collision-endpoint-generation-good} converges to zero.

The random-target collision limit~\eqref{eq:macro-half-scale-tail} states that $\sfS_{\col}-\frac12\log n=O_{\pr}(1)$. Letting first $n\to\infty$ in~\eqref{eq:collision-endpoint-generation-good} and then $A\to\infty$, we deduce that for each fixed $\gd\in(0,1/\theta)$, with probability tending to one,
\begin{align*}
 \left|
 \Gen_i(V_{i,n})-
 \frac{\sfS_{\col}}{\theta}
 \right|
 \le
 \gd\sfS_{\col},
 \qquad i\in\{1,2\}.
\end{align*}
Using~\eqref{eq:macro-hop-generation-sum},
\begin{align*}
 \abs{
 \frac{H_n}{\log n}-\frac1\theta
 }
 \le
 \frac{2\gd\sfS_{\col}+1}{\log n}
 +
 \frac1\theta
 \abs{
 \frac{2\sfS_{\col}}{\log n}-1
 }.
\end{align*}
Since $2\sfS_{\col}/\log n\stackrel{\pr}{\to}1$, the second term converges to zero in probability and the first is $\gd+o_{\pr}(1)$. Letting $\gd\downarrow0$ proves
\begin{align}\label{eq:macro-hop-probability}
 \frac{H_n}{\log n}
 \stackrel{\pr}{\to}
 \frac1\theta.
\end{align}

To establish uniform integrability, fix a sufficiently large $q>0$ and a sufficiently small $c_0>0$. Since $\go^\theta\sim\Exp(1)$,
\begin{align*}
 \E e^{-q\fa_\ell\go}
 =
 \int_0^\infty
 e^{-q\fa_\ell t}
 \theta t^{\theta-1}e^{-t^\theta}
 \,\dd t 
 \le
 \int_0^\infty
 e^{-q\fa_\ell t}
 \theta t^{\theta-1}
 \,\dd t 
 =
 \frac{\Gamma(\theta+1)}
 {(q\fa_\ell)^\theta}.
\end{align*}
We can thus select $q$ and $c_0$ such that, for some $\rho\in(0,1)$ and all large $n$,
\begin{align}
 e^{qc_0}(2\ell)\E e^{-q\fa_\ell\go}
 \le\rho.
\label{eq:macro-hop-path-laplace}
\end{align}

For any integer $k\ge1$, on the event $\big\{ H_n\ge k,\ \fa_\ell T_n^{\unif}\le c_0k \big\}$, the first $k$ edges of the geodesic form a simple path of rescaled weight at most $c_0k$. Since there are at most $(2\ell)^k$ such paths, a union bound combined with the exponential Markov inequality and~\eqref{eq:macro-hop-path-laplace} yields
\begin{align}\label{eq:macro-hop-path-tail}
 \pr\big(
 H_n\ge k,\
 \fa_\ell T_n^{\unif}\le c_0k
 \big)
 \le
 (2\ell)^k
 e^{qc_0k}
 \left(\E e^{-q\fa_\ell\go}\right)^k 
 \le
 \rho^k.
\end{align}

Since $\ell/n\to\gl\in(0,1/2)$, applying~\eqref{eq:macro-order-Lp} with $p=2$, followed by conditioning on $U_n$, gives
\begin{align}
 \E \bigl(\fa_\ell T_n^{\unif}\bigr)^2
 \le
 C(\log n)^2.
\label{eq:macro-random-target-L2}
\end{align}
Combining~\eqref{eq:macro-hop-path-tail},~\eqref{eq:macro-random-target-L2}, and Markov's inequality, we obtain constants $c,C>0$ such that
\begin{align}
 \pr(H_n\ge k)
 \le
 e^{-ck} + C\frac{(\log n)^2}{k^2},
 \quad k\ge1.
\label{eq:macro-hop-uniform-tail}
\end{align}

Setting $Y_n:=H_n/\log n$ and applying~\eqref{eq:macro-hop-uniform-tail} with $k=\lceil x\log n\rceil$ yields, for all $x\ge1$ and large $n$ (after adjusting the constants $c,C$),
\begin{align*}
 \pr(Y_n>x)
 \le
 e^{-cx\log n}
 +
 \frac{C}{x^2}.
\end{align*}
Therefore, for $M\ge1$,
\begin{align*}
 \E\left[
 Y_n\ind\{Y_n>M\}
 \right]
 &=
 M\pr(Y_n>M)
 +
 \int_M^\infty\pr(Y_n>x)\,\dd x 
 \\
 &\le
 M e^{-cM\log n}
 +
 \frac{C}{M}
 +
 \frac{e^{-cM\log n}}{c\log n}
 +
 \frac{C}{M}.
\end{align*}
The right-hand side vanishes as $M\to\infty$ uniformly over all sufficiently large $n$. The finitely many remaining indices do not affect uniform integrability. Thus the family $\{H_n/\log n\}_{n\ge1}$ is uniformly integrable. Combining this with~\eqref{eq:macro-hop-probability} yields
\begin{align*}
 \E\abs{
 \frac{H_n}{\log n}-\frac1\theta
 }
 \to0,
\end{align*}
which implies the expectation convergence in~\eqref{eq:macro-hop-lln-main}. This completes the proof.
\end{proof}

%%%%%%%%%%%%%%%%%%%%%%%%%%%%%%%%%%%%%%%%%%%%%%%%%%%%%%%%%
\section{Discussion and open problems}\label{sec:open}

We conclude with several open problems concerning sharp fluctuations, the spatial-to-mean-field crossover, and the structural robustness of the CMJ mechanisms.

\begin{enumeraten}
	\item \textbf{Numerical support for mesoscopic conjectures}
For $\theta=1$ and $\sfu_n=\lfloor n/3\rfloor$, we simulated the first-passage time $T_n$ using Dijkstra's algorithm. For the variance experiment, we set $\ell=\lfloor n^\ga\rfloor$ for $\ga\in\{0.25, 0.30, 0.35, 0.40\}$ and $2^{14}\le n\le 2^{22}$, drawing $750$--$6000$ independent samples per parameter pair. If the variance scaled strictly as $n/\ell^3$, the ratio $R_n:=\ell^3\cdot \var(T_n)/n$ would be approximately constant. Instead, Figure~\ref{fig:meso-variance-numerics} exhibits a slow decay across all trajectories, and the data points nearly collapse when plotted against $\ell$. This suggests
\begin{align*}
 \var(T_n)\asymp {n}{\ell^{-3+o(1)}}.
\end{align*}

\begin{figure}[htbp]
 \centering
 \includegraphics[width=0.96\linewidth]
 {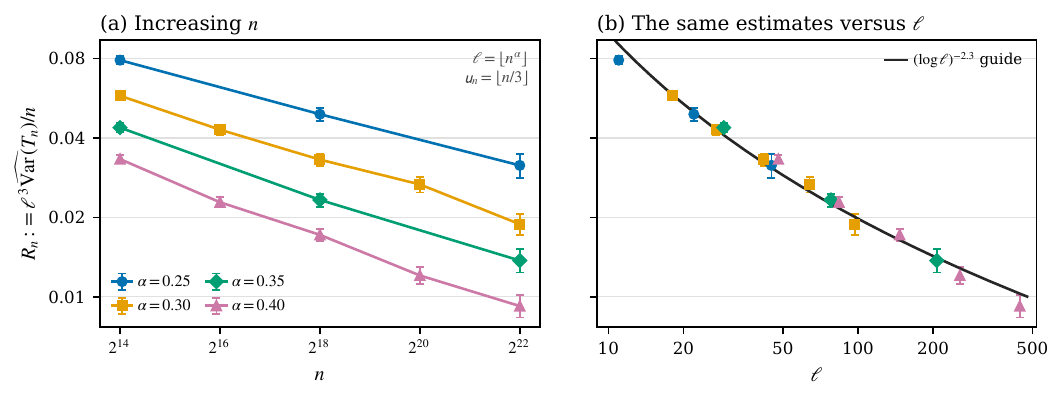}
 \caption{Variance scaling for FPP with $\Exp(1)$ weights.
 Panel~\textup{(a)} shows $R_n$ along four rays $\ell=\lfloor n^\ga\rfloor$ as $n$ increases.
 Panel~\textup{(b)} replots exactly the same estimates against $\ell$, using the same colors and markers. Their approximate collapse suggests that the remaining correction is mainly a function of $\ell$. The black curve is the finite-range guide $C(\log\ell)^{-2.3}$, not an asymptotic fit. Vertical error bars are cellwise $95\%$ bootstrap intervals.}
 \label{fig:meso-variance-numerics}
\end{figure}

For the distributional experiment, we set $\ga=0.60$, $n=2^{22}$, and $\ell=9410$ with $1500$ independent samples (yielding $n/\ell\simeq 446$). The empirical skewness ($0.026$), excess kurtosis ($0.012$), and Kolmogorov--Smirnov distance to a fitted normal ($0.016$) all fall strictly within the 95\% Gaussian reference bounds. The histogram closely tracks the standard normal density, and all Q--Q points in Figure~\ref{fig:meso-clt-numerics} lie within the simultaneous $95\%$ Gaussian envelope. For $\ga=0.70$, the largest simulation yields only $n/\ell\simeq 97$ blocks, rendering this regime numerically inconclusive. Combined with Remark~\ref{rem:meso-clt-window}, this  provides exploratory evidence for the conjecture
\begin{align*}
 \frac{T_n-\E T_n}{\sqrt{\var(T_n)}}
 \Rightarrow \N(0,1)
 \text{ whenever }\ell\log n=o(n),
\end{align*}
covering the near-linear window $\ell\ll n/\log n=n^{1-o(1)}$.

\begin{figure}[htbp]
 \centering
 \includegraphics[width=0.96\linewidth]
 {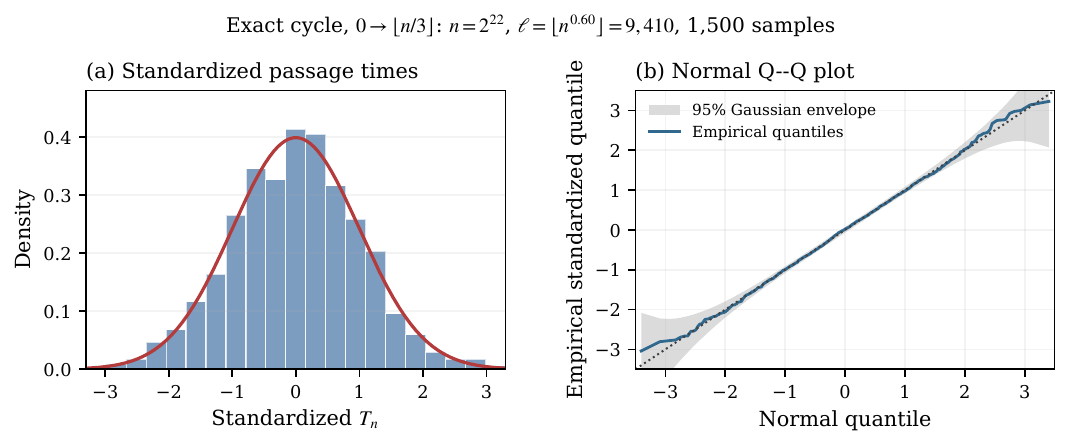}
 \caption{Standardized first-passage times for
 $\ell=\lfloor n^{0.60}\rfloor$, $n=2^{22}$, and $1500$ samples.
 In panel~\textup{(a)}, the bars are the empirical density and the red curve is the standard Gaussian density. Panel~\textup{(b)} is the normal Q--Q plot with a simultaneous $95\%$ Gaussian envelope.}
 \label{fig:meso-clt-numerics}
\end{figure}

 \item \textbf{The mesoscopic--macroscopic crossover.}
Theorem~\ref{thm:macro-cross} locates the first-order transition at $\ell\asymp {n}/{\log n}$, but does not identify the critical law. One can analyze $
 \ell=\frac{n}{\log n}\cdot c_n$ when $c_n$ tends to $0$, a positive constant, or $\infty$, and determine the interpolation between a spatial-front law and a two-source collision law.

 \item \textbf{Hop-count central limit theorems.}
Theorem~\ref{thm:hop-count-asymptotics} identifies the first-order hop-count constants in both regimes, but no second-order result is proved here. In the macroscopic regime one expects, in analogy with~\cite{bhh10a,bh12}, a Gaussian limit on the $\sqrt{\log n}$ scale,
\begin{align*}
 \frac{H_n^{(\ell)}(0,U_n)-\theta^{-1}\log n}{\sqrt{\log n}}
 \Rightarrow
 \N\bigl(0,\gs_{\mathrm{macro\text{-}hop}}^2\bigr),
\end{align*}
which should follow from the generation concentration of Lemma~\ref{lem:macro-bridge-gen-conc} together with a local CLT for the number of generations of the two spines. In the mesoscopic regime one expects, for some $\gs_{\mathrm{meso\text{-}hop}}^2\in(0,\infty)$,
\begin{align*}
 \frac{H_n^{(\ell)}(0,\sfu_n)-\E H_n^{(\ell)}(0,\sfu_n)}
 {\sqrt{\sfu_n/\ell}}
 \Rightarrow
 \N\bigl(0,\gs_{\mathrm{meso\text{-}hop}}^2\bigr).
\end{align*}
A joint CLT with the passage time and other geodesic observables is also natural. The main obstacle in the mesoscopic case is that the block decomposition controls passage-time errors but does not directly control the identity or the generation of the minimizing path.

\item \textbf{Lower-tail universality beyond the Weibull law.}
Assumption~\ref{ass:weight} fixes the exact Weibull law $\go\sim \Exp(1)^{1/\theta}$; this exactness is used twice, in the hazard-scale memorylessness of Lemma~\ref{lem:hazard-residual} and in the exact Cox representation of Proposition~\ref{prop:cox}. It is natural to conjecture that all the results hold whenever
$
 \pr(\go\le t)=t^{\theta+o(1)}
$
as $t\downarrow0$, with $\fa_\ell$ replaced by the appropriate quantile normalization, since only the behavior of $F$ near the origin should matter.

\item \textbf{Universality across dense spatial graphs.}
One may ask whether the pair $(\vstar,S_\infty)$ depends on the edge-weight law only through $\theta$ and the constant in the previous item, and whether the macroscopic limit $S_\infty$ is universal across a wider class of dense spatial graphs, in the spirit of~\cite{bhh17}.

 \item \textbf{General dispersal kernels.}
Here every pair at distance at most $\ell$ is joined, so that the rescaled displacement is uniform on $[-1,1]$ and $\gf(\beta)=(\sinh\beta)/\beta$. For a general dispersal density $\rho$ on $[-1,1]$, where a pair at distance $j\le\ell$ is joined with probability $\rho(j/\ell)/\ell$, the same scheme should give $\gf(\beta)=\int e^{\beta x}\rho(x)\,\dd x$ and hence a $\rho$-dependent speed $\vstar$; heavy-tailed kernels, for which $\gf$ is infinite, should instead produce a genuinely different regime.

 \item \textbf{Multi-source and multi-species growth.}
Starting from finitely many seeds, possibly with type-dependent edge laws or dispersal kernels, one may study limiting type proportions, coexistence, and competition interfaces. Mesoscopic interfaces should be governed by directional spatial-CMJ speeds, whereas macroscopic competition should reduce to interacting CMJ clusters and their first bridge hazards. A joint limit for the growth clusters and the induced random Voronoi partition would connect these two mechanisms.

 \item \textbf{Random-shortcut extensions.}
One may add random long edges as in the Newman--Watts model, or randomly rewire local edges as in the Watts--Strogatz model~\cite{wattsstrogatz98,newmanwatts99,kv16}. A non-degenerate interpolation with the deterministic spread-out model requires a joint scaling of $\ell$, the shortcut density or rewiring probability, and possibly the shortcut-weight scale. Determining the correct scaling window and the resulting phase diagram is open.
	
 \item \textbf{Extensions to higher dimensions.}
The whole analysis is carried out on the one-dimensional cycle. On the $\ell$-spread-out torus $(\dZ/n\dZ)^d$ with $d\ge2$, the local Dijkstra exploration should again converge to a spatial CMJ branching random walk, now with $\Unif([-1,1]^d)$ displacements, so that the mesoscopic time constant is the associated multi-dimensional front speed. Both the exact order of $\var(T_n)$ and the nature of the fluctuations (Gaussian versus KPZ-type) are open for $d\ge2$; even the analogue of the variance upper bound of Lemma~\ref{lem:block-moments} requires a substitute for the one-dimensional block decomposition.	
\end{enumeraten}

\noindent\textbf{Acknowledgements.}
The authors thank Greg Terlov for many helpful conversations during the early stage of this project. They are also grateful to Tianyi Huang, Hansen Liu, and Yongzheng Yang for their contributions regarding simulations for the project in the Spring of 2023, which was recognized as a runner-up for the Illinois Mathematics Lab Research Award.

\appendix

%%%%%%%%%%%%%%%%%%%%%%%%%%%%%%%%%%%%%%%%%%%%%%%%%%%
\section{Proofs for the CMJ to \texorpdfstring{$\ell$}{l}-spread out coupling}\label{app:cmj-to-lspread}

This section contains the proof of Lemmas and Propositions used in Section~\ref{sec:cmj-mean}.
%%%%%%%%%%%%%%%%%%%%%%%%%%%%%%%%%%%%%%%%%%%%%%%%%%%
\subsection{Proof of Lemma~\ref{lem:finite-direction-renewal-bound}}\label{pf:finite-direction-renewal-bound}
Define the intensity measure
\begin{align*}
\nu_{\ell,\beta}(B)
:= \sum_{z\in\cD_\ell}e^{\beta z/\ell}
\pr(\fa_\ell\go_z\in B).
\end{align*}
We first prove its Laplace transform asymptotics: for any fixed $q>0$,
\begin{align}\label{eq:finite-weighted-laplace}
\int_0^\infty e^{-qs}\nu_{\ell,\beta}(\dd s)
\to
\gf(\beta)q^{-\theta} 
\text{ as } \ell \to \infty.
\end{align}
By integration by parts, the integral equals $q\int_0^\infty e^{-qs}\nu_{\ell,\beta}([0,s])\,\dd s$. Recall $\fa_\ell^\theta = (2\ell-1)\Gamma(\theta+1)$. Using the elementary inequality $1-e^{-x}\le x$ for $x\ge 0$, we have
\begin{align*}
\nu_{\ell,\beta}([0,s])
&=
\sum_{z\in\cD_\ell}e^{\beta z/\ell}
\left(
1-\exp\left\{
-\frac{s^\theta}{(2\ell-1)\Gamma(\theta+1)}
\right\}
\right) \\
&\le
\left(\frac{1}{2\ell-1}\sum_{z\in\cD_\ell}e^{\beta z/\ell}\right)
\frac{s^\theta}{\Gamma(\theta+1)}
\le C_\beta s^\theta.
\end{align*}
As $\ell \to \infty$, the prefactor $\frac{1}{2\ell-1}\sum_{z\in\cD_\ell}e^{\beta z/\ell}$ converges as a Riemann sum to $\int_{-1}^1 e^{\beta x}dx/2 = \gf(\beta)$. Thus, by the dominated convergence theorem (using the integrable dominating function $q C_\beta e^{-qs} s^\theta$),
\begin{align*}
 q\int_0^\infty e^{-qs}\nu_{\ell,\beta}([0,s])\,\dd s
 \to
 q\int_0^\infty e^{-qs} \gf(\beta)\frac{s^\theta}{\Gamma(\theta+1)}\,\dd s
 = \gf(\beta)q^{-\theta},
\end{align*}
which proves~\eqref{eq:finite-weighted-laplace}.

Next, choose $q \in (\gk_\theta(\beta),\,\gk_\theta(\beta)+\eta)$. Since $\gk_\theta(\beta) = \gf(\beta)^{1/\theta}$, we have $\gf(\beta)q^{-\theta}<1$. By~\eqref{eq:finite-weighted-laplace}, there exists $r<1$ such that for all sufficiently large $\ell$,
\begin{align*}
\int_0^\infty e^{-qs}\nu_{\ell,\beta}(\dd s) \le r.
\end{align*}

Now we bound the expected weighted population. By independence of the edge weights along ancestral lineages, the contribution from generation $k$ is exactly the $k$-fold convolution of $\nu_{\ell,\beta}$:
\begin{align*}
\E\sum_{|v|=k}
e^{\beta \sfp_\ell(v)}\ind_{\{\sft_\ell(v)\le t\}} 
=
\nu_{\ell,\beta}^{*k}([0,t]).
\end{align*}
Applying the exponential Markov inequality and the multiplicativity of Laplace transforms,
\begin{align*}
 \nu_{\ell,\beta}^{*k}([0,t])
 \le e^{qt}
 \int_0^\infty e^{-qs}\nu_{\ell,\beta}^{*k}(\dd s) 
 = e^{qt}
 \left(
 \int_0^\infty e^{-qs}\nu_{\ell,\beta}(\dd s)
 \right)^k
 \le e^{qt}r^k.
\end{align*}
Summing over all generations $k \ge 0$ yields
\begin{align*}
 \E\sum_{v\in\cN_t^{(\ell)}}e^{\beta \sfp_\ell(v)}
 = \sum_{k\ge0}\nu_{\ell,\beta}^{*k}([0,t])
 \le \frac{e^{qt}}{1-r}
 \le C_{\beta,\eta}e^{(\gk_\theta(\beta)+\eta)t},
\end{align*}
which concludes the proof of~\eqref{eq:finite-tree-exp-moment}.\qed
%%%%%%%%%%%%%%%%%%%%%%%%%%%%%%%%%%%%%%%%%%%%%%%%%%%

%%%%%%%%%%%%%%%%%%%%%%%%%%%%%%%%%%%%%%%%%%%%%%%%%%%
\subsection{Proof of Lemma~\ref{lem:cmj-ell-tree-coupling}}\label{pf:cmj-ell-tree-coupling}
We first restrict to a high-probability event on which both trees have at most $N_\ell$ particles. Since $t_\ell^\star=O(h_\ell)=o(\log\ell)$, Lemma~\ref{lem:finite-direction-renewal-bound}
with $\beta=0$ gives, for every fixed $\eta_0>0$,
\begin{align*}
 \E|\cN_{t_\ell^\star}^{(\ell)}|
 \le
 C e^{(1+\eta_0)t_\ell^\star}
 =
 \ell^{o(1)}.
\end{align*}
 Similarly, for the limiting CMJ tree, for every fixed $q>1$,
\begin{align*}
 \E|\cN_{t_\ell^\star}|
 =
 \sum_{k\ge0}\mu_\theta^{*k}([0,t_\ell^\star])
 \le
 C_q e^{qt_\ell^\star}
 =
 \ell^{o(1)}.
\end{align*}
Hence, with
\begin{align}\label{eq:E-pop-def}
 \cE_\ell^{\mathrm{pop}}
 :=
 \left\{
 |\cN_{t_\ell^\star}|\le N_\ell,\,
 |\cN_{t_\ell^\star}^{(\ell)}|\le N_\ell
 \right\},
\end{align}
Markov's inequality yields
\begin{align}\label{eq:E-pop-prob}
 \pr\big((\cE_\ell^{\mathrm{pop}})^c\big)=o(1).
\end{align}

\noindent
\textbf{Step 1 -- Coupling the unmarked trees.}
For $0<t\le t_\ell^\star$, let
\begin{align*}
 \Xi_{\ell,t}^{0}
 :=
 \sum_{z\in\cD_\ell}
 \ind_{\{\fa_\ell\go_z\le t\}}
 \gd_{\fa_\ell\go_z}
\end{align*}
be the unmarked offspring-age point process of one $\ell$-spread-out parent, restricted to ages at most $t$.
 Put
\begin{align*}
 F_\ell(s)
 :=
 \pr(\fa_\ell\go\le s)
 =
 1-\exp\left(
 -\frac{s^\theta}{(2\ell-1)\Gamma(\theta+1)}
 \right),
 \quad 0\le s\le t,
\end{align*}
and define
\begin{align*}
 \bar\mu_{\ell,t}([0,s])
 :=
 2\ell F_\ell(s),
 \quad 0\le s\le t.
\end{align*}
For each $z \in \cD_\ell$, the summand in $\Xi_{\ell,t}^{0}$ is a Bernoulli point process with mass probability $F_\ell(t)$. By coupling each Bernoulli count with an independent $\operatorname{Poisson}(F_\ell(t))$ count and matching the point locations identically when both are 1, the total variation error for each component is bounded by $F_\ell(t)^2$. Combining these independent components yields:
\begin{align}\label{eq:bernoulli-poisson-age}
 \dTV
 \left(
 \cL(\Xi_{\ell,t}^{0}),
 \operatorname{PPP}(\bar\mu_{\ell,t})
 \right)
 &\le
 \sum_{z\in\cD_\ell}F_\ell(t)^2
 \le
 C\frac{t^{2\theta}}{\ell}
 \le
 C\frac{(t_\ell^\star)^{2\theta}}{\ell},
\end{align}
uniformly for $0 < t \le t_\ell^\star$, where the second inequality uses $1-e^{-x} \le x$.

Next compare $\bar\mu_{\ell,t}$ with the limiting reproduction measure $\mu_\theta|_{[0,t]}$, where $\mu_\theta(\dd s)=\frac{s^{\theta-1}}{\Gamma(\theta)}\,\dd s$ defined in~\eqref{eq:limiting-mu}.
Using $\Gamma(\theta+1)=\theta\Gamma(\theta)$, for $0<s\le t$,
\begin{align*}
 \frac{d\bar\mu_{\ell,t}}{d\mu_\theta}(s)
 =
 \frac{2\ell}{2\ell-1}
 \exp\left(
 -\frac{s^\theta}{(2\ell-1)\Gamma(\theta+1)}
 \right),
\end{align*}
and hence
\begin{align*}
 \left|
 \frac{d\bar\mu_{\ell,t}}{d\mu_\theta}(s)-1
 \right|
 \le
 \frac{C}{\ell}(1+s^\theta).
\end{align*}
Integrating over $[0,t]$ gives, uniformly for $0<t\le t_\ell^\star$,
\begin{align}\label{eq:finite-intensity-tv}
 \dTV\left(
 \bar\mu_{\ell,t},\mu_\theta|_{[0,t]}
 \right)
 \le
 C\frac{(t_\ell^\star)^\theta+(t_\ell^\star)^{2\theta}}{\ell}.
\end{align}
Since total variation between Poisson point processes is bounded by the total variation of their intensity measures,~\eqref{eq:bernoulli-poisson-age} and~\eqref{eq:finite-intensity-tv} imply that, uniformly for $0<t\le t_\ell^\star$,
\begin{align}\label{eq:age-tv-uniform}
 \dTV
 \left(
 \cL(\Xi_{\ell,t}^{0}),
 \operatorname{PPP}(\mu_\theta|_{[0,t]})
 \right)
 \le
 C\frac{(t_\ell^\star)^\theta+(t_\ell^\star)^{2\theta}}{\ell}.
\end{align}

We now explore the two trees recursively. Whenever a limiting parent $u$ has already been matched to a finite parent $u'$ and $\sft(u)\le t_\ell^\star$, we couple the two offspring-age point processes up to the remaining time $t_\ell^\star-\sft(u)$ by a maximal coupling. If the coupling succeeds, the children born before time $t_\ell^\star$ are paired in increasing order of their ages. 
Let $\cE_\ell^{\mathrm{age}}$ be the event that all these age couplings succeed during the exploration up to time $t_\ell^\star$. On $\cE_\ell^{\mathrm{pop}}$, there are at most $N_\ell$ matched parents to explore. Since $t_\ell^\star=o(\log\ell)$ and $\gd<1/2$, by~\eqref{eq:age-tv-uniform},
\begin{align}\label{eq:E-age-error}
 \pr\left(
 (\cE_\ell^{\mathrm{age}})^c
 \cap
 \cE_\ell^{\mathrm{pop}}
 \right)
 \le
 C N_\ell
 \frac{(t_\ell^\star)^\theta+(t_\ell^\star)^{2\theta}}{\ell}
 =
 o(1).
\end{align}
On $\cE_\ell^{\mathrm{age}}$, the unmarked genealogies and all birth times agree up to time $t_\ell^\star$.\medskip

\noindent
\textbf{Step 2 -- Coupling the spatial marks.} 
Choose a measurable map $R_\ell:[-1,1]\to\cD_\ell$ such that $R_\ell(\xi)$ is uniform on $\cD_\ell$ when $\xi\sim\Unif(-1,1)$, and
\begin{align}\label{eq:rounding-map-error}
 \left|
 \frac{R_\ell(\xi)}{\ell}-\xi
 \right|
 \le
 \frac1\ell.
\end{align}

Work on $\cE_\ell^{\mathrm{age}}$. For a matched parent $u$, let
$K_u$ be the number of its children born before time $t_\ell^\star$, and list them as
\[
 w_1(u),\ldots,w_{K_u}(u)
\]
in increasing order of birth age. If $\xi_{u,j}\sim\Unif(-1,1)$ is the limiting displacement on the edge $u\to w_j(u)$, then
\begin{align}\label{eq:limiting-tuple}
 \big(
 R_\ell(\xi_{u,1}),\ldots,R_\ell(\xi_{u,K_u})
 \big)
\end{align}
has the uniform \emph{with}-replacement law on $\cD_\ell^{K_u}$. 
On the finite $\ell$-spread-out side, conditional on the same age data, the corresponding direction tuple has the uniform \emph{without}-replacement law.

For $K\le |\cD_\ell|$, if $U_1,\ldots,U_K$ are i.i.d.~uniform on $\cD_\ell$ and $A_K:=\{U_1,\ldots,U_K\text{ are all distinct}\}$, then the without-replacement law is the law of $(U_1,\ldots,U_K)$ conditioned on
$A_K$. Hence
\begin{align}
 &\dTV
 \left(
 \text{uniform with-replacement }K\text{-tuple},
 \text{uniform without-replacement }K\text{-tuple}
 \right) \notag\\
 &\qquad\le
 \pr(A_K^c)
 \le
 \sum_{1\le i<j\le K}\pr(U_i=U_j)
 \le
 C\frac{K^2}{\ell}.
\label{eq:with-without-replacement-tv}
\end{align}

Let $\cF_{t_\ell^\star}^{\mathrm{tree}}$ be the $\gs$-algebra generated by the matched unmarked genealogy and birth times. Conditional on $\cF_{t_\ell^\star}^{\mathrm{tree}}$, couple the two direction tuples maximally, parent by parent. Denote the finite directions by $(Y_{u,1},\ldots,Y_{u,K_u})$, and let
\begin{align}\label{eq:E-mark-def}
 \cE_\ell^{\mathrm{mark}}
 :=
 \bigcap_{u\in\cN_{t_\ell^\star}}
 \bigcap_{j=1}^{K_u}
 \{Y_{u,j}=R_\ell(\xi_{u,j})\}.
\end{align}
By~\eqref{eq:with-without-replacement-tv},
\begin{align*}
 \pr\left(
 (\cE_\ell^{\mathrm{mark}})^c
 \mid
 \cF_{t_\ell^\star}^{\mathrm{tree}}
 \right)
 \le
 \frac{C}{\ell}
 \sum_{u\in\cN_{t_\ell^\star}}K_u^2.
\end{align*}
On $\cE_\ell^{\mathrm{pop}}$,
\[
 \sum_{u\in\cN_{t_\ell^\star}}K_u
 =
 |\cN_{t_\ell^\star}|-1
 \le
 N_\ell,
\]
so $\sum_uK_u^2\le N_\ell^2$. Therefore, since $\gd<1/2$, we obtain
\begin{align}\label{eq:total-mark-error}
 \pr\left(
 (\cE_\ell^{\mathrm{mark}})^c
 \cap
 \cE_\ell^{\mathrm{pop}}
 \cap
 \cE_\ell^{\mathrm{age}}
 \right)
 \le
 C\frac{N_\ell^2}{\ell}
 =
 o(1).
\end{align}

On $\cE_\ell^{\mathrm{age}}\cap\cE_\ell^{\mathrm{mark}}$, define
$\Phi_\ell$ recursively by $\Phi_\ell(\varnothing)=\varnothing$ and
\begin{align}\label{eq:Phi-recursive-definition}
 \Phi_\ell(w_j(u))
 :=
 \Phi_\ell(u)\,Y_{u,j}
 =
 \Phi_\ell(u)\,R_\ell(\xi_{u,j}).
\end{align}
The map preserves ancestry by construction. It is injective because images of vertices in an ancestor-descendant relation have different word lengths, while two vertices whose paths split at a common ancestor receive distinct finite directions at the splitting step.

The age coupling gives
\[
 \sft_\ell(\Phi_\ell(u))=\sft(u),
 \quad u\in\cN_{t_\ell^\star}.
\]
Also, by~\eqref{eq:rounding-map-error}, each edge contributes spatial error at most $1/\ell$. Hence
\begin{align*}
 |\sfp_\ell(\Phi_\ell(u))-\sfp(u)|
 \le
 \frac{|u|}{\ell},
 \quad u\in\cN_{t_\ell^\star}.
\end{align*}
On $\cE_\ell^{\mathrm{pop}}$, every vertex in $\cN_{t_\ell^\star}$ has generation at most $|\cN_{t_\ell^\star}|-1\le N_\ell$. Therefore
\begin{align}\label{eq:Phi-space-error}
 |\sfp_\ell(\Phi_\ell(u))-\sfp(u)|
 \le
 \frac{N_\ell}{\ell}
 =
 r_\ell,
 \quad u\in\cN_{t_\ell^\star},
\end{align}
on $\cE_\ell^{\mathrm{pop}}\cap\cE_\ell^{\mathrm{age}}\cap
\cE_\ell^{\mathrm{mark}}$.

\medskip
\noindent
\textbf{Step 3 -- Boundary stability.}
Set
\begin{align}\label{eq:E-bdry-def}
 \cE_\ell^{\mathrm{bdry}}
 :=
 \left\{
 \operatorname{dist}
 \left(
 \sfp(u),\{0,h_\ell,h_\ell+1\}
 \right)>3r_\ell
 \text{ for every non-root }u\in\cN_{t_\ell^\star}
 \right\}.
\end{align}
Let
\[
 Z_\ell
 :=
 \bigcup_{x\in\{0,h_\ell,h_\ell+1\}}
 [x-3r_\ell,x+3r_\ell].
\]
Then $|Z_\ell|\le 18r_\ell$. Conditional on the unmarked limiting tree, each
non-root position $\sfp(u)$ is a sum of at least one independent $\Unif(-1,1)$
increment, and hence has a density bounded by an absolute constant
$C_{\max}$. Thus
\[
 \pr\big(\sfp(u)\in Z_\ell\mid \text{unmarked tree}\big)
 \le
 C_{\max}|Z_\ell|
 \le
 C r_\ell.
\]
Using this conditional bound and then $\cE_\ell^{\mathrm{pop}}$, we obtain
\begin{align}
 \pr\left(
 (\cE_\ell^{\mathrm{bdry}})^c
 \cap
 \cE_\ell^{\mathrm{pop}}
 \right)
 &\le
 \E\left[
 \ind_{\cE_\ell^{\mathrm{pop}}}
 \sum_{u\in\cN_{t_\ell^\star}\setminus\{\varnothing\}}
 \ind_{\{\sfp(u)\in Z_\ell\}}
 \right]\le
 C r_\ell
 \E\left[
 \ind_{\cE_\ell^{\mathrm{pop}}}
 |\cN_{t_\ell^\star}|
 \right]
 \le
 C\frac{N_\ell^2}{\ell}
 =
 o(1).
\label{eq:E-bdry-error}
\end{align}

Finally put
\[
 \cE_\ell
 :=
 \cE_\ell^{\mathrm{pop}}
 \cap
 \cE_\ell^{\mathrm{age}}
 \cap
 \cE_\ell^{\mathrm{mark}}
 \cap
 \cE_\ell^{\mathrm{bdry}}.
\]
Combining~\eqref{eq:E-pop-prob},~\eqref{eq:E-age-error},~\eqref{eq:total-mark-error}, and~\eqref{eq:E-bdry-error} gives $\pr(\cE_\ell^c)=o(1)$. 
On $\cE_\ell$, the map $\Phi_\ell$ has all the properties stated above, and the proof is complete.\qed
%%%%%%%%%%%%%%%%%%%%%%%%%%%%%%%%%%%%%%%%%%%%%%%%%%%

%%%%%%%%%%%%%%%%%%%%%%%%%%%%%%%%%%%%%%%%%%%%%%%%%%%
\subsection{Proof of Lemma~\ref{lem:finite-direc-one-side-speed}}\label{pf:finite-direc-one-side-speed}
We first prove the one-sided estimate for the limiting CMJ branching random walk. Put
\begin{align*}
 \sfM_t^+
 :=
 \max\left\{
 \sfp(v):
 v\in\cN_t,\
 \sfp(v|_j)\ge0 \text{ for all }0\le j\le |v|
 \right\}.
\end{align*}
We claim that, for every $\ga>0$,
\begin{align}
 \pr\left(\sfM_{(1/\vstar+\ga)h}^+\ge h\right)
 \to 1,
 \qquad h\to\infty,
\label{eq:limiting-one-sided-claim}
\end{align}
and that the crossing particle may be chosen as the first particle on its ancestral line entering $[h,h+1]$.

Fix $\ga>0$. Choose $s<\vstar$ and $\eta\in(0,\ga)$ such that
\begin{align}
 s\left(\frac1{\vstar}+\ga-\eta\right)>1.
\label{eq:s-choice-one-sided}
\end{align}
Since $\gk_\theta'(\gbc)=\vstar$, choose $\beta\in(0,\gbc)$ sufficiently close to $\gbc$ and then $\rho>0$ such that
\begin{align}
 s<x<x+\rho<\frac{\gk_\theta(\beta)}{\beta},
 \qquad
 x:=\gk_\theta'(\beta).
\label{eq:beta-rho-choice}
\end{align}

For $L>0$, set
\begin{align*}
 k_L:=\left\lfloor \frac{\gk_\theta(\beta)}{\theta}\,L\right\rfloor .
\end{align*}
Let $S_j:=\xi_1+\cdots+\xi_j$, where $\xi_i$ are i.i.d.~$\Unif[-1,1]$, and let $\cF_k^\xi:=\gs(\xi_1,\ldots,\xi_k)$.
For each $k$, define the $\beta$-tilted law on $\cF_k^\xi$ by
\begin{align}
 \frac{\dd\widehat{\pr}_{\beta,k}}{\dd\pr}
 = \frac{e^{\beta S_k}}{\gf(\beta)^k}.
\label{eq:finite-dimensional-tilt}
\end{align}
Equivalently, under the infinite tilted product law $\widehat{\pr}_{\beta}$, the increments have density
\begin{align*}
 \widehat\nu_\beta(\dd x)
 :=
 \frac{e^{\beta x}}{\gf(\beta)}
 \frac{\ind_{[-1,1]}(x)}{2}\,\dd x .
\end{align*}
Under $\widehat{\pr}_{\beta,k_L}$,
\begin{align}
 \frac{S_{k_L}}{L}
 \to
 \frac{\gk_\theta(\beta)}{\theta}
 \frac{\gf'(\beta)}{\gf(\beta)}
 = \gk_\theta'(\beta)
 = x
\label{eq:tilted-walk-lln}
\end{align}
in probability. Moreover the tilted increments are bounded and have positive mean. Hence
\begin{align*}
 \widehat{\pr}_\beta\left(\inf_{j\ge0}S_j\ge0\right)>0.
\end{align*}
Together with~\eqref{eq:tilted-walk-lln}, this gives constants $c_1>0$ and $t_1>0$ such that, for all $L\ge t_1$,
\begin{align*}
 \widehat{\pr}_{\beta,k_L}(A_L)\ge c_1,
\end{align*}
where
\begin{align*}
 A_L:=
 \left\{
 \min_{0\le j\le k_L}S_j\ge0,\
 sL\le S_{k_L}\le (x+\rho)L
 \right\}.
\end{align*}
By~\eqref{eq:finite-dimensional-tilt},
\begin{align*}
 \pr(A_L)
 = \widehat{\E}_{\beta,k_L}
 \left[
 \gf(\beta)^{k_L}e^{-\beta S_{k_L}}
 \ind_{A_L}
 \right] 
 \ge c_1\gf(\beta)^{k_L}e^{-\beta(x+\rho)L}.
\end{align*}

Define the number of good generation-$k_L$ descendants by
\begin{align*}
 Z_L:=
 \sum_{|v|=k_L}
 \ind\left\{
 \sft(v)\le L,\
 \min_{0\le j\le k_L}\sfp(v|_j)\ge0,\
 sL\le \sfp(v)\le (x+\rho)L
 \right\}.
\end{align*}
Since the time part and the space part are independent, by the iterated integrals for convolutions,
\begin{align*}
 \E Z_L
 = \mu_\theta^{*k_L}([0,L])\,\pr(A_L) 
 \ge c_1 \frac{L^{k_L\theta}}{\Gamma(k_L\theta+1)}
 \gf(\beta)^{k_L}
 e^{-\beta(x+\rho)L}.
\end{align*}
Since $k_L\theta=\gk_\theta(\beta)L+O(1)$ and $\gk_\theta(\beta)=\gf(\beta)^{1/\theta}$, Stirling's formula gives
\begin{align*}
 \frac{L^{k_L\theta}}{\Gamma(k_L\theta+1)}
 \gf(\beta)^{k_L}
 = \exp\{\gk_\theta(\beta)L+o(L)\}.
\end{align*}
Consequently,
\begin{align}
 \E Z_L
 \ge
 c_1
 \exp\left\{
 L\left(\gk_\theta(\beta)-\beta(x+\rho)+o(1)\right)
 \right\}.
\label{eq:ZT-exponential-growth}
\end{align}
By~\eqref{eq:beta-rho-choice}, choose $L_0$ so large that
\begin{align*}
 \fm_0:=\E Z_{L_0}>1.
\end{align*}

For this fixed $L_0$, define the selected offspring set of a particle $u$ by
\begin{align*}
\mathsf{Selected}(u):=
 \left\{
 v\succeq u:
 \begin{array}{l}
 |v|-|u|=k_{L_0},\quad \sft(v)-\sft(u)\le L_0,\\[2mm]
 \min_{u\preceq w\preceq v}\{\sfp(w)-\sfp(u)\}\ge0,\\[1mm]
 sL_0\le \sfp(v)-\sfp(u)\le (x+\rho)L_0
 \end{array}
 \right\}.
\end{align*}
Set $Z_{L_0}(u):=|\mathsf{Selected}(u)|$. For any finite collection of
pairwise incomparable particles $u_1,\ldots,u_k$, the branching property implies that
\begin{align*}
 Z_{L_0}(u_1),Z_{L_0}(u_2),\ldots,Z_{L_0}(u_k)
\end{align*}
are independent and each has the same law as $Z_{L_0}$. More explicitly, define
\begin{align*}
 \cZ_0(u):=\{u\},\qquad
 \cZ_{r+1}(u):=
 \bigcup_{w\in\cZ_r(u)}\mathsf{Selected}(w),
 \qquad r\ge0.
\end{align*}
Then $\{|\cZ_r(u)|:r\ge0\}$ is a Galton--Watson process with offspring
distribution $Z_{L_0}$. Indeed, by induction, the particles in
$\cZ_r(u)$ are pairwise incomparable. Therefore the descendant subtrees rooted at distinct particles in $\cZ_r(u)$ are independent. Since $\E Z_{L_0}=\fm_0>1$, this Galton--Watson process is supercritical. Let $\gz_0>0$ be its survival probability.

By time $\eta h$, the root has
\begin{align*}
 N_h
 :=
 \#\{i:\tau_{\varnothing,i}\le \eta h,\ \xi_{\varnothing,i}\ge1/2\}
 \sim
 \operatorname{Poisson}
 \left(
 \frac14\,\mu_\theta(0,\eta h]
 \right)
\end{align*}
children with position at least $1/2$. Conditional on these children, the auxiliary Galton--Watson processes initiated from them are independent, and each survives with probability $\gz_0$. Hence, by Poisson thinning,
\begin{align}\label{eq:no-surviving-child}
 \pr(\text{no child has an infinite selected line of descent})
 =
 \exp\left(
 -\frac{\gz_0}{4}\mu_\theta(0,\eta h]
 \right)
 = o(1).
\end{align}

On the complementary event, choose a child $u_0$ whose selected Galton--Watson process survives. Define
\begin{align*}
 r_h:=
 \left\lfloor
 \frac{(1/\vstar+\ga-\eta)h}{L_0}
 \right\rfloor .
\end{align*}
Since the selected process from $u_0$ survives, there exist particles $u_0\prec u_1\prec\cdots\prec u_{r_h}$ such that $u_i\in\mathsf{Selected}(u_{i-1})$ for every $1\le i\le r_h$.
By the definition of $\mathsf{Selected}(\cdot)$,
\begin{align*}
 \sft(u_i)-\sft(u_{i-1})\le L_0,\qquad
 \sfp(u_i)-\sfp(u_{i-1})\ge sL_0,
 \quad 1\le i\le r_h.
\end{align*}
In addition, every ancestor between $u_{i-1}$ and $u_i$ has position at least $\sfp(u_{i-1})$. Hence the ancestral path from $u_0$ to $u_{r_h}$ stays above $\sfp(u_0)\ge1/2$, and
\begin{align*}
 \sfp(u_{r_h})
 \ge \sfp(u_0)+r_hsL_0
 \ge \frac12+r_hsL_0 .
\end{align*}
Using~\eqref{eq:s-choice-one-sided},
\begin{align*}
 \frac12+r_hsL_0
 \ge \frac12+s\left((1/\vstar+\ga-\eta)h-L_0\right)
 > h
\end{align*}
for all large $h$. Also,
\begin{align*}
 \sft(u_{r_h})
 \le \eta h+r_hL_0
 \le \left(\frac1{\vstar}+\ga\right)h .
\end{align*}

Let $v$ be the first ancestor of $u_{r_h}$ whose position is at least $h$. The ancestral path from the root to $u_{r_h}$ stays nonnegative, and $\sfp(u_{r_h})>h$, so such a vertex $v$ exists. By the choice of $v$,
\begin{align*}
 \sfp(w)<h
 \quad
 \text{for every ancestor }w\prec v .
\end{align*}
If $v^{-}$ denotes the parent of $v$, then $\sfp(v^{-})<h$ and
$\sfp(v)-\sfp(v^{-})\le1$. Hence
\begin{align*}
 \sfp(v)\in[h,h+1],
 \qquad
 \sft(v)\le \sft(u_{r_h})
 \le
 \left(\frac1{\vstar}+\ga\right)h .
\end{align*}
This proves~\eqref{eq:limiting-one-sided-claim}.

It remains to transfer the limiting one-sided estimate to the $\ell$-spread-out
tree. Put
\begin{align*}
 t_\ell^\star:=
 \left(\frac1{\vstar}+\frac{\eps}{2}\right)h_\ell.
\end{align*}
Define the limiting good event
\begin{align*}
 \cG_\ell:=
 \left\{
 \exists v\in\cN_{t_\ell^\star}:
 \begin{array}{l}
 \sfp(v|_i)\ge0,\quad 0\le i\le |v|,\\
 \sfp(v)\in[h_\ell,h_\ell+1],\\
 \sfp(v|_i)<h_\ell,\quad 0\le i<|v|
 \end{array}
 \right\}.
\end{align*}
By~\eqref{eq:limiting-one-sided-claim}, with $\ga=\eps/2$,
\begin{align*}
 \pr(\cG_\ell)=1-o(1).
\end{align*}
Let $\cG_\ell^{(\ell)}$ be the corresponding event for the $\ell$-spread-out tree, defined by replacing $\sfp(\cdot)$ with $\sfp_\ell(\cdot)$.

Apply Lemma~\ref{lem:cmj-ell-tree-coupling} with this $t_\ell^\star$.
We claim that
\begin{align}\label{eq:limit-good-witness-ell-good}
 \cG_\ell\cap\cE_\ell
 \subseteq
 \cG_\ell^{(\ell)}.
\end{align}
On $\cG_\ell\cap\cE_\ell$, choose a limiting lineage $\{ \varnothing=v|_{0}\prec v|_{1}\prec\cdots\prec v|_{k}=v \}$ satisfying $\cG_\ell$, and define the corresponding $\ell$-spread-out lineage as $\widehat v_i := \Phi_\ell(v|_i)$ for $0\le i\le k$.
 By the coupling properties, $\sft_\ell(\widehat v_k) = \sft(v) \le t_\ell^\star$, which implies $\widehat v_k \in \cN_{t_\ell^\star}^{(\ell)}$. Furthermore, the coupling guarantees $|\sfp_\ell(\widehat v_i)-\sfp(v|_i)|\le r_\ell$ for all $0 \le i \le k$.

For any non-root vertex, the boundary-stability property on $\cE_\ell$ ensures that the position $\sfp(v|_i)$ strictly maintains a distance greater than $3r_\ell$ from the boundary set $\{0, h_\ell, h_\ell+1\}$. Combined with the uniform coupling error bound $r_\ell$, this margin robustly preserves all defining inequalities of $\cG_\ell$ under the map $\Phi_\ell$:
\begin{itemize}
 \item \textbf{Non-negativity:} Since $\sfp(v|_i) > 3r_\ell$ for $i \ge 1$ and $\sfp_\ell(\widehat v_0)=0$, we have $\sfp_\ell(\widehat v_i) > 2r_\ell > 0$ for all $i$.
 \item \textbf{Target interval:} Since $h_\ell+3r_\ell < \sfp(v) < h_\ell+1-3r_\ell$, it follows that $\sfp_\ell(\widehat v_k) \in (h_\ell+2r_\ell, h_\ell+1-2r_\ell) \subset (h_\ell, h_\ell+1)$.
 \item \textbf{Pre-terminal bound:} Since $\sfp(v|_i) \le h_\ell-3r_\ell$ for $i < k$, we have $\sfp_\ell(\widehat v_i) \le h_\ell-2r_\ell < h_\ell$.
\end{itemize}

Thus, the mapped lineage $\widehat v_k$ satisfies $\cG_\ell^{(\ell)}$, proving~\eqref{eq:limit-good-witness-ell-good}. Consequently,
\begin{align*}
 \pr(\cG_\ell^{(\ell)})
 \ge \pr(\cG_\ell\cap\cE_\ell)
 \ge 1-o(1).
\end{align*} 

Since
\begin{align*}
 t_\ell^\star
 = \left(\frac1{\vstar}+\frac{\eps}{2}\right)h_\ell
 < \left(\frac1{\vstar}+\eps\right)h_\ell,
\end{align*}
this proves the $\ell$-spread-out one-sided estimate, including the first-entrance assertion.\qed
%%%%%%%%%%%%%%%%%%%%%%%%%%%%%%%%%%%%%%%%%%%%%%%%%%%

%%%%%%%%%%%%%%%%%%%%%%%%%%%%%%%%%%%%%%%%%%%%%%%%%%%
\subsection{Proof of Proposition~\ref{prop:auxiliary-one-sided-line-crossing}}\label{pf:auxiliary-one-sided-line-crossing}
Set
\begin{align*}
 A:=\frac1{\vstar}+\frac{\eps}{2}.
\end{align*}
Define the $\ell$-spread-out tree good event
\begin{align*}
 \cG_\ell:=
 \left\{
 \exists\,\varnothing=v(0)\prec v(1)\prec\cdots\prec v(m):
 \begin{array}{l}
 \sft_\ell(v(m))\le Ah_\ell,\\
 \sfp_\ell(v(j))\ge0,\quad 0\le j\le m,\\
 \sfp_\ell(v(m))\in[h_\ell,h_\ell+1],\\
 \sfp_\ell(v(j))<h_\ell,\quad 0\le j<m
 \end{array}
 \right\}.
\end{align*}
By Lemma~\ref{lem:finite-direc-one-side-speed},
\begin{align}\label{eq:aux-tree-good}
 \pr(\cG_\ell)=1-o(1).
\end{align}

Set 
\begin{align*}
 \cV_\ell:=\cN_{Ah_\ell}^{(\ell)}\setminus\{\varnothing\}.
\end{align*}
Apply Lemma~\ref{lem:finite-direction-renewal-bound} with $\beta=0$, for which $\gk_\theta(0)=\gf(0)^{1/\theta}=1$, and with an arbitrary fixed $\eta_0>0$. Since $\cV_\ell \subset \cN_{Ah_\ell}^{(\ell)}$, substituting $t=Ah_\ell$ yields
\begin{align*}
 \E \abs{\cV_\ell}
 \le \E \big|\cN_{Ah_\ell}^{(\ell)}\big|
 \le C \exp\left\{ (1+\eta_0)Ah_\ell \right\}
 = \exp\left\{ o(\log\ell) \right\}
 = \ell^{o(1)}.
\end{align*}

Choose $\gd\in(0,1/2)$ and put
\begin{align*}
 N_\ell:=\ell^\gd.
\end{align*}
Then
\begin{align}\label{eq:aux-explored-size-tail}
 \pr(\abs{\cV_\ell}>N_\ell)
 \le
 \frac{\E \abs{\cV_\ell}}{N_\ell}
 = o(1).
\end{align}

For an $\ell$-spread-out tree vertex $v=(z_1,\ldots,z_k)$, define its unscaled spatial position on $\dZ$ by
\begin{align*}
 S(\varnothing):=0,
 \qquad
 S(v):=\sum_{i=1}^k z_i,
 \text{ and }
 S(uz)=S(u)+z,
 \quad z\in\cD_\ell.
\end{align*}
For a directed tree edge $e=(u,uz)$, define its induced undirected edge in the spread-out line graph by
\begin{align*}
 I(e):=\{S(u),S(u)+z\}.
\end{align*}
Denote the set of explored tree edges by
\begin{align*}
 \cE_\ell^{\mathrm{tr}} := \{(u,uz): uz\in\cV_\ell\}.
\end{align*}

We need the spatial footprint of the explored tree to form a collision-free simple path environment on $\dZ$. Thus, we define the overall collision event
\begin{align*}
 \fB_\ell
 :=
 \fB_\ell^{\mathrm{bt}}
 \cup
 \fB_\ell^{\mathrm{v}}
 \cup
 \fB_\ell^{\mathrm{e}},
\end{align*}
where the respective events for immediate backtracking, vertex collisions, and undirected edge collisions are given by
\begin{align*}
 \fB_\ell^{\mathrm{bt}}
 &:=
 \left\{
 \exists\,v=(z_1,\ldots,z_k)\in\cV_\ell:
 k\ge2,\ z_k=-z_{k-1}
 \right\},\\
 \fB_\ell^{\mathrm{v}}
 &:=
 \left\{
 \exists\,u\ne v\in\{\varnothing\}\cup\cV_\ell:
 S(u)=S(v)
 \right\},\\
 \fB_\ell^{\mathrm{e}}
 &:=
 \left\{
 \exists\,e\ne f\in\cE_\ell^{\mathrm{tr}}:
 I(e)=I(f)
 \right\}.
\end{align*}

Let $\cF_\ell^0$ be the sigma-field generated by the unlabelled genealogy and all birth ages up to time $Ah_\ell$. 
On $\{\abs{\cV_\ell}\le N_\ell\}$, order the explored tree edges measurably as $\{ e_1,\ldots,e_{\abs{\cV_\ell}} \}$ by the birth times of their child endpoints. 
Write
\begin{align*}
 e_k=(e_k^-,e_k^+),
 \qquad
 e_k^+=e_k^-L_k,
 \qquad
 L_k\in\cD_\ell,
\end{align*}
where $L_k$ is the direction label of $e_k$. Define the discrete label filtration
\begin{align*}
 \cH_0:=\cF_\ell^0,
 \qquad
 \cH_k:=\cF_\ell^0\vee\gs(L_1,\ldots,L_k),
 \quad 1\le k\le \abs{\cV_\ell}.
\end{align*}
Conditional on $\cF_\ell^0$, the direction labels born from each parent are sampled uniformly without replacement from $\cD_\ell$, independently across different parents. Hence, on $\{\abs{\cV_\ell}\le N_\ell\}$, for every $1\le k\le \abs{\cV_\ell}$ and every $\cH_{k-1}$-measurable random variable $Z$ with values in $\cD_\ell$,
\begin{align}
 \pr(L_k=Z\mid \cH_{k-1})
 \le
 \frac1{2\ell-N_\ell}
 \le
 \frac{C}{\ell},
\label{eq:aux-predictable-label-bound}
\end{align}
for all large $\ell$.

We first bound immediate backtracking. If $e_k$ immediately backtracks, then $e_k^-\ne\varnothing$ and, writing $e_i$ for the parent edge of $e_k^-$, we have $i<k$ and $L_k=-L_i$.
Since $-L_i$ is $\cH_{k-1}$-measurable,~\eqref{eq:aux-predictable-label-bound} gives, on $\{\abs{\cV_\ell}\le N_\ell\}$,
\begin{align}
 \pr(\fB_\ell^{\mathrm{bt}}\mid \cF_\ell^0)
 \le
 C\frac{\abs{\cV_\ell}}{\ell}
 \le
 C\frac{N_\ell}{\ell}.
\label{eq:aux-bt-bound}
\end{align}

Next we bound vertex collisions. For a vertex $v\in\{\varnothing\}\cup\cV_\ell$,
let
\begin{align*}
 \cA(v):=\{k:e_k\text{ lies on the ancestral path from }\varnothing
 \text{ to }v\}.
\end{align*}
Then
\begin{align*}
 S(v)=\sum_{k\in\cA(v)}L_k.
\end{align*}
Fix distinct vertices $u\ne v$. Since $\cA(u)\triangle\cA(v)\ne\varnothing$, let
\begin{align*}
 k_*:=\max\bigl(\cA(u)\triangle\cA(v)\bigr).
\end{align*}
There are signs $\gs_k\in\{-1,1\}$ such that
\begin{align*}
 S(u)-S(v)
 =
 \sum_{k\in\cA(u)\triangle\cA(v)}\gs_k L_k.
\end{align*}
Thus the equation $S(u)=S(v)$ can be written as
\begin{align*}
 L_{k_*}
 =
 -\gs_{k_*}
 \sum_{\substack{k\in\cA(u)\triangle\cA(v)\\ k<k_*}}
 \gs_kL_k.
\end{align*}
The right-hand side is $\cH_{k_*-1}$-measurable. If this value falls outside $\cD_\ell$, the probability is trivially zero; otherwise,~\eqref{eq:aux-predictable-label-bound} applies. In either case, on $\{|\cV_\ell| \le N_\ell\}$, we obtain
\begin{align}
 \pr(S(u)=S(v)\mid \cF_\ell^0)
 \le
 \frac{C}{\ell}.
\label{eq:aux-one-vertex-collision-bound}
\end{align}
Taking the union bound over at most $(\abs{\cV_\ell}+1)^2$ ordered pairs gives
\begin{align}
 \pr(\fB_\ell^{\mathrm{v}}\mid \cF_\ell^0)
 \le
 C\frac{N_\ell^2}{\ell}
 \quad\text{on }
 \{\abs{\cV_\ell} \le N_\ell\}.
\label{eq:aux-vertex-collision-bound}
\end{align}

It remains to control undirected edge collisions. Fix two distinct explored tree edges $e \ne f$. If they are siblings ($e^-=f^-$), then $I(e)=I(f)$ would imply their assigned directions are identical ($L_e=L_f$). This is impossible since sibling directions are sampled without replacement from $\cD_\ell$.

If they are not siblings ($e^- \ne f^-$), the equality $I(e)=I(f)$ requires their endpoint sets to be identical. This forces at least one nontrivial vertex collision $S(u')=S(v')$ for distinct $u',v'\in\{e^-,e^+,f^-,f^+\}$. By~\eqref{eq:aux-one-vertex-collision-bound}, the conditional probability of this event is at most $C/\ell$. Taking the union bound over all possible edge pairs (at most $\abs{\cV_\ell}^2$), we obtain
\begin{align}
 \pr(\fB_\ell^{\mathrm{e}}\mid \cF_\ell^0)
 \le C\frac{\abs{\cV_\ell}^2}{\ell}
 \le C\frac{N_\ell^2}{\ell}
 \quad\text{ on }\{\abs{\cV_\ell} \le N_\ell\}.
\label{eq:aux-edge-collision-bound}
\end{align}

Combining the conditional bounds for immediate backtracking, vertex collisions, and edge collisions, the overall collision event is uniformly bounded by $C N_\ell^2/\ell$ on the event $\{\abs{\cV_\ell}\le N_\ell\}$. Integrating out the conditioning yields
\begin{align}
 \pr(\fB_\ell)
 &\le \pr(\abs{\cV_\ell}>N_\ell)
 + \E\left[ \ind_{\{\abs{\cV_\ell}\le N_\ell\}} \pr(\fB_\ell\mid \cF_\ell^0) \right] \notag\\
 &\le o(1) + C\frac{N_\ell^2}{\ell} = o(1),
\label{eq:aux-bad-event-prob}
\end{align}
because $N_\ell^2/\ell = \ell^{2\gd-1}$ and $\gd < 1/2$.

We now couple the $\ell$-spread-out tree weights to $\ell$-spread-out line weights. Give the ambient spread-out line independent scaled edge weights
\begin{align*}
 W_e^{\mathrm{line}}
 :=
 \fa_\ell\go_e,
 \qquad
 \go_e\stackrel{d}=E_e^{1/\theta},
 \quad
 E_e\sim\operatorname{Exp}(1).
\end{align*}
During the tree exploration, whenever an explored tree edge $e=(u,uz)$ is requested, we couple its weight to the induced undirected line edge $I(e) = \{S(u),S(u)+z\}$. On the collision-free event $\fB_\ell^c$, all explored tree edges are mapped injectively to distinct line edges. Consequently, for every explored lineage $\varnothing=v(0)\prec\cdots\prec v(m)$, its passage time matches the sum of the line weights:
\begin{align}
 \sft_\ell(v(m))
 = \sum_{j=1}^m W_{\{S(v(j-1)),S(v(j))\}}^{\mathrm{line}}.
\label{eq:tree-line-weight-identity}
\end{align}

On $\cG_\ell\cap\fB_\ell^c$, choose a good lineage $\varnothing=v(0)\prec v(1)\prec\cdots\prec v(m)$, and set its spatial footprint $x_j:=S(v(j))$ for $0\le j\le m$. Since $\sfp_\ell(v(j))=x_j/\ell$, the defining conditions of $\cG_\ell$ guarantee that the trajectory stays non-negative ($x_j \ge 0$), enters the terminal window strictly at the end ($x_m \in [H_\ell, H_\ell+\ell]$), and does not exceed the threshold beforehand ($x_j \le H_\ell$ for $j<m$). Furthermore, since each direction belongs to $\cD_\ell$, we have $1 \le |x_j-x_{j-1}| \le \ell$.

Since $\fB_\ell^c$ guarantees the absence of vertex and edge collisions, the sequence $(x_0,x_1,\ldots,x_m)$ constitutes a \emph{simple path} in the induced spread-out line graph on $[0,H_\ell+\ell]$, starting from $0$ and ending in $[H_\ell,H_\ell+\ell]$. Since the first-passage time is bounded by the weight of any valid simple path,~\eqref{eq:tree-line-weight-identity} yields
\begin{align}
 \fa_\ell T_{[0,H_\ell+\ell]} \left(0,[H_\ell,H_\ell+\ell]\right)
 \le \sum_{j=1}^m W_{\{x_{j-1},x_j\}}^{\mathrm{line}}
 = \sft_\ell(v(m)) 
 \le Ah_\ell.
\label{eq:aux-line-bound-on-good-event}
\end{align}

Since $Ah_\ell = (1/\vstar+\eps/2)h_\ell < (1/\vstar+\eps)h_\ell$, we conclude that
\begin{align*}
 \pr\left( \fa_\ell T_{[0,H_\ell+\ell]} \left(0,[H_\ell,H_\ell+\ell]\right) \le \left(\frac1{\vstar}+\eps\right)h_\ell \right)
 &\ge \pr(\cG_\ell\cap\fB_\ell^c) \\
 &\ge 1-\pr(\cG_\ell^c)-\pr(\fB_\ell) = 1-o(1),
\end{align*}
which completes the proof.\qed
%%%%%%%%%%%%%%%%%%%%%%%%%%%%%%%%%%%%%%%%%%%%%%%%%%%

\subsection{Proof of Proposition~\ref{prop:finite-graph-window-tools}}\label{pf:finite-graph-window-tools}
We first prove the greedy bounds. At any vertex $u$, the greedy path proceeds among the edges $\{u, u+q\}$ for $\lceil\ell/2\rceil \le q \le \ell$. The number of available forward edges is $\sfN_\ell := \ell-\lceil\ell/2\rceil+1 \ge \ell/2$. Because the sequence of visited vertices is strictly increasing, the sets of forward edges at each step are disjoint. Thus, conditionally on the past, the newly queried edge weights are independent copies of $E^{1/\theta}$.

For $\sfN_\ell$ independent copies $\go_1,\ldots,\go_{\sfN_\ell}$ of $E^{1/\theta}$, and any $t\ge0$,
\begin{align}
 \pr\Big( \fa_\ell\min_{1\le i\le \sfN_\ell}\go_i>t \Big)
 = \exp\Big( -\sfN_\ell\Big(\frac{t}{\fa_\ell}\Big)^\theta \Big)
 \le \exp\big(-c t^\theta\big),
\label{eq:greedy-one-step-tail-cmj}
\end{align}
since $\fa_\ell^\theta=(2\ell-1)\Gamma(\theta+1)$ and $\sfN_\ell\ge \ell/2$. This exponential tail bound ensures that all moments are finite. In particular, for every fixed $p\ge1$,
\begin{align}
 \Big\| \fa_\ell\min_{1\le i\le \sfN_\ell}\go_i \Big\|_p \le C_p,
\label{eq:greedy-one-step-lp-cmj}
\end{align}
uniformly in $\ell$.

Before the greedy path enters $[r,r+\ell]$, its current position is strictly smaller than $r$. Each step is at least $\lceil\ell/2\rceil$, so after at most $\lceil 2r/\ell\rceil+1$ steps, the path has entered $[r,r+\ell]$ and all vertices used lie in $[0,r+\ell]$. Minkowski's inequality and
\eqref{eq:greedy-one-step-lp-cmj} give
\begin{align*}
\Big\|
 \fa_\ell
 T_{[0,r+\ell]}(0,[r,r+\ell])
\Big\|_p
\le C_p\left(\frac r\ell+1\right),
\end{align*}
which proves~\eqref{eq:greedy-upper-cmj}.

We next prove~\eqref{eq:greedy-endpoint-cmj}. Assume $r\ge\ell$. Run the same greedy procedure until the first entrance into $[r-\ell,r]$. Let $U_j$ be the position after $j$ greedy steps and set the first step index
\begin{align*}
 \tau:=\inf\{j\ge0:U_j\in[r-\ell,r]\}.
\end{align*}
Let $G_\tau$ be the passage time accumulated by this greedy procedure up to time $\tau$. The argument used for~\eqref{eq:greedy-upper-cmj}, with target window $[r-\ell,r]$, gives
\begin{align}\label{eq:endpoint-before-window-cost}
 \left\|\fa_\ell G_\tau\right\|_p
 \le C_p\left(\frac r\ell+1\right).
\end{align}

It remains to justify that the internal edge weights in $[r-\ell,r]$ are still fresh after this stopped greedy exploration. Define
\begin{align*}
 \cG_{\rm out}
 &:=
 \gs\left(
 \go_{\{u,v\}}:
 1\le |u-v|\le\ell,\ 
 \{u,v\}\not\subset [r-\ell,r]
 \right),\\
 \cG_{\rm in}
 &:=
 \gs\left(
 \go_{\{u,v\}}:
 u,v\in [r-\ell,r],\
 1\le |u-v|\le\ell
 \right).
\end{align*}
For every $j<\tau$ we have $U_j<r-\ell$, and every edge at step $j$ has the form
\begin{align*}
 \{U_j,U_j+q\},
 \qquad
 q\in\{\lceil\ell/2\rceil,\ldots,\ell\}.
\end{align*}
Thus no edge revealed before entrance has both endpoints in $[r-\ell,r]$.
The random variables
\begin{align*}
	 \tau,
	 \quad 
	 (U_j)_{0\le j\le\tau},
	 \quad
 \text{and all edge weights revealed before entrance}
\end{align*}
are measurable with respect to $\cG_{\rm out}$. Hence the sigma-field $\cF_\tau$ generated by the stopped greedy exploration up to entrance satisfies
\begin{align}
 \cF_\tau\subseteq \cG_{\rm out}.
\label{eq:greedy-stopped-sigma-out-cmj}
\end{align}
Since $\cG_{\rm out}$ and $\cG_{\rm in}$ are independent, conditionally on $\cF_\tau$, the induced complete graph on $[r-\ell,r]$ has fresh independent edge weights with law $\Exp(1)^{1/\theta}$.

We now use the two-point complete-graph bound of Lemma~\ref{lem:complete-flooding}. Since the induced graph on the interval $[r-\ell,r]$ is the complete graph $K_{\ell+1}$, applying the two-point bound in~\eqref{eq:complete-flooding} with $m=\ell+1$ gives, uniformly in $U_\tau$,
\begin{align}
 \left\|
 \fa_\ell
 T_{[r-\ell,r]}(U_\tau,r)
 \right\|_p
 \le
 C_p\log\ell .
\label{eq:endpoint-two-point-window-cmj}
\end{align}
Using
\[
 T_{[0,r]}(0,r)
 \le
 G_\tau+T_{[r-\ell,r]}(U_\tau,r),
\]
together with~\eqref{eq:endpoint-before-window-cost} and~\eqref{eq:endpoint-two-point-window-cmj}, we obtain
\begin{align*}
\left\|
 \fa_\ell
 T_{[0,r]}(0,r)
\right\|_p
&\le C_p\left(\frac r\ell+\log\ell+1\right).
\end{align*}
This proves~\eqref{eq:greedy-endpoint-cmj}.

(ii) We now prove the one-sided short-window mean. 
For simplicity, define
\begin{align*}
 \fC_\ell(h)
 := T_{[0,\lfloor h\ell\rfloor+\ell]}
 \left(0,[\lfloor h\ell\rfloor,\lfloor h\ell\rfloor+\ell]\right),
 \quad\text{for } h\ge 1.
\end{align*}
By~\eqref{eq:aux-one-sided-line-crossing}, for every $\eps>0$,
\begin{align}
\pr\left(
 \fa_\ell\fC_\ell(h_\ell)
 \le
 \left(\frac1{\vstar}+\eps\right)h_\ell
\right)
\to1.
\label{eq:short-window-high-prob-cmj}
\end{align}
Let $F_\ell$ be the complementary event in~\eqref{eq:short-window-high-prob-cmj}. Put $ H_\ell:=\lfloor h_\ell\ell\rfloor$.
By the definition of $\fC_\ell(h_\ell)$ and by~\eqref{eq:greedy-upper-cmj} with $p=2$ and $r=H_\ell$,
\begin{align*}
 \left\|
 \fa_\ell\fC_\ell(h_\ell)
 \right\|_2
 \le C\left(\frac{H_\ell}{\ell}+1\right)
 \le C h_\ell
\end{align*}
for all sufficiently large $\ell$. Therefore, by the Cauchy--Schwarz inequality,
\begin{align*}
\fa_\ell\E \fC_\ell(h_\ell)
&\le \left(\frac1{\vstar}+\eps\right)h_\ell
+ \left\|\fa_\ell\fC_\ell(h_\ell)\right\|_2
\pr(F_\ell)^{1/2}  \le \left(\frac1{\vstar}+\eps\right)h_\ell
+o(h_\ell).
\end{align*}
Since $\eps>0$ is arbitrary,~\eqref{eq:window-mean-upper} follows.

(iii) Since the edge weights follow a continuous distribution ($\Exp(1)^{1/\theta}$), the minimum-cost simple stopped path $\pi^\star$ from $x$ to $I_x$ inside $[x, x+H+\ell]$ is uniquely defined almost surely. Let $Y$ be its terminal point and $W(\pi^\star)$ be its passage time.

 Every candidate path stops upon its first entrance into $I_x$, so all of its queried edges have their left endpoints strictly less than $x+H \le y$. Therefore, the event $\{\pi^\star = \pi, Y = y, W(\pi^\star) \le s\}$ is determined by the past $\gs$-field $\cF_y^- = \gs\left( \go_{\{u,v\}}: u<v,\ 1\le v-u\le \ell,\ u<y \right)$ generated by the edges to the left of $y$.

 Moreover, $\cF_y^-$ is independent of the forward edges $\{ \go_{\{u,v\}} : 1 \le |u-v| \le \ell,\ u,v \ge y \}$. Thus, conditionally on the chosen path $\pi^\star$, its cost $W(\pi^\star)$, and $Y=y$, the forward edge weights remain completely unobserved and hence they have independent $\Exp(1)^{1/\theta}$ laws.\qed

\section{Proofs of the finite-\texorpdfstring{$n$}{n} CMJ inputs}
\label{app:macro-cmj-inputs}

The finite-$n$ CMJ lemmas are stated in Subsection~\ref{ssec:finite-cmj}, immediately before they are used in Proposition~\ref{prop:cmj-package}. This appendix contains only their proofs. The first proof combines the model-specific Malthusian calculations with the intrinsic-martingale estimate. The second proof incorporates the only renewal input needed for the window-uniform random-characteristic theorem. The final two proofs contain the genuinely spatial Fourier estimate and the ghost correction.

\subsection{Proof of the finite-$n$ normalization Lemma~\ref{lem:finite-n-cmj-martingale}}
Since $\ell/n\to\gl\in(0,1/2)$, we have $|\cD_n|=2\ell$ for all sufficiently large $n$, giving $|\cD_n|/(2\ell-1) = 1+O(n^{-1})$ and $\fa_\ell^{-\theta} = O(n^{-1})$. By~\eqref{eq:bar-mu-density}, applying $1-e^{-x}\le x$ and differentiating under the integral sign yields
\begin{align*}
 \fm_n(q) = q^{-\theta}+O(n^{-1}), \qquad \fm_n'(q) = -\theta q^{-\theta-1}+O(n^{-1}),
\end{align*}
uniformly for $q$ in compact subsets of $(0,\infty)$. The continuous, strictly decreasing function $\fm_n$ satisfies $\fm_n(0)>1$ and $\fm_n(\infty)=0$, ensuring a unique solution to $\fm_n(\alpha_n)=1$. Since $\fm_n(1)=1+O(n^{-1})$ and $\fm_n'$ is bounded away from zero near $1$, the mean-value theorem yields $\alpha_n = 1+O(n^{-1})$. Consequently, the mean age is $\theta_n = -\fm_n'(\alpha_n) = \theta+O(n^{-1})$.

The density of $\nu_n$ is $\frac{|\cD_n|}{2\ell-1} \frac{s^{\theta-1}}{\Gamma(\theta)} e^{-\alpha_ns-(s/\fa_\ell)^\theta}$, which converges pointwise to the $\Gamma(\theta,1)$ density $s^{\theta-1}e^{-s}/\Gamma(\theta)$. Because $\alpha_n\to1$, these densities are uniformly dominated by $C(1+s^{\theta-1})e^{-s/2}$. Dominated convergence thus provides both the $L^1$ convergence and a common exponential moment bound for some $\rho>0$.

To control the intrinsic martingale (whose standard properties are detailed in~\cite{nerman81,jn84}), define the offspring sum $Q_n := \sum_{z\in\cD_n} e^{-\alpha_n\fa_\ell\go_z}$. Independence and $\E Q_n = \fm_n(\alpha_n) = 1$ give
\begin{align*}
 \E Q_n^2 = \fm_n(2\alpha_n) + 1 - |\cD_n|^{-1}.
\end{align*}
By the earlier expansion,
\begin{align}\label{eq:finite-n-second-laplace}
 \fm_n(2\alpha_n) \to 2^{-\theta} < 1.
\end{align}
Thus, for a suitable $q \in (0,1)$ and all large $n$, $\sup_n \E Q_n^2 < \infty$ and $\fm_n(2\alpha_n) \le q$.

Conditionally on generation $r$, the martingale differences $\cW_{n,r+1} - \cW_{n,r} = \sum_{|u|=r} e^{-\alpha_n\tau_u} (Q_{n,u}-1)$ consist of independent, centered summands. Orthogonality then yields
\begin{align*}
 \E\abs{\cW_{n,r+1} - \cW_{n,r}}^2 = \var(Q_n) \E\sum_{|u|=r}e^{-2\alpha_n\tau_u} = \var(Q_n)\fm_n(2\alpha_n)^r \le Cq^r.
\end{align*}
Consequently, $(\cW_{n,r})_{r\ge0}$ converges in $L^2$ to a mean-one limit $\cW_{n,\infty}$, with the uniform tail bound $\E\abs{\cW_{n,\infty} - \cW_{n,r}}^2 \le Cq^r$.

For fixed $r$, let $\cW_{n,r}^{(L)}$ denote the martingale restricted to lineages with successive birth ages at most $L$. The truncation error is uniformly bounded by $r\int_L^\infty e^{-\alpha_ns}\bar\mu_n(\dd s)$, which vanishes as $L\to\infty$ via the common exponential tail. For fixed $L$ and $r$, recursive application of Lemma~\ref{lem:local-offspring-pp} across $r$ generations gives $\cW_{n,r}^{(L)} \Rightarrow \cW_r^{(L)}$ (the difference between $2\ell$ and $2\ell-1$ offspring is negligible on bounded intervals). Since the limiting process satisfies the analogous truncation bound $\E|\cW_r-\cW_r^{(L)}| \le r\int_L^\infty e^{-s}\mu_\theta(\dd s)$, standard converging-together arguments imply $\cW_{n,r} \Rightarrow \cW_r$, where $\cW_r$ is the generation-$r$ intrinsic martingale of the limiting CMJ process. 

Letting $n\to\infty$ followed by $r\to\infty$, the uniform $L^2$ bounds ensure $\cW_{n,\infty} \Rightarrow \cW_\infty$. In Nerman's normalization, $M_\infty = \cW_\infty/\theta$. Since $\alpha_n\theta_n \to \theta$, Slutsky's theorem yields $\frac{\cW_{n,\infty}}{\alpha_n\theta_n} \Rightarrow M_\infty$. Independence extends this to the joint limit. \qed

%%%%%%%%%%%%%%%%%%%
\subsection{Proof of the uniform random-characteristic lemma}

Before proving Lemma~\ref{lem:uniform-triangular-nerman}, we establish two necessary uniform estimates. Throughout this subsection, let $\cR_n:=\sum_{m\ge0}\nu_n^{*m}$.

\begin{lem}[Uniform renewal input]
\label{lem:finite-n-uniform-renewal}
By enlarging $n_0$ if necessary,
\begin{align}\label{eq:finite-n-local-renewal}
 \sup_{n\ge n_0}\sup_{x\ge0} \cR_n([x,x+1]) < \infty.
\end{align}
Moreover, for every bounded step function $\phi$ with compact support and every sequence $b_n\to\infty$,
\begin{align}\label{eq:finite-n-key-renewal}
 \sup_{v\ge b_n}
 \abs{
 \int_{[0,v]} e^{-\alpha_n(v-u)}\phi(v-u)\cR_n(\dd u)
 -
 \frac1{\theta_n} \int_0^\infty e^{-\alpha_na}\phi(a)\,\dd a
 }
 \to 0.
\end{align}
\end{lem}

\begin{proof}
Let $p_n$ and $p$ denote the densities of $\nu_n$ and $\Gamma(\theta,1)$, respectively. By Lemma~\ref{lem:finite-n-cmj-martingale}, $p_n\to p$ in $L^1$, $\theta_n\to\theta>0$, and the sequence $(\nu_n)$ shares a common exponential moment. Consequently, the characteristic functions converge uniformly: $\sup_{t\in\dR} \abs{\widehat\nu_n(t)-\widehat\nu(t)} \le \norm{p_n-p}_{L^1} \to 0$. Since $\abs{\widehat\nu(t)} = (1+t^2)^{-\theta/2}$, it follows that $\sup_{|t|\ge1}|\widehat\nu_n(t)|$ is bounded strictly away from $1$ uniformly for all large $n$. Because the remaining finitely many $\nu_n$ admit densities (and thus are non-lattice with characteristic functions vanishing at infinity), the sequence $(\nu_n)_{n\ge n_0}$ is uniformly strongly non-lattice. 

Combined with the tightness implied by the common exponential moment, the uniform renewal theorem \cite[Theorem~1, Lemma~4]{bg07} guarantees constants $c,C>0$ and $\kappa_n^{\mathrm R}$ such that
\begin{align}\label{eq:finite-n-renewal-expansion}
 \sup_{n\ge n_0}\sup_{x\ge0}
 e^{cx} \abs{ \cR_n([0,x]) - \frac{x}{\theta_n} - \kappa_n^{\mathrm R} } \le C.
\end{align}
Since $\theta_n \to \theta > 0$, \eqref{eq:finite-n-local-renewal} immediately follows.

Now fix a step function $\phi$ supported in $[0,L]$. Since $\alpha_n\to1$, the functions $a\mapsto e^{-\alpha_na}\phi(a)$ possess a uniformly bounded supremum and total variation. For $v > L$, Stieltjes integration by parts on $[v-L,v]$ applied to the remainder in \eqref{eq:finite-n-renewal-expansion} yields
\begin{align*}
 \abs{
 \int_{[0,v]} e^{-\alpha_n(v-u)}\phi(v-u)\cR_n(\dd u)
 -
 \frac1{\theta_n} \int_0^\infty e^{-\alpha_na}\phi(a)\,\dd a
 }
 \le C_\phi e^{-c(v-L)}.
\end{align*}
Because $b_n \to \infty$, this bound vanishes for $v \ge b_n$, proving \eqref{eq:finite-n-key-renewal}.
\end{proof}

\begin{lem}[Uniform characteristic bounds]
\label{lem:finite-n-characteristic-moments}
Let $\phi:[0,\infty)\to\dR$ satisfy $|\phi(a)|\le Ce^{\eta a}$ for some $\eta<1$. Define $h_{n,\phi}(t) := e^{-\alpha_nt}\E Z_n^\phi(t)$ and $\cZ_{n,\phi}(t) := e^{-\alpha_nt}\bigl(Z_n^\phi(t)-\E Z_n^\phi(t)\bigr)$ for $t\ge0$ (and zero otherwise). Then
\begin{align}\label{eq:finite-n-characteristic-moment-bound}
 \sup_{n\ge n_0}\sup_{t\ge0} \left( \abs{h_{n,\phi}(t)} + \E\abs{\cZ_{n,\phi}(t)}^2 \right) <\infty.
\end{align}
If, moreover, $\sup_{n\ge n_0} \bigl( \sup_{a\ge0}e^{-\alpha_na}\abs{\phi(a)} + \operatorname{TV}_{[0,\infty)} \bigl(e^{-\alpha_n\cdot}\phi\bigr) \bigr) <\infty$, there exists $C_\phi<\infty$ such that, for any interval $I$ of length at most one,
\begin{align}\label{eq:finite-n-unit-window}
 \E\sup_{t\in I} \abs{ \sum_{j=1}^m a_j\cZ_{n,\phi}^{(j)}(t-b_j) }^2 \le C_\phi\sum_{j=1}^m a_j^2
\end{align}
for all $a_j\in\dR$, $b_j\ge0$, and independent copies $\cZ_{n,\phi}^{(j)}$, uniformly in $n\ge n_0$.
\end{lem}

\begin{proof}
The mean renewal equation is $h_{n,\phi}(t) = \int_{[0,t]} e^{-\alpha_n(t-s)}\phi(t-s)\cR_n(\dd s)$. Since $\alpha_n\to1$, $\eta<1$, and \eqref{eq:finite-n-local-renewal} holds, $h_{n,\phi}$ is uniformly bounded.

For the variance, the root decomposition and the preceding mean bound yield
\begin{align*}
 e^{-2\alpha_nt}\var\bigl(Z_n^\phi(t)\bigr)
 \le
 C_\phi + \int_{[0,t]} e^{-2\alpha_ns} \left[ e^{-2\alpha_n(t-s)} \var\bigl(Z_n^\phi(t-s)\bigr) \right] \bar\mu_n(\dd s).
\end{align*}
By \eqref{eq:finite-n-second-laplace}, $\fm_n(2\alpha_n)\le q<1$ uniformly for all large $n$. Applying this recursively to generation-truncated processes and passing to the limit gives $\sup_{n,t} e^{-2\alpha_nt}\var\bigl(Z_n^\phi(t)\bigr) \le C_\phi/(1-q)$.

Under the additional bounded-variation assumption, convolution with \eqref{eq:finite-n-local-renewal} ensures that 
\[
\sup_{n\ge n_0} \sup_{|I|\le1} \operatorname{TV}_I(h_{n,\phi}) <\infty.
\]
We utilize the elementary Rademacher maximal inequality for functions of bounded variation (derived via the Jordan decomposition and Doob's inequality):
\begin{align*}
 \E_\eps\sup_{t\in I} \abs{\sum_i\eps_i f_i(t)}^2 \le C\sum_i \bigl( |f_i(\inf I)| + \operatorname{TV}_I(f_i) \bigr)^2.
\end{align*}
Writing $\sigma_z:=\fa_\ell\go_z$ and letting $R_{n,\phi}(t)$ be the centered conditional-mean term, standard symmetrization applied to the centered root decomposition yields
\begin{align*}
 \E\sup_{t\in I} \abs{ \sum_{j=1}^m a_jR_{n,\phi}^{(j)}(t-b_j) }^2 \le C_\phi\sum_{j=1}^m a_j^2,
\end{align*}
where we used $\sum_{z\in\cD_n}\E e^{-2\alpha_n\sigma_z} = \fm_n(2\alpha_n) \le q$. If $K_r$ is the optimal constant in \eqref{eq:finite-n-unit-window} for the process truncated at $r$ generations, the descendant terms yield $\sqrt{K_r} \le C_\phi+\sqrt q\,\sqrt{K_{r-1}}$ with $K_0=0$. Thus $\sup_rK_r<\infty$, and non-explosion coupled with Fatou's lemma proves \eqref{eq:finite-n-unit-window}.
\end{proof}

\begin{lem}[Stopping-line bounds]
\label{lem:finite-n-coming-generation}
For $r\ge0$, define the stopping line $\cI_n(r)$ and its associated martingale value $\cW_n(r)$ by
\begin{align*}
 \cI_n(r):= \{u\ne\varnothing:\tau_{u^-}\le r<\tau_u\}, \qquad 
 \cW_n(r):= \sum_{u\in\cI_n(r)}e^{-\alpha_n\tau_u},
\end{align*}
where $u^-$ denotes the parent of $u$. For $n_0$ sufficiently large, there exist constants $c,C>0$ such that uniformly for $n\ge n_0$ and $r\ge0$,
\begin{alignat*}{2}
 \E\cW_n(r)&=1, 
 &\E\sum_{u\in\cI_n(r)}e^{-2\alpha_n\tau_u} &\le e^{-\alpha_nr},\\
 \E\abs{\cW_{n,\infty}-\cW_n(r)}^2 &\le Ce^{-\alpha_nr},\qquad 
 &\E\sum_{u\in\cI_n(r)} e^{-\alpha_n\tau_u}\ind\{\tau_u>2r\} 
 &\le C(1+r)e^{-cr}.
\end{alignat*}
\end{lem}

\begin{proof}
Since the process is non-explosive, $\cI_n(r)$ is almost surely finite. Let $\cF_{\cI_n(r)}$ be the stopping-line $\sigma$-algebra. For any generation $m$, the intrinsic martingale $\cW_{n,m}$ satisfies
\begin{align*}
 \cW_{n,m} = \sum_{\substack{u\in\cI_n(r)\\ |u|\le m}} e^{-\alpha_n\tau_u}\cW_{n,m-|u|}^{(u)} + \sum_{\substack{|v|=m\\ \tau_v\le r}} e^{-\alpha_n\tau_v},
\end{align*}
where $\cW_{n,k}^{(u)}$ is the generation-$k$ martingale of the subtree rooted at $u$. Taking expectations, the second sum equals $\nu_n^{*m}([0,r])$. This vanishes as $m\to\infty$, ensuring $L^1$ convergence to zero. By Lemma~\ref{lem:finite-n-cmj-martingale}, $\cW_{n,m}\to\cW_{n,\infty}$ in $L^1$ and the subtree martingales converge almost surely. Since $\cI_n(r)$ is finite, passing to the limit $m\to\infty$ yields
\begin{align*}
 \cW_{n,\infty} = \sum_{u\in\cI_n(r)} e^{-\alpha_n\tau_u}\cW_{n,\infty}^{(u)}.
\end{align*}
Conditional on $\cF_{\cI_n(r)}$, the limits $\cW_{n,\infty}^{(u)}$ are independent mean-one copies of $\cW_{n,\infty}$. Thus,
\begin{align*}
 \E\bigl(\cW_{n,\infty}\mid\cF_{\cI_n(r)}\bigr) &= \cW_n(r),\\
 \E\left[ \abs{\cW_{n,\infty}-\cW_n(r)}^2 \bigm| \cF_{\cI_n(r)} \right] &= \var(\cW_{n,\infty}) \sum_{u\in\cI_n(r)}e^{-2\alpha_n\tau_u}.
\end{align*}
Because $\tau_u>r$ on $\cI_n(r)$,
\begin{align*}
 \sum_{u\in\cI_n(r)}e^{-2\alpha_n\tau_u} \le e^{-\alpha_nr}\cW_n(r).
\end{align*}
Taking expectations and bounding $\var(\cW_{n,\infty})$ uniformly via Lemma~\ref{lem:finite-n-cmj-martingale} establishes the first three estimates.
For the tail bound, the branching property yields
\begin{align*}
 \E\sum_{u\in\cI_n(r)} e^{-\alpha_n\tau_u}\ind\{\tau_u>2r\} = \int_{[0,r]} \cR_n(\dd s)\, \nu_n\bigl((2r-s,\infty)\bigr).
\end{align*}
Since the measures $(\nu_n)$ share a common exponential moment (Lemma~\ref{lem:finite-n-cmj-martingale}), the tail $\nu_n((2r-s,\infty))$ is uniformly bounded by $C e^{-c(2r-s)} \le C e^{-cr}$ for $s\le r$. Combined with $\cR_n([0,r]) \le C(1+r)$ from \eqref{eq:finite-n-local-renewal}, we obtain
\begin{align*}
 \E\sum_{u\in\cI_n(r)} e^{-\alpha_n\tau_u}\ind\{\tau_u>2r\} \le Ce^{-cr}\cR_n([0,r]) \le C(1+r)e^{-cr}.
\end{align*}
This completes the proof.
\end{proof}

\begin{proof}[Proof of Lemma~\ref{lem:uniform-triangular-nerman}]
For a characteristic $\psi$, define 
\[
c_{n,\psi} := \frac1{\theta_n} \int_0^\infty e^{-\alpha_na}\psi(a)\,\dd a
\text{ and }
X_{n,\psi}(t) := e^{-\alpha_nt}Z_n^\psi(t) - c_{n,\psi}\cW_{n,\infty}.
\]

First, let $\phi$ be a compactly supported bounded step function. Set $r:=\fs_0/3$, and recall $\cI_n(r)$ and $\cW_n(r)$ from Lemma~\ref{lem:finite-n-coming-generation}. For all large $n$, $\supp(\phi)\subset[0,r]$. For $t\in[\fs_0,\fs_1]$, any particle contributing to $Z_n^\phi(t)$ is born after time $t-r\ge2r$. The stopping-line decomposition therefore yields
\begin{align}\label{eq:finite-n-stopping-decomp-short}
 X_{n,\phi}(t) = M_{n,\phi}(t) + D_{n,\phi}(t) + c_{n,\phi} \bigl( \cW_n(r)-\cW_{n,\infty} \bigr),
\end{align}
where
\begin{align*}
 M_{n,\phi}(t) &:= \sum_{u\in\cI_n(r)} e^{-\alpha_n\tau_u} \cZ_{n,\phi}^{(u)}(t-\tau_u),
 \qquad D_{n,\phi}(t) := \sum_{u\in\cI_n(r)} e^{-\alpha_n\tau_u} \bigl( h_{n,\phi}(t-\tau_u)-c_{n,\phi} \bigr).
\end{align*}
Since $c_{n,\phi}$ is uniformly bounded, Lemma~\ref{lem:finite-n-coming-generation} ensures the final term in~\eqref{eq:finite-n-stopping-decomp-short} converges to zero in $L^2$.

By Lemma~\ref{lem:finite-n-uniform-renewal}, $\delta_n := \sup_{v\ge r} \abs{h_{n,\phi}(v)-c_{n,\phi}} \to 0$. Splitting the stopping line at $2r$ bounds the drift term:
\begin{align*}
 \sup_{\fs_0\le t\le\fs_1}\abs{D_{n,\phi}(t)} \le \delta_n\cW_n(r) + C_\phi\sum_{u\in\cI_n(r)} e^{-\alpha_n\tau_u}\ind\{\tau_u>2r\}.
\end{align*}
By Lemma~\ref{lem:finite-n-coming-generation}, $\cW_n(r)=O_{\pr}(1)$ and the second sum vanishes in $L^1$. Thus, $\sup_{\fs_0\le t\le\fs_1}\abs{D_{n,\phi}(t)} \stackrel{\pr}{\to} 0$.

Finally, covering $[\fs_0,\fs_1]$ by $O(1+\fs_1)$ unit intervals and applying~\eqref{eq:finite-n-unit-window} conditionally on the stopping line gives
\begin{align*}
 \E\sup_{\fs_0\le t\le\fs_1} \abs{M_{n,\phi}(t)}^2 \le C_\phi(1+\fs_1) \E\sum_{u\in\cI_n(r)}e^{-2\alpha_n\tau_u} \le C_\phi(1+\fs_1)e^{-\alpha_nr} = o(1),
\end{align*}
where the second inequality utilizes Lemma~\ref{lem:finite-n-coming-generation}. Combining these estimates completes the proof for compactly supported step functions.

For the tail approximation, fix $\zeta\in[0,1)$ and put $\phi_\zeta(a):=e^{\zeta a}$. Using the stopping-line decomposition alongside Lemma~\ref{lem:finite-n-characteristic-moments} and \eqref{eq:finite-n-local-renewal}, the three components map directly to uniform bounds:
\begin{align*}
 \E\sup_{\fs_0\le t\le\fs_1} e^{-(\alpha_n-\zeta)t} \sum_{v:\tau_v\le r}e^{-\zeta\tau_v} &\le C_\zeta e^{-c_\zeta r}, \\
 \sup_{\fs_0\le t\le\fs_1} \sum_{u\in\cI_n(r)} e^{-\alpha_n\tau_u} \abs{h_{n,\phi_\zeta}(t-\tau_u)} &\le C_\zeta\cW_n(r) = O_{\pr}(1), \\
 \E\sup_{\fs_0\le t\le\fs_1} \abs{ \sum_{u\in\cI_n(r)} e^{-\alpha_n\tau_u} \cZ_{n,\phi_\zeta}^{(u)}(t-\tau_u) }^2 &\le C_\zeta(1+\fs_1)e^{-\alpha_nr} = o(1).
\end{align*}
Consequently, the weighted envelope is uniformly controlled: $\sup_{\fs_0\le t\le\fs_1} e^{-\alpha_nt}Z_n^{\phi_\zeta}(t) = O_{\pr}(1)$.

Now assume $\phi$ is continuous and $|\phi(a)|\le Ce^{\eta a}$ with $\eta<1$. Choose $\zeta\in(\max\{\eta,0\},1)$. There exist compactly supported step functions $\phi_m$ such that $\norm{\phi-\phi_m}_\zeta := \sup_{a\ge0}e^{-\zeta a}|\phi(a)-\phi_m(a)| \to 0$. The difference $|X_{n,\phi}(t)-X_{n,\phi_m}(t)|$ is bounded by $\norm{\phi-\phi_m}_\zeta$ multiplied by an $O_{\pr}(1)$ bracketed term derived from the envelope bound. Taking $n\to\infty$ and then $m\to\infty$ extends the compact-step result to the continuous case.

Finally, suppose $\phi$ is bounded and locally Riemann-integrable. Set $\phi^{(L)}:=\phi\ind_{[0,L]}$. The truncation error is dominated by $\norm{\phi}_\infty e^{-\alpha_nL}$ multiplied by the supremum of the $\phi\equiv1$ process shifted by $L$, which the continuous case confirms is $O_{\pr}(1)$. Thus, for any $\delta>0$, $\lim_{L\to\infty}\limsup_{n\to\infty} \pr( \sup_{\fs_0\le t\le\fs_1} | X_{n,\phi}(t)-X_{n,\phi^{(L)}}(t) | >\delta ) = 0$. For a fixed $L$, construct bounded step functions $\phi_m^-\le\phi^{(L)}\le\phi_m^+$ supported on $[0,L]$ such that $\int_0^L (\phi_m^+(a)-\phi_m^-(a))\,\dd a \to 0$. Monotonicity sandwiches the approximation; letting $n\to\infty$, $m\to\infty$, and finally $L\to\infty$ completes the proof.
\end{proof}

%%%%%%%%%%%%%%%%%%%%%%%%%%%%%%%%%%%%%%%%%%%%%%%%%%%%%%%%%%%%%%%%%%%%%%%%%
\subsection{Proof of the Fourier-decay Lemma~\ref{lem:fourier-many-to-two}}
Fix $k\in\dZ\setminus\{0\}$ and let
\begin{align*}
 \psi_{k,n}
 :=\frac1{|\cD_n|}
 \sum_{z\in\cD_n}e^{2\pi\mathrm{i}kz/n}.
\end{align*}
Since $\ell/n\to\gl\in(0,1/2)$, the limit $\psi_{k,n}\to\frac{\sin(2\pi k\gl)}{2\pi k\gl}$ ensures $|\psi_{k,n}|$ is strictly bounded away from $1$. Thus, we can choose $\rho_k\in(1/2,1)$ close enough to $1$ such that $\sup_{n\ge n_0}|\psi_{k,n}|\fm_n(\rho_k)<1$ and $2\rho_k>\alpha_n$ for all sufficiently large $n$.

For a bounded step function $g$ with finitely many jumps, taking expectations in the root decomposition yields
\begin{align*}
 \E Y_{n,k,g}^{\op}(t)
 =
 g(t)
 +
 \psi_{k,n}
 \int_{[0,t]}
 \E Y_{n,k,g}^{\op}(t-s)\bar\mu_n(\dd s).
\end{align*}
This implies $|\E Y_{n,k,g}^{\op}(t)| \le C_{k,g}e^{\rho_kt}$. Defining the normalized mean $h_{n,k,g}(t) := e^{-\rho_kt}\E Y_{n,k,g}^{\op}(t)$, the renewal equation guarantees that $h_{n,k,g}$ has uniformly bounded supremum and total variation, since the normalized renewal kernel's total variation is strictly bounded by $1$.

Let
\begin{align*}
 S_{n,k,g}(t)
 :=\E\abs{Y_{n,k,g}^{\op}(t)}^2,
 \qquad
 a_{n,k,g}(t)
 :=
 \frac1{|\cD_n|}
 \int_{[0,t]}
 \E Y_{n,k,g}^{\op}(t-s)\bar\mu_n(\dd s).
\end{align*}
Expanding the variance over independent child branches yields
\begin{align*}
 S_{n,k,g}(t)
 &={}
 |g(t)|^2
 +
 \int_{[0,t]}S_{n,k,g}(t-s)\bar\mu_n(\dd s)
 \\
 &\quad+
 2\Re\big(
 \overline{g(t)}
 |\cD_n|\psi_{k,n}a_{n,k,g}(t)
 \big)
 +
 (|\cD_n|^2|\psi_{k,n}|^2-|\cD_n|)|a_{n,k,g}(t)|^2.
\end{align*}
Applying the first-moment bound $|a_{n,k,g}(t)| \le \frac{C_{k,g}}{|\cD_n|} e^{\rho_kt}\fm_n(\rho_k)$ simplifies this to
\begin{align*}
 S_{n,k,g}(t)
 \le
 C_{k,g}e^{2\rho_kt}
 +
 \int_{[0,t]}S_{n,k,g}(t-s)\bar\mu_n(\dd s).
\end{align*}
Because $\sup_{n\ge n_0}\fm_n(2\rho_k)<1$, iterating this bound over generation-truncated processes and passing to the limit yields $\E|Y_{n,k,g}^{\op}(t)|^2 \le C_{k,g}e^{2\rho_kt}$.

For the unit-window maximum, we apply the Rademacher maximal inequality from the proof of Lemma~\ref{lem:finite-n-characteristic-moments} to the centered process $Y_{n,k,g}^{\op}-\E Y_{n,k,g}^{\op}$, utilizing the normalization $\rho_k$ and child coefficients $e^{2\pi\mathrm{i}kz/n}$ (which have modulus one). The bounded-variation estimate for $h_{n,k,g}$ and the strict inequality $\sup_{n\ge n_0}\fm_n(2\rho_k)<1$ guarantee that
\begin{align*}
 \E\sup_{j\le t\le j+1}
 \abs{Y_{n,k,g}^{\op}(t)}^2
 \le
 C_{k,g}e^{2\rho_k(j+1)},
 \qquad j\ge0.
\end{align*}
Taking the union bound over unit intervals up to $T\ge1$ gives $\E\sup_{0\le t\le T} e^{-2\rho_kt} |Y_{n,k,g}^{\op}(t)|^2 \le C_{k,g}(1+T)$. Since $\fs_0=\eps\log n$ and $\fs_1=\frac12\log n+y$, Markov's inequality ensures
\begin{align*}
 \pr\Big(
 \sup_{\fs_0\le s\le\fs_1}
 e^{-s}\abs{Y_{n,k,g}^{\op}(s)}>\delta
 \Big)
 \le
 C\delta^{-2}(1+\fs_1)
 e^{-2(1-\rho_k)\fs_0}
 \to0,
\end{align*}
establishing the result for step functions.

For bounded, locally Riemann-integrable $g$, we truncate at a large age and approximate via bounded step functions $g_m$. The difference is point-wise dominated:
\begin{align*}
 \abs{Y_{n,k,g}^{\op}(s)-Y_{n,k,g_m}^{\op}(s)}
 \le
 Z_n^{|g-g_m|}(s).
\end{align*}
Lemma~\ref{lem:uniform-triangular-nerman} provides uniform control of this upper bound on the collision window. Since $\sup_{s\le\fs_1}|e^{(\alpha_n-1)s}-1|\to0$, letting $n\to\infty$ and subsequently eliminating the approximation error completes the proof.\qed

%%%%%%%%%%%%%%%%%%%%%%%%%%%%%%%%%%%%%%%%%%%%%%%%%%%%%%%%%%%%%%%%%%%%%%%%%
\subsection{Proof of the ghost-suppression Lemma~\ref{lem:ghost-compensator}}
Define the combined proposal population $Z_{\tot,n}^{\op}(t) := Z_{1,n}^{\op}(t)+Z_{2,n}^{\op}(t)$. A ghost is \emph{primary} if its parent is retained but its proposed label is already occupied; all ghosts descend from a unique primary ghost.

Setting $\eta_+:=\max\{\eta,0\}$, the uniform mean bounds from Lemma~\ref{lem:finite-n-characteristic-moments} yield
\begin{align*}
 \E\sum_{u:\tau_u\le t}
 e^{\eta_+(t-\tau_u)}
 \le Ce^{\alpha_nt},
 \qquad t\ge0.
\end{align*}
By the branching property, a fresh proposal subtree born at time $r$ contributes at most $C_g e^{\alpha_n(\fs_1-r)}$ in expectation to the supremum of the absolute $g$-characteristic on $[r,\fs_1]$. Summing over all primary ghosts gives
\begin{align*}
 \E G_{n,g}(\fs_1)
 \le
 C_g\E\sum_{v \text{ primary}}
 e^{\alpha_n(\fs_1-\tau_v)}.
\end{align*}

We bound primary ghost births via a one-child-deleted construction. Fix a parent and direction, remove this potential child, and condition on the remaining proposal variables. The target label's first occupation time is then independent of the removed child's birth age. Thus, a birth at age $s$ is primary only if its target was occupied strictly before $\tau_u+s$. Because there are at most $Z_{\tot,n}^{\op}(r-)$ directions from any parent to occupied labels at time $r-$, summing over directions and parents (and relaxing the retention condition) yields the compensator bound
\begin{align*}
 \E\sum_{v \text{ primary}}F(\tau_v)
 \le
 \frac1{|\cD_n|}
 \E\int_0^\infty
 F(r)Z_{\tot,n}^{\op}(r-)
 \,\dd Z_{\tot,n}^{\op}(r)
\end{align*}
for any non-negative measurable function $F$.

Applying this to $F(r)=e^{\alpha_n(\fs_1-r)}\ind\{r\le\fs_1\}$ yields
\begin{align*}
 \E G_{n,g}(\fs_1)
 \le
 \frac{C_g}{|\cD_n|}
 \E\int_{(0,\fs_1]}
 e^{\alpha_n(\fs_1-r)}
 Z_{\tot,n}^{\op}(r-)
 \,\dd Z_{\tot,n}^{\op}(r).
\end{align*}
Lemma~\ref{lem:finite-n-characteristic-moments} (with $\phi\equiv1$) provides the second-moment bound $\E[Z_{\tot,n}^{\op}(t)^2] \le Ce^{2\alpha_nt}$. Since the population increases by unit jumps, $2Z_{\tot,n}^{\op}(r-)\,\dd Z_{\tot,n}^{\op}(r) \le \dd\bigl(Z_{\tot,n}^{\op}(r)^2\bigr)$. Integration by parts against $e^{\alpha_n(\fs_1-r)}$ thus bounds the integral by
\begin{align*}
 \E\int_{(0,\fs_1]}
 e^{\alpha_n(\fs_1-r)}Z_{\tot,n}^{\op}(r-)
 \,\dd Z_{\tot,n}^{\op}(r)
 &\le
 C\E\bigl[Z_{\tot,n}^{\op}(\fs_1)^2\bigr]
 +
 C\alpha_n\int_0^{\fs_1}
 e^{\alpha_n(\fs_1-r)}
 \E\bigl[Z_{\tot,n}^{\op}(r)^2\bigr]\,\dd r
 \\
 &\le
 C_y e^{2\alpha_n\fs_1}.
\end{align*}
This estimate strictly relies on integrated birth measures, preserving its validity for $0<\theta<1$. Consequently,
\begin{align*}
 \E G_{n,g}(\fs_1)
 \le
 \frac{C_{y,g}}{|\cD_n|}
 e^{2\alpha_n\fs_1}
 =O(1),
\end{align*}
utilizing $|\cD_n|\asymp n$, $\fs_1=\frac12\log n+y$, and $\alpha_n=1+O(n^{-1})$. Finally, since $e^{\fs_0}=n^\eps\to\infty$, Markov's inequality guarantees
\begin{align*}
 e^{-\fs_0}G_{n,g}(\fs_1)
 \stackrel{\pr}{\to}0,
\end{align*}
which verifies the ghost suppression.\qed
%%%%%%%%%%%%%%%%%%%%%%%%%%%%%%%%%%%%%%%%%%%%%%%%%%%

\bibliographystyle{alphaurl}
\bibliography{fppso} 
\end{document}

%% file: SL_style.tex
\usepackage{systeme,mathrsfs,xfrac} 
\usepackage[colorlinks=true,linkcolor=blue,citecolor=blue,urlcolor=blue,pdfencoding=auto, psdextra]{hyperref} 
\usepackage{amsmath,amssymb,amsthm,amsfonts,amsbsy,latexsym,dsfont,color} 
\usepackage[foot]{amsaddr}
\usepackage{comment}
\usepackage{tikz}
\usetikzlibrary{positioning,arrows.meta,calc,bending}

\usepackage[textsize=tiny]{todonotes}
\usepackage{regexpatch}
\makeatletter
\xpatchcmd{\@todo}{\setkeys{todonotes}{#1}}{\setkeys{todonotes}{inline,#1}}{}{}
\makeatother

\usepackage{enumerate}
\newenvironment{enumeratei}{\begin{enumerate}[\upshape i.]}{\end{enumerate}}
\newenvironment{enumeratea}{\begin{enumerate}[\upshape a)]}{\end{enumerate}}
\newenvironment{enumeraten}{\begin{enumerate}[\upshape 1.]}{\end{enumerate}}

\newtheorem{thm}{Theorem}[section]
\newtheorem{lem}[thm]{Lemma}
\newtheorem{cor}[thm]{Corollary}
\newtheorem{prop}[thm]{Proposition}
\newtheorem{defn}[thm]{Definition}
\newtheorem{rem}[thm]{Remark}

\newtheorem{ass}[thm]{Assumption}

\renewcommand{\le}{\leqslant} 
\renewcommand{\ge}{\geqslant} 
\renewcommand{\leq}{\leqslant} 
 
\newcommand{\eset}{\varnothing}
\newcommand{\ra}{\rangle}
\newcommand{\la}{\langle}

\newcommand{\ind}{\mathds{1}}
\newcommand{\eps}{\varepsilon}

\newcommand{\norm}[1]{\left\Vert#1\right\Vert}
\newcommand{\abs}[1]{\left\vert#1\right\vert}

\newcommand{\ie}{\emph{i.e.,}}

\newcommand{\equald}{\stackrel{\mathrm{d}}{=}}

\def\qed{ \hfill $\blacksquare$}  
\let\ga=\alpha   \let\gd=\delta 
\let\gf=\varphi    \let\gk=\kappa \let\gl=\lambda        \let\go=\omega  \let\gr=\rho \let\gs=\sigma \let\gt=\tau \let\gth=\vartheta
  \let\gz=\zeta
 \let\gD=\Delta  \let\gL=\Lambda

\newcommand{\cA}{\mathcal{A}}\newcommand{\cB}{\mathcal{B}}\newcommand{\cC}{\mathcal{C}}
\newcommand{\cD}{\mathcal{D}}\newcommand{\cE}{\mathcal{E}}\newcommand{\cF}{\mathcal{F}}
\newcommand{\cG}{\mathcal{G}}\newcommand{\cH}{\mathcal{H}}\newcommand{\cI}{\mathcal{I}}
\newcommand{\cL}{\mathcal{L}}
\newcommand{\cN}{\mathcal{N}}
\newcommand{\cR}{\mathcal{R}}
\newcommand{\cS}{\mathcal{S}}\newcommand{\cT}{\mathcal{T}}
\newcommand{\cV}{\mathcal{V}}\newcommand{\cW}{\mathcal{W}}
\newcommand{\cZ}{\mathcal{Z}}  
\newcommand{\fB}{\mathfrak{B}}\newcommand{\fC}{\mathfrak{C}}

\newcommand{\fa}{\mathfrak{a}}

\newcommand{\fm}{\mathfrak{m}}

\newcommand{\fs}{\mathfrak{s}}

\newcommand{\dN}{\mathds{N}}

\newcommand{\dR}{\mathds{R}}
\newcommand{\dS}{\mathds{S}}\newcommand{\dT}{\mathds{T}}

\newcommand{\dZ}{\mathds{Z}} 
\newcommand{\sT}{\mathscr{T}}

\DeclareMathOperator{\E}{\mathds{E}}
\DeclareMathOperator{\pr}{\mathds{P}}

\DeclareMathOperator{\var}{Var}

\DeclareMathOperator{\N}{N}
\newcommand{\oT}{\overline{T}}

\newcommand{\sfP}{\mathsf{P}}
\newcommand{\sfT}{\mathsf{T}}

\newcommand{\sfN}{\mathsf{N}}

\newcommand{\supp}{\operatorname{supp}}

\newcommand{\oS}{\overline{S}}

\newcommand{\vstar}{\mathsf{v}^\star}
\newcommand{\act}{\mathrm{act}}
\newcommand{\col}{\mathrm{col}}

\newcommand{\dd}{\mathrm d}
\newcommand{\Exp}{\mathrm{Exp}}
\newcommand{\Unif}{\mathrm{Unif}}

\newcommand{\op}{\mathsf{op}}

\newcommand{\pihat}{\widehat\pi}
\newcommand{\gbc}{\beta_{\mathsf{c}}}
\newcommand{\sfS}{\mathsf{S}}
\DeclareMathOperator{\Gen}{\operatorname{Gen}}
\providecommand{\dTV}{d_{\mathrm{TV}}}
\newcommand{\sfM}{\mathsf{M}}
\newcommand{\unif}{\mathrm{unif}}
\newcommand{\sfu}{\mathsf{u}}
\newcommand{\sfp}{\mathsf{p}}
\newcommand{\sft}{\mathsf{t}}
\newcommand{\Bad}{\mathrm{Bad}}
\newcommand{\tot}{\mathrm{tot}}
\newcommand{\CTV}{C_{\mathrm{TV}}}